\documentclass[11pt,a4paper]{amsart}

\usepackage{paralist}				
\usepackage{amsfonts} 				
\usepackage{amsmath} 				
\usepackage{amsopn}				
\usepackage{amsthm} 				
\usepackage{yfonts}
\usepackage[arrow, matrix, curve]{xy} 		
\usepackage{extarrows}				
\usepackage{mathtools}				
	
\usepackage{leftidx}
	
\usepackage{listofsymbols}

\usepackage{hyperref}

\usepackage[mathscr]{euscript}
\let\euscr\mathscr \let\mathscr\relax
\usepackage[scr]{rsfso}

	\DeclareMathOperator{\id}{id}

	\DeclareMathOperator{\Alt}{\mathcal{A}lt}

	\DeclareMathOperator{\pr}{pr}

	\newcommand{\bl}{{\scriptscriptstyle \bullet}}
	
	\newcommand{\A}{\mathcal{A}}
	\newcommand{\B}{\mathcal{B}}
	
	\newcommand{\E}{\mathcal{E}}
	\newcommand{\F}{\mathcal{F}}

	\newcommand{\I}{\mathcal{I}}

	\renewcommand{\L}{{\Lambda}}
	\newcommand{\Lc}{\mathcal{L}}
	\newcommand{\M}{\mathcal{M}}
	\newcommand{\N}{\mathbb{N}}
	\newcommand{\Nc}{\mathcal{N}}
	\renewcommand{\O}{\mathcal{O}}
	\renewcommand{\P}{\mathcal{P}}
	\newcommand{\Q}{\mathbb{Q}}
	\newcommand{\R}{\mathbb{R}}
	\newcommand{\T}{\mathcal{T}}
	\newcommand{\U}{\mathcal{U}}
	\newcommand{\V}{\mathcal{V}}
	\newcommand{\W}{\mathcal{W}}
	
	\newcommand{\Z}{\mathbb{Z}}
	
	\newcommand{\ep}{\varepsilon}
	\renewcommand{\a}{\alpha}
	\renewcommand{\b}{\beta}
	\newcommand{\g}{\gamma}
	\renewcommand{\i}{\iota}

	\newcommand{\ol}[1]{\overline{#1}}
	
	\newcommand{\Aut}{\mathrm{Aut}}
	\newcommand{\Hom}{\mathrm{Hom}}

	\newcommand{\0}{{\overline{0}}}
	\newcommand{\Dom}{\mathbf{Dom}} 
	
	\newcommand{\sgn}{{\mathrm{sgn}}}
	\newcommand{\nil}{{+}}
	
	\newcommand{\SC}{\mathcal{SC}^{\infty}}
	\newcommand{\Cs}{\mathcal{C}^{\infty}}

	\newcommand{\de}{\mathrm{d}}
	\newcommand{\Top}{\mathbf{Top}}
	\newcommand{\SVec}{{\mathbf{SVec}_{lc}}}
	\newcommand{\SDom}{\mathbf{SDom}}

	\newcommand{\SDomk}[1]{\mathbf{SDom}^{(#1)}}
	\newcommand{\Frechet}{Fr\'echet}
	\newcommand{\Gr}{\mathbf{Gr}}
	\newcommand{\Grk}[1]{\mathbf{Gr}^{(#1)}}
	\newcommand{\Set}{\mathbf{Set}}
	\newcommand{\SMan}{\mathbf{SMan}}
	\newcommand{\SMank}[1]{\mathbf{SMan}^{(#1)}}
	\newcommand{\VBun}{\mathbf{VBun}}
	\newcommand{\SVBun}{\mathbf{SVBun}}
	\newcommand{\SVBunk}[1]{\mathbf{SVBun}^{(#1)}}
	\newcommand{\Man}{\mathbf{Man}}
	\newcommand{\MBunk}[1]{\mathbf{MBun}^{(#1)}}
	\newcommand{\MBun}{\mathbf{MBun}}
	\newcommand{\MSpacek}[1]{\mathbf{MSpace}^{(#1)}}
	\newcommand{\MSpaceke}[1]{\mathbf{MSpace}_{\ol{0}}^{(#1)}}

	\newcommand{\oddgen}{\lambda}
	
	\newcommand{\po}{\mathbf{p}}

	\newcommand{\Po}{\euscr{P}} 
	\newcommand{\Poo}[1]{\euscr{P}^{#1}_1}
	\newcommand{\Poe}[1]{\euscr{P}^{#1}_0}
	\newcommand{\Poee}[1]{\euscr{P}^{#1}_{0,+}}
	\newcommand{\Part}{\mathscr{P}} 
	\newcommand{\Parte}{\mathscr{P}_{\overline{0}}}
	
	\newcommand{\Lz}{\Lambda_{\overline{0}}}
	\newcommand{\Lo}{\Lambda_{\overline{1}}}
	
	\newcommand{\Sy}{\mathfrak{S}} 
	
	\newcommand{\Shf}{\mathscr{O}}
	
	\newcommand{\pre}[2]{\prescript{}{#1}{#2}}

	\newcommand{\Cc}{\mathcal{C}}
	\newcommand{\Dc}{\mathcal{D}}
	
	\newcommand{\del}{\partial}
	
	\newcommand{\alt}[1]{\mathfrak{A}^{#1}}

\usepackage{scalerel}
\usepackage{stackengine}
\stackMath
\def\hatgap{2pt}
\def\subdown{-2pt}
\newcommand\reallywidehat[2][]{%
\renewcommand\stackalignment{l}%
\stackon[\hatgap]{#2}{%
\stretchto{%
    \scalerel*[\widthof{$#2$}]{\kern-.6pt\bigwedge\kern-.6pt}%
    {\rule[-\textheight/2]{1ex}{\textheight}}
}{0.5ex}
_{\smash{\belowbaseline[\subdown]{\scriptstyle#1}}}%
}}

\usepackage[refpage]{nomencl}
\makenomenclature

\usepackage{multicol}
\makeatletter
\@ifundefined{chapter}
  {\def\wilh@nomsection{section}}
  {\def\wilh@nomsection{chapter}}

\def\thenomenclature{%
  \begin{multicols}{2}[%
    \csname\wilh@nomsection\endcsname*{\nomname}
    \if@intoc\addcontentsline{toc}{\wilh@nomsection}{\nomname}\fi
    \nompreamble]
  \list{}{%
    \labelwidth\nom@tempdim
    \leftmargin\labelwidth
    \advance\leftmargin\labelsep
    \itemsep\nomitemsep
    \let\makelabel\nomlabel}%
}
\def\endthenomenclature{%
  \endlist
  \end{multicols}
  \nompostamble}
\makeatother

	\theoremstyle{plain}
	\newtheorem{theorem}{Theorem}[section]		
	\newtheorem{proposition}[theorem]{Proposition}
	\newtheorem{lemma}[theorem]{Lemma}
	
	\newtheorem{corollary}[theorem]{Corollary}
	
	\theoremstyle{definition}
	\newtheorem{lemma/definition}[theorem]{Lemma/Definition}
	\newtheorem{example}[theorem]{Example}
	\newtheorem{definition}[theorem]{Definition}
	\newtheorem*{problem}{Problem}
	\newtheorem{remark}[theorem]{Remark}
	
	\setdefaultitem{\textbullet}{-}{}{}		 
\setdefaultenum{(a)}{(i)}{}{}	

\allowdisplaybreaks

\begin{document}
\title{Infinite-Dimensional Supermanifolds via Multilinear Bundles}
\author{Jakob Sch\"utt}
\address{Institut f\"ur Mathematik, Universit\"at Paderborn, 
Warburgerstrasse 100, 33098 Paderborn, Germany}
\email{spoon@math.upb.de}

\subjclass[2010]{Primary: 58A50; Secondary: 58B99}

\keywords{Infinite-dimensional supermanifold}

\thanks{This work consists of results obtained in the PhD thesis
``Infinite-Dimensional Supermanifolds, Lie Supergroups and the Supergroup of 
Superdiffeomorphisms''
of the author \cite{DissSchuett}.}

\begin{abstract}
In this paper, we provide an accessible introduction to the theory 
of locally convex supermanifolds in the categorical approach.
In this setting, a supermanifold is a functor $\M\colon\Gr\to\Man$ from the 
category of Grassmann algebras to the category of locally convex manifolds that 
has certain local models, forming something akin to an atlas.
We give a mostly self-contained, concrete definition of supermanifolds along 
these lines, closing several gaps in the literature on the way.

If $\L_n\in\Gr$ is the Grassmann algebra with $n$ generators, we show that 
$\M_{\L_n}$ has the structure of a so called multilinear bundle over the base 
manifold $\M_\R$. We use this fact to show that the 
projective limit $\varprojlim_n\M_{\L_n}$ exists 
in the category of manifolds. In fact, this gives us a faithful functor
$\varprojlim\colon\SMan\to\Man$
from the category of supermanifolds to the category of manifolds. This functor 
respects products, commutes with the respective tangent functor and retains the 
respective Hausdorff property. In this way, supermanifolds can be seen as a
particular kind of infinite-dimensional fiber bundles.
\end{abstract}

\maketitle

\tableofcontents

\section{Introduction}
The first rigorous definition of infinite-dimensional supermanifolds, and also 
the one we will use in this work, is the categorical approach suggested by 
Molotkov in \cite{Mol2}.\footnote{Throughout this work, we will cite 
the more readily available and slightly updated article \cite{Mol}.}
In this approach supermanifolds are defined to be functors from the category of 
finitely generated Grassmann algebras $\Gr$ to the category of manifolds 
$\Man$ with additional local information contained in an `atlas' consisting of 
certain natural transformations.
Let us briefly relate this to the usual sheaf theoretic approach
by Berezin and Le\v{i}tes \cite{BerLei} in the case of finite-dimensional 
supermanifolds. 
In the latter, the functor of points (i.e., the Yoneda embedding)
has long been known to be a useful tool (see for example \cite{Lei80}). 
Moreover, to fully understand the functor of points, it suffices to consider 
supermanifolds whose base manifold is a single point, the so called 
\textit{superpoints}. The superpoints are parametrized by the Grassmann 
algebras and
for every superpoint $\P$ the set of morphisms $\Hom_{\SMan}(\P,\M)$ to a given 
supermanifold $\M$ can be turned into a smooth manifold. In this way one 
obtains a functor 
$\Gr\to\Man$. 
Shvarts \cite{Shv} and Voronov \cite{Vor} had the idea
to use such functors to \textit{define} finite-dimensional 
supermanifolds and Molotkov extended this definition to
infinite-dimensional 
supermanifolds.\footnote{
Most statements in \cite{Mol} are made for Banach supermanifolds but many 
can be easily transferred to \Frechet or locally convex supermanifolds  
(compare \cite[8.5, p.418]{Mol}).}
We call this the categorical approach.

Because of its close relation to the functor of points, some 
of the intuition from the finite-dimensional situation carries over to the 
infinite-dimensional setting. For 
example, the definition of an internal Hom and the related 
superdiffeomorphisms are obtained quite easily in this way (see \cite[8.2, 
p.415 
and 8.4, p.417]{Mol}).
Using this, Hanisch \cite{Hanisch} was able to endow the inner Hom object 
of two finite-dimensional supermanifolds with a supermanifold structure 
in Molotkov's framework.
Another nice feature of the categorical approach is that the definition of 
finite-dimensional and infinite-dimensional supermanifolds, along with their 
morphisms and 
their tangent bundles, is exactly the same. No special topological 
considerations are necessary. Similarly, as has been shown in \cite{AllLau}, 
it lends itself to easy generalization beyond the real or complex case.
What is more, many constructions and calculations can essentially be done 
pointwise, i.e., for every Grassmann algebra. This means that for 
finite-dimensional supermanifolds one often only has to deal with 
finite-dimensional ordinary manifolds.

Despite these advantages, the categorical approach has rarely been used and 
even where it appears, it is usually only applied half-heartedly. For instance, 
when superspaces of morphisms between supermanifolds are considered, the 
morphisms are usually expressed in the sheaf theoretic language (see for 
example \cite{SaWo}, 
\cite{Hanisch} and \cite{BoKo}).
The reason for this lack of interest appears to be twofold. 
For one, the categorical language of natural transformations, 
Grothendieck topologies, sheaves in categories and so on is rather abstract 
and not part of the usual toolbox employed in the field of analysis.
This is then exacerbated by the fact that Molotkov's foundational article 
\cite{Mol2} (resp. \cite{Mol}) contains almost no proofs.
While some proofs for Molotkov's statements were subsequently offered by
Sachse in 
\cite{SachseDiss} and
\cite{Sachse2},
he often falls back to the sheaf theoretic approach so that the statements 
are not shown in their original generality and one obtains little 
intuition for the categorical approach.

We attempt to remedy both of these problems in this work. 
On the one hand, we give a complete definition of infinite-dimensional 
supermanifolds and their morphisms, proving all statements that we use 
(with the rare exception where the proof in the literature can 
directly be applied to our situation and is relatively straightforward).
On the other hand, we simplify the categorical language as much as possible.
As 
it turns out, one can develop the categorical approach in fairly concrete terms 
closely resembling the definition of ordinary manifolds. In this way, we 
completely avoid dealing with more involved questions like representability.

Remarkably, this concrete point of view leads to a canonical faithful functor 
from the category of 
 supermanifolds to the category of manifolds.
This functor has good properties such as
respecting products (i.e., mapping Lie supergroups to Lie groups), commuting 
with the respective tangent functor and retaining the respective Hausdorff 
property.
It can be turned into an equivalence of categories if one 
considers a specific type of fiber bundles on the right-hand side.
 In other words, \textit{we may consider supermanifolds as ordinary manifolds 
with a particular kind of atlas in 
a canonical, well-behaved way.}
All non-trivial supermanifolds are at best mapped to
\Frechet\ manifolds and one wonders whether techniques of infinite-dimensional 
analysis could prove useful in finite-dimensional superanalysis.

To streamline our work, we only consider supermanifolds over the base field 
$\R$. However, many of our constructions derive from \cite{AllLau} and 
\cite{Bert}, where much more general fields and even rings are considered.  
We have consciously formulated our proofs in such a way that they can 
easily be generalized where possible. The only noteworthy obstacles to such 
generalizations 
are Batchelor's Theorem (which necessitates a partition of unity) and 
combinatorial formulas which do not allow for base rings with positive 
characteristic. For the latter, we indicate ways around the problem.

Many standard constructions are beyond the scope of this paper, but we 
hope to have provided the reader with the tools to rectify this with relative 
ease. While equivalences between certain categories of 
supermanifolds in the sheaf theoretic, the concrete and the categorical 
approach 
have been discussed in some detail in \cite{AllLau}, it is not immediately 
obvious how objects like vector fields can be translated between 
the different point of views.
More work to this effect will be critical to enable one to pick and 
choose effectively which approach is most suitable for the problem at hand.
One final drawback of our work that should not go unmentioned is that 
in trying to be as concrete as possible, we
lose some of the intuition offered by the functor of points approach. Thus, 
a close reading of \cite{Mol} is still advisable.

\subsection{Overview and Main Results}
A Grassmann algebra is a free associative $\R$-algebra 
$\L_n:=\R[\oddgen_1,\ldots,\oddgen_n]$, where the 
generators satisfy the relation $\oddgen_i\oddgen_j=-\oddgen_j\oddgen_i$.
There exists a natural grading $\L_n=\L_{n,\ol{0}}\oplus\L_{n,\ol{1}}$ and
the set of objects $\{\R,\L_1,\L_2,\ldots\}$ together 
with the graded morphisms form the category $\Gr$ of Grassmann algebras.
Generators of Grassmann algebras behave infinitesimally in the sense that 
$\oddgen_i^2=0$ and we will see that for this reason (together with 
functoriality) the structure of supermanifolds has many similarities to the 
structure of higher tangent bundles. This enables us to make heavy use of 
the techniques developed by Bertram in \cite{Bert} for dealing with higher 
tangent bundles, higher tangent groups and higher order diffeomorphism groups.

As mentioned, we want to define supermanifolds as functors from the category 
of Grassmann algebras to the category of manifolds with certain local models.
In analogy to ordinary manifolds, we begin by describing the differential 
calculus on the model space:
\begin{compactenum}[1.]
 \item Instead of a vector space, the model space of a supermanifold is
a functor of the form
\[
 \ol{E}\colon\Gr\to\Top,\quad 
\L\mapsto\ol{E}_\L:=(E_0\otimes\Lz)\oplus(E_1\otimes\Lo),
\]
where $E=E_0\oplus E_1$ is a $\Z_2$-graded Hausdorff locally convex vector 
space and $\ol{E}_\L$ is given the obvious product topology.
Then $\ol{E}_\L$ is a $\Lz$-module and the functor $\ol{E}$ 
has the structure of a so called $\ol{\R}$-module in the category $\Top^\Gr$.
\item Open subsets of the model space correspond to open subfunctors, i.e. 
functors 
\[
 \U\colon\Gr\to\Top
\]
such that $\U_\L\subseteq\ol{E}_\L$ is open for all $\L\in\Gr$ and the 
inclusion is a natural transformation. We call such functors \textit{super 
domains}. One 
can show that superdomains have the form
\[
 \U_\L=\U_\R\times(E_0\otimes\Lz^\nil)\times(E_1\otimes\Lo),
\]
where $\Lz^\nil$ is the nilpotent part of $\Lz$.
\item Smooth functions correspond to supersmooth morphisms, i.e. natural 
transformations
\[
 f\colon\U\to\ol{F}
\]
such that $f_\L$ is smooth for all $\L\in\Gr$ and
\[
 df_\L\colon\U_\L\times\ol{E}_\L\to\ol{F}_\L
\]
is $\Lz$-linear in the second component.
\end{compactenum}
Using the infinitesimal behavior of the generators, one obtains an ``exact 
Taylor expansion'' for 
supersmooth morphisms. This can then be used to identify a supersmooth morphism 
$f\colon\U\to\ol{F}$
with its \textit{skeleton}, i.e., a family $(f_k)_{k\in\N_0}$ of maps
$f_k\colon\U_\R\to\Alt^k(E_1,F_{k\ \mathrm{mod}\ 2})$ that are smooth in an 
appropriate sense. Skeletons are of utmost importance for many proofs and the 
description of spaces of supersmooth morphisms.

A supermanifold is defined to be a functor $\M\colon\Gr\to\Man,\ 
\L\mapsto\M_\L$ such that there exists an atlas of natural transformations 
$\varphi^\a\colon\U^\a\to\M$ from superdomains $\U^\a$ to $\M$ for which any 
change of charts is supersmooth.
If $\ep_{\L_n}\colon\L\to\R$ denotes the natural projection, we show that
$\M_{\ep_{\L_n}}\colon\M_{\L_n}\to\M_\R$ gives $\M_{\L_n}$ the structure of a 
so called 
multilinear bundle of degree $n$ over the base manifold $\M_\R$ (compare 
\cite{Bert}). What is more, 
we show in Theorem \ref{thrmsmanmbun} that the family $(\M_{\L})_{\L\in\Gr}$ 
gives one an inverse system of such bundles and that the limit
$\varprojlim_n\M_{\L_n}$
exists in the category of manifolds. This provides us with the functor
\[
 \varprojlim\colon\SMan\to\Man
\]
from the category of supermanifolds to the category of manifolds mentioned 
above. Multilinear bundles and their limits are discussed in Appendix 
\ref{chapmullin}.

In the sheaf theoretic approach every manifold together with its sheaf of 
functions is clearly a supermanifold.
For us the situation is a bit more complicated since a manifold is not a 
functor $\Gr\to\Man$. However, there exists a natural embedding
\[
 \i\colon\Man\to\SMan
\]
introduced by Molotkov in \cite{MolICTP}. In Proposition \ref{proppeventk}, we 
give a description of $\i(M)$ via higher tangent bundles of the manifold $M$, 
which is particularly useful for understanding Lie supergroups.
Similarly, Molotkov constructed a faithful functor
\[
 \i^1_\infty\colon\VBun\to\SMan
\]
from the category of vector bundles to the category of supermanifolds. He 
showed in \cite{Mol2} that any supermanifold whose base manifold allows a 
partition of unity is (non-canonically) isomorphic to a supermanifold that 
comes from a vector bundle. Since this result, generally known as Batchelor's 
Theorem, is important for us and \cite{Mol2} is rather difficult to find, we 
briefly summarize its proof.

\section{Preliminaries and Notation}\label{chap1}
We set $\N:=\{1,2,\ldots\}$\nomenclature{$\N$, $\N_0$}{} and 
$\N_0:=\{0,1,2,\ldots\}$,
respectively.
Let $k\in\N_0$.
Throughout this work, we will write $\ol{k}:=k\ 
\mathrm{mod}\ 2\in\{0,1\}$\nomenclature{$\ol{k}$}{}.

A \textit{locally convex super vector space $E$} is a locally convex vector 
space $E$ together with a fixed decomposition $E=E_0\oplus E_1$, where $E_0$ 
and $E_1$ are locally convex vector spaces and the direct sum is a topological. 
A continuous linear map $f\colon E\to F$ between locally convex super vector 
spaces is a \textit{morphism} of locally convex super vector spaces if
 $f(E_i)\subseteq F_i$ holds for $i\in\{0,1\}$. We denote by $\SVec$ the 
category of Hausdorff locally convex super vector spaces and their morphisms.

We denote by $\Sy_k$\nomenclature{$\Sy_n$}{} the symmetrical group of 
order $k$ and let $\sgn(\sigma)\in\{1,-1\}$ be the \textit{sign} of a 
permutation $\sigma\in\Sy_k$.
\nomenclature{$\sgn$}{}
If $E_1,\ldots,E_k,\ E$ and $F$ are 
locally convex spaces, we let 
$\Lc^k(E_1,\ldots,E_k;F)$ be the $\R$-vector space of continuous 
$k$-multilinear maps
\[
 f\colon E_1\times\cdots\times E_k\to F.
\]
On $\Lc^k(E;F):=\Lc^k(E,\ldots, E;F)$, $\Sy_k$ acts from the left via
\[
 f\circ\sigma(v):=f(v^\sigma):=f(v_{\sigma(1)},\ldots,v_{\sigma(k)})
\]
for $f\in \Lc^k(E;F)$, $\sigma\in\Sy_k$ and $v=(v_1,\ldots,v_k)\in E^k$.
We denote by $\Alt^k(E;F)\subseteq \Lc^k(E;F)$ the space of 
continuous alternating 
$k$-multilinear maps.\nomenclature{$\Alt$}{}

We let $\Lc^0(E;F)=\Alt^0(E;F):=F$ and define the projection
\[
 \alt{k}\colon \Lc^k(E;F)\to\Alt^k(E;F),\quad f\mapsto 
\sum_{\sigma\in\Sy_k}\frac{\sgn(\sigma)}{k!} f\circ \sigma.
\]
\nomenclature{$\alt{k}$}{}
\subsection{Partitions}
We largely use the notation of \cite{Bert} for partitions.
Let $A$ be a finite set. A \textit{partition}
\index{Partition} 
of $A$ is a 
subset 
$\nu=\{\nu_1,\ldots,\nu_{\ell}\}$ of the power set  
$\Po(A)$\nomenclature{$\Po(A)$}{} of $A$ such 
that the sets $\nu_i$, $1\leq i\leq\ell$, are non-empty, pairwise disjoint and 
their union is 
$A$. In this situation, we call $A$ \textit{the total set 
of}\index{Partition!total set of a} $\nu$ and let 
$\underline{\nu}:=A$\nomenclature{$\underline{\nu}$}{}. We define the \textit{ 
length}\index{Partition!length of} of 
the partition $\nu$ as $\ell(\nu):=|\nu|$. Furthermore, we 
denote by $\Part(A)$\nomenclature{$\Part(A)$}{} the set of all partitions of 
$A$ and by $\Part_{\ell}(A)$ \nomenclature{$\Part_{\ell}(A)$}{}
the set of all partitions of $A$ of length $\ell$. If $|A|$ is even, then we 
define $\Part(A)_{\ol{0}}$\nomenclature{$\Part(A)_{\ol{0}}$}{} as those 
partitions which only 
contain sets of even cardinality and $\Part_\ell(A)_{\ol{0}}$ as the partitions 
from $\Part(A)_{\ol{0}}$ of length $\ell$.

For $k\in\N$, we define $\Po^k:=\Po(\{1,\ldots,k\})$
\nomenclature{$\Po^k$, $\Po^k_+$}{}
and $\Po^k_+:=\Po^k\setminus\{\emptyset\}$.
Occasionally, it will be convenient to consider only subsets of even, resp.\ 
odd, cardinality 
and we define $\Poe{k}:=\{A\in\Po^k\colon|A|\ 
\text{even}\}$\nomenclature{$\Poe{k}$, $\Poee{k}$}{}, 
$\Poo{k}:=\{A\in\Po^k\colon|A|\ \text{odd}\}$
\nomenclature{$\Poo{k}$}{} as well as
$\Poee{k}:=\Poe{k}\setminus\{\emptyset\}$.
As a convention, 
$\{i_1,\ldots,i_r\}\subseteq\{1,\ldots,k\}$ is understood to imply 
$i_1<\ldots<i_r$. With this, the lexicographic order\index{Lexicographic order} 
induces a total order on 
the power set $\Po^k$ and every partition $\nu$ can be viewed as an 
ordered $\ell(\nu)$-tuple, which we will do in the sequel (compare 
\cite[MA.4, p.170]{Bert}).
There is another total order on $\Po^k$ that will be useful for us: 
On $\Poe{k}$ and $\Poo{k}$, we use the order induced by $\Po^k$ but for all 
$B\in\Poe{k}$ and all $C\in\Poo{k}$, we let $B<C$.
We will specify whenever we want to 
use this order which we will call the \textit{graded lexicographic 
order}\index{Lexicographic order!graded}.
For a partition 
$\nu=(\nu_1,\ldots,\nu_\ell)$, we define $e(\nu)$\nomenclature{$e(\nu)$}{}, 
resp.\ $o(\nu)$\nomenclature{$o(\nu)$}{}, as the number of sets in $\nu$ with 
even, resp.\ 
odd, cardinality.
In other words, in the graded lexicographic order, we have
\[
 \underbrace{\nu_1<\ldots<\nu_{e(\nu)}}_{\text{even 
cardinality}}<\underbrace{\nu_{e(\nu)+1}<\ldots<\nu_{e(\nu)+o(\nu)}}_{\text{odd 
cardinality}}.
\]
Let $A$ be a finite set and $\nu,\omega\in \Part(A)$. We call $\nu$ a \textit{ 
refinement}\index{Partition!refinement of a} of $\omega$, or $\omega$ \textit{ 
coarser}\index{Partition!coarser} than $\nu$, and write 
$\omega\preceq\nu$\nomenclature{$\omega\preceq\nu$}{} if for every set 
$L\in\nu$ there exists a set 
$O\in\omega$ such that $L\subseteq O$. For $\omega\preceq\nu$ and 
$O\in\omega$, we define the $\nu$\textit{-induced partition of} 
$O$\index{Partition!induced} by
\[
 O|\nu:=\{L\in\nu|L\subseteq O\}\in \Part(O).
\]
\nomenclature{$O|\nu$}{}
In this situation, $\{\omega_1|\nu,\ldots,\omega_{\ell(\omega)}|\nu\}$ 
is a partition of the finite set $\nu$. One easily checks that this 
defines a one-to-one correspondence between partitions that are coarser than 
$\nu$ and $\Part(\nu)$.
\subsection{The Category of Grassmann Algebras}
For any $k\in\N_0$, we let 
$\Lambda_k:=\R[\oddgen_1,\ldots,\oddgen_k]$\nomenclature{$\Lambda_k$, $\L$}{} 
be the unital
associative 
algebra freely generated by the
generators $\oddgen_i$\nomenclature{$\oddgen_i$, $\oddgen_I$}{} with the 
relation 
$\oddgen_i\oddgen_j=-\oddgen_j\oddgen_i$ for 
all $i,j\in\N$. Note that this implies $\oddgen_i\oddgen_i=0$. 
For $I=\{i_1,\ldots,i_\ell\}\subseteq\N$ with $1\leq i_1<\ldots< 
i_\ell\leq k$, we set $\oddgen_I:=\oddgen_{i_1}\cdots\oddgen_{i_\ell}$.
These so called 
\textit{Grassmann algebras}\index{Grassmann algebra} have a natural 
$\Z_2$-grading 
given by 
$\Lambda_{k,\ol{0}}:=\bigoplus_{I\in \Poe{k}}
\oddgen_I\R$ and $\Lambda_{k,\ol{1}}:=\bigoplus_{I\in\Poo{k}}
\oddgen_I\R$ which, with the product topology, turns them into 
topological $\R$-algebras.
A morphism $\varphi\colon\L\to\L'$ between two Grassmann algebras is a morphism 
of 
unital $\R$-algebras that is even in the sense that
\[
 \varphi(\L_{\ol{i}})\subseteq \L'_{\ol{i}}\quad\text{for}\ i\in\{0,1\}.
\]
 We denote by
$\Gr$\nomenclature{$\Gr$, $\Grk{n}$}{} the category of Grassmann algebras, and 
for every 
$n\in\N_0$, we let $\Grk{n}$
be 
the full subcategory containing 
only the objects $\Lambda_0,\ldots,\Lambda_n$. For the sake of convenience, we 
let $\Grk{\infty}:=\Gr$.

We denote the subalgebra of nilpotent 
elements of $\Lambda$ by $\Lambda^\nil$ and set 
$\Lambda^+_{\ol{1}}:=\Lambda^+_{\ol{1}}$ and 
$\Lambda^\nil_{\0}:=\Lambda^\nil\cap\Lz$.
For every $m\geq n\geq 0$, we fix morphisms $\ep_{m,n}\colon 
\Lambda_m\to\Lambda_n$\nomenclature{$\ep_{m,n}$}{} and 
$\eta_{n,m}\colon\Lambda_n\to\Lambda_m$\nomenclature{$\eta_{n,m}$}{} by 
setting 
\[
\ep_{m,n}(\oddgen_k):=\begin{cases}
                  \oddgen_k&\ \text{if}\ k\leq n\\
                  0&\ \text{otherwise}
                  \end{cases}
\]
and $\eta_{n,m}(\oddgen_k)=\oddgen_k$ for $1\leq k\leq n$. 
In the special case $n=0$, we let 
$\ep_{\Lambda_m}:=\ep_{m,0}\colon\Lambda_m\to\R$ 
\nomenclature{$\ep_{\Lambda}$}{}
and 
$\eta_{\Lambda_m}:=\eta_{0,m}\colon\R\to\Lambda_m$.
\nomenclature{$\eta_{\Lambda}$}{ }
\subsection{Locally Convex Manifolds}
All locally convex vector spaces in this thesis are meant to be Hausdorff 
locally convex $\R$-vector spaces.\index{Locally convex space}
\subsubsection{Differential calculus in locally convex spaces}
 A very general differential calculus for topological modules was developed in 
\cite{BGN}.\footnote{To be precise:
Hausdorff topological modules over Hausdorff topological rings whose 
unit group is dense.}
We follow this approach but restrict ourselves to the case of Hausdorff locally 
convex $\R$-vector spaces. In this situation, the $\mathcal{C}^k$-maps coincide 
with the classical $\mathcal{C}^k$-maps in the sense of Bastiani 
\cite{Bastiani} (also known as Keller's $C^k_c$-maps, see \cite{Keller}).
However, it is useful to keep the more general setting in mind since large 
parts of this work can be easily generalized without substantial changes.
See also \cite[Chapter I, p.14ff.]{Bert} for a concise overview.
\begin{definition}
 Let $E,F$ be locally convex spaces, $U\subseteq E$ be 
open 
and $f\colon U\to F$ continuous.
We define the open set $U^{[1]}:=\{(x,v,t)\colon x\in U,x+tv\in U\}\subseteq 
U\times E\times\R$
and say that $f$ is \textit{$\mathcal{C}^1$} if there exists a 
continuous 
map
\[
 f^{[1]}\colon U^{[1]}\to F
 \nomenclature{$f^{[1]}$}{}
\]
such that
\[
 f(x+tv)-f(x)=t\cdot f^{[1]}(x,v,t)
\]
for $(x,v,t)\in U^{[1]}$.
The \textit{differential of $f$ at $x\in U$} is then defined as
\[
 df(x)\colon E\to F,\quad v\mapsto df(x)(v):=f^{[1]}(x,v,0).
\]
We also use the notation $df(x,v):=df(x)(v)$.
Inductively, we say $f$ is \textit{$\mathcal{C}^{k+1}$} if $f^{[1]}$ 
is $\mathcal{C}^k$ for $k\in\N$. If $f$ is $\mathcal{C}^k$ 
for every $k\in\N$, we call $f$ \textit{smooth} or $\mathcal{C}^\infty$.
\end{definition}
The usual rules for differentials apply and we sum them up and fix our 
notation in the following remark.
\begin{remark}
 In the situation of the definition, the map $f^{[1]}$ is unique and 
 $df(x)(v)$ is linear in $v$. If $f$ is $\mathcal{C}^2$, then for every $v\in 
E$ 
the partial map $\del_v f:=df(\bl,v)$ is $\mathcal{C}^{1}$ and we define
$d^kf(x)(v_1,\ldots,v_k):=\del_{v_1}\cdots\del_{v_k} f(x)$
\nomenclature{$d^k$}{}
if $f$ is 
$\mathcal{C}^k$. 
The map $d^kf(x)\colon E^k\to F$ is continuous, $\R$-$k$-multilinear and 
symmetric. In particular 
Schwarz's theorem holds in this setting.
If $V\subseteq U$ is open and $f$ is $\mathcal{C}^k$, then the restriction 
$f|_V$ is so.
If $g$ and $f$ are $\mathcal{C}^k$ and composable, then $g\circ f$ is 
$\mathcal{C}^k$ and we have the the \textit{chain rule}
\[
 d(g\circ f)(x,v)=dg(f(x),df(x,v)).
\]
If $h\colon U_1\times U_2\to F$ is $\mathcal{C}^1$ 
we define $d_1h(x_1,x_2)(v_1):=dh(x_1,x_2)(v_1,0)$ and
$d_2h(x_1,x_2)(v_2):=dh(x_1,x_2)(0,v_2)$
and we have the \textit{rule of partial differentials}
\[
 dh(x_1,x_2)(v_1,v_2)=d_1h(x_1,x_2)(v_1)+d_2h(x_1,x_2)(v_2).
\]
If $f$ is of the form $(f_1,f_2)$ or $f_1\times f_2$, then $f$ is 
$\mathcal{C}^k$ if and only if $f_1$ and $f_2$ are $\mathcal{C}^k$ and it holds
that $df=df_1\times df_2$.
\end{remark}
The following lemma is well-known. As it is instrumental for the rest of the 
work, we give a quick proof nevertheless. Clearly, the proof works in the 
most general setting as well.
\begin{lemma}\label{lemmullinder}
 Let $n\in\N$ and $E_1,\ldots, E_n$ and $F$ be locally convex 
spaces. Each continuous $\R$-$n$-multilinear map 
$f\colon E_1\times\cdots\times E_n\to F$ is automatically $\mathcal{C}^1$ and 
thus 
smooth by induction. In this case, we have
\[
 df(x)(v)=\sum_{i=1}^n f(x_1,\ldots,x_{i-1},v_i,x_{i+1},\ldots,x_n)
\]
for $x=(x_1,\ldots, x_n), v=(v_1,\ldots v_n)\in E_1\times\cdots\times E_n$.
\end{lemma}
\begin{proof}
 Let $\pre{0}{y_i}:=x_i$ and $\pre{1}{y_i}=v_i$ for $1\leq i \leq n$. With this 
we calculate
\[
 f(x+tv)-f(x)=t\cdot\underbrace{\sum_{j\in\{0,1\}^n,\ell_j\geq 1} 
t^{\ell_j-1}f(\pre{j_1}{y_1},\ldots,\pre{j_n}{y_n})}_{f^{[1]}(x,v,t):=},
\]
where $\ell_j:=j_1+\ldots+j_n$. As $f^{[1]}$ is continuous, the statement 
follows.
\end{proof}
\begin{corollary}\label{corafgb}
 Let $E,F,E'$ and $F'$ be locally convex spaces, $U\subseteq E$, $U'\subseteq 
E'$ be open and $f\colon U\to F$ and $g\colon U'\to F'$ be smooth maps. 
Moreover, let $\a\colon F\to F'$ and $\b\colon E'\to E'$ be continuous linear 
maps such that $\b(U)\subseteq U'$ and
$\a\circ f= g\circ \b|_U$.
Then we have
\[
 \a\circ d^nf= d^n g \circ(\b|_U\times \b^n)
\]
for all $n\in\N_0$.
\end{corollary}
\begin{proof}
 This follows from applying the chain rule and Lemma \ref{lemmullinder} to
 \[
  d^n(\a\circ f)=d^n(g\circ \b|_U).
 \]
\end{proof}
\subsubsection{Products and Inverse Limits}
\begin{lemma}[{\cite[Lemma 1.3, p.24]{GloDL}}]\label{lemrestrclosed}
 Let $E,F$ be locally convex spaces, $U\subseteq E$ open and $f\colon U\to F$. 
If $f(U)\subseteq F'$ for a closed vector subspace $F'\subseteq F$, then $f$ is 
smooth if and only if its co-restriction $f|^{F'}\colon U\to F'$ is smooth.
\end{lemma}
Let $J$ be a set and $(F_j)_{j\in J}$ be a family of locally convex spaces. 
Then the product $\prod_{j\in J}F_j$ equipped with the product topology is a 
Hausdorff locally convex space.
\begin{lemma}[\cite{GloBO}]\label{lemprodsmooth}
 Let $E$ be a locally convex space, $U\subseteq E$ open and $(F_j)_{j\in J}$ be 
a family of locally convex spaces. Let $F:=\prod_{j\in 
J}F_j$ and let $\pr_j\colon F\to F_j$ be the projection onto the $j$-th 
component. A map $f\colon U\to F$ is smooth if and only if $f_j:=\pr_j\circ 
f\colon U\to F_j$ is smooth for every $j\in J$. In this case, we have
\[
 df(x,y)=\big(df_j(x,y)\big)_{j\in J}\quad\text{for all}\ x\in U\ \text{and}\ 
y\in E.
\]
\end{lemma}
Let $J$ be a totally ordered index set. The inverse limit\index{Inverse 
limit!of locally convex spaces} 
\[
\textstyle
\varprojlim_{j\in J} F_j:=\Big\{
(x_j)_{j\in J}\in\prod_{j\in J}F_j\colon q^j_i(x_j)=x_i\ 
\text{for all}\ i\leq j
\Big\}
\]
 of an inverse system 
$((F_j)_{j\in J}, (q^j_i)_{i\leq j})$ of locally convex spaces, with continuous 
maps 
$q^j_i\colon F_j\to F_i$ for all $i\leq j$, is a closed subset of 
$\prod_{j\in J}F_j$ and thus a Hausdorff locally convex space.
A direct consequence of this is the following lemma.
\begin{lemma}[\cite{GloBO}]\label{lemdirectlimsmooth}
 Let $E, F$ be locally convex spaces, $U\subseteq E$ be open and $f\colon U\to 
F$ be a map. Assume that $F=\varprojlim F_j$ for an inverse system 
$((F_i)_{i\in J}, (q^j_i)_{i\leq j})$ of locally convex spaces and continuous 
linear maps $q^j_i\colon F_j\to F_i$, with limit maps $q_i\colon F\to F_i$. 
Then $f$ is smooth if and only if $q_i\circ f\colon U\to F_i$ is smooth for 
each $i\in J$. In this case, we have
\[
 df(x,y)=\big(d(q_i\circ f)((x,y))\big)_{i\in J}\quad\text{for all}\ x\in U\ 
\text{and}\ 
y\in E.
\]
\end{lemma}
\subsubsection{Manifolds}\label{subsectman}
With the above, the definition of manifolds over locally convex spaces is 
analogous to the finite-dimensional case (see also \cite[Section 8, p.253]{BGN} 
or \cite[Section 2, p.20]{Bert}).
We fix a locally convex space $E$ and
let $M$ be a topological space. A set
$\A:=\{\varphi_\a\colon U_\a\to V_\a\colon\a\in A\}$
such that 
$U_\a\subseteq M$ and $V_\a\subseteq E$ are open, $\varphi_\a$ is a 
homeomorphism, $\bigcup_{\a\in A}U_\a=M$ and
\[
 \varphi_{\a\b}:=\varphi_\a\circ\varphi_\b^{-1}|_{\varphi_\b(U_\a\cap 
U_\b)}\colon \varphi_\b(U_\a\cap U_\b)\to \varphi_\a(U_\a\cap U_\b)
\]
and its inverse $\varphi_{\b\a}$ are smooth is called a \textit{(smooth) atlas 
of $M$} and the elements of $\A$ are called \textit{charts of 
$M$}.\footnote{The 
index set $A$ 
is just to simplify our notation, it does not 
belong to 
the data of the atlas $\A$.}
Two atlases of $M$ are equivalent if and only if their union is again an atlas. 
Together with an equivalence class of atlases, $M$ is called a \textit{(smooth) 
manifold modelled on $E$} and $E$ is the \textit{model space of $M$}. We 
usually only mention a representative atlas of the equivalence class. Moreover, 
we will generally assume manifolds to be  Hausdorff.\footnote{This assumption 
is necessary to guarantee the existence of smooth partitions of unity for 
finite-dimensional $\sigma$-compact manifolds.}
A manifold is called \textit{paracompact}, resp.\ \textit{$\sigma$-compact}, if 
it is so as a topological space and \textit{finite-dimensional} if its model 
space is finite-dimensional.
If $M$ and $N$ are manifolds with the atlases $\{\varphi_\a\colon\a\in A\}$ and 
$\{\psi_\a\colon\b\in B\}$, then $\{\varphi_\a\times\psi_\b\colon a\in A,\b\in 
B\}$ is an atlas of $M\times N$.

A continuous map $f\colon M\to N$ between two manifolds is a \textit{morphism 
of (smooth) 
manifolds} if for any charts $\varphi\colon U\to\varphi(U)$ of $M$ and 
$\psi\colon W\to\psi(W)$ of $N$ the map
\[
 \psi\circ f\circ\varphi^{-1}\colon\varphi(U\cap f^{-1}(W))\to\psi(W)
\]
is smooth. This property is independent of the choice of atlases.
If $M$ and $N$ are manifolds, we denote by 
$\Cs(M,N)$\nomenclature{$\Cs(M,N)$}{} 
the set of all smooth maps $M\to N$ and we denote by 
$\Man$\nomenclature{$\Man$}{} the 
category of Hausdorff manifolds and their morphisms.

The definition of 
\textit{vector bundles}\index{Vector bundle} or more general \textit{fiber 
bundles}\index{Fiber bundle}, their 
morphisms and their products is similar to the finite-dimensional case 
and for it, we refer to \cite[p.255]{BGN} or \cite[Section 3, p.22]{Bert}.
The particular charts of a bundle are called \textit{bundle charts}\index{Fiber 
bundle!bundle chart} and they 
are elements of a \textit{bundle atlas}.\index{Fiber bundle!bundle atlas} We 
write $\VBun$ for the category of 
vector bundles.

The definition of the \textit{tangent bundle}\index{Tangent bundle} 
$\pi_M\colon\ TM\to M$ of a 
manifold $M$ via equivalence 
classes of smooth curves works as in the finite-dimensional case. For locally 
convex $\R$-vector spaces, this is equivalent to the more general 
definition in \cite[p.254]{BGN} and \cite[Section 3, p.22]{Bert} (see 
\cite{GloBO}).
For the elements of the tangent bundle, we occasionally write $[t\mapsto 
v_t]\in T_{v_0}M$, where $t\mapsto v_t$ denotes some curve in $M$. 
In this notation one has
\[
 Tf[t\mapsto v_t]:=[t\mapsto f(v_t)]\in T_{f(v_0)}N,
\]
if $f\colon M\to N$ is a 
smooth map between manifolds.
If $F$ is a locally convex space, one has a natural isomorphism $TF\cong 
F\times F$ and if $g\colon M\to F$ is smooth,
we also write $dg\colon TM\to F$ for $\pr_2\circ Tg$ with the projection 
$\pr_2\colon F\times F \to F$ onto the second component.
Like in the finite-dimensional case, the above defines 
a functor $T\colon\Man\to\VBun$ and considering $\VBun$ as a subcategory of 
$\Man$, we define $T^0=\id_{\Man}$ and $T^n:=T\circ 
T^{n-1}\colon\Man\to\Man$\nomenclature{$T^n$}{} 
for $n\in\N$.
Finally, if $\{\varphi_\a\colon\a\in A\}$ is 
an atlas of $M$, then $\{T\varphi_\a\colon \a\in A\}$ is a bundle atlas of $TM$ 
and there is a natural isomorphism of vector bundles $T(M\times N)\cong 
TM\times TN$.
\paragraph{Smooth partitions of unity}
\textit{A smooth partition of 
unity}\index{Partition of unity} of a manifold $M$ is an open 
covering $(U_i)_{i\in I}$ of $M$ together with smooth maps $h_i\colon M\to\R$, 
such that
\begin{enumerate}
 \item For all $x\in M$, we have $h_i(x)\geq 0$.
 \item The support of $h_i$ is contained in $U_i$ for all $i\in I$.
 \item The covering is locally finite.
 \item For each $x\in M$, we have $\sum_{i\in I}h_i(x)=1$.
\end{enumerate}
In this situation, we say that $(h_i)_{i\in I}$ is a partition of 
unity that is \textit{subordinate} to $(U_i)_{i\in I}$. We say that a manifold 
$M$ \textit{admits partitions of unity} if it is paracompact and for every 
locally finite open cover $(U_i)$ of $M$, we find smooth maps $h_i\colon 
M\to\R$ 
that constitute a partition of unity subordinate to $(U_i)$ (see 
\cite[p.34]{Lang}).
Paracompact (and in particular $\sigma$-compact) finite-dimensional manifolds 
always admit partitions of unity (compare \cite[Corollary 3.8, p.38]{Lang}).

\subsection{Categories}
We follow \cite{Schub} in the standard definitions. Let us give a brief 
overview to fix our notations.
Throughout, we fix a universe $\mathscr{U}$ (see \cite[3.2.1, 
p.17]{Schub}) that contains the natural numbers $\N$ as an element. 
\textit{Sets} are then elements of 
$\mathscr{U}$ and \textit{classes} are subsets of $\mathscr{U}$.
A category $\Cc$ consists of a class of objects 
$|\Cc|$ and a set of morphisms $\Hom_\Cc(A,B)$ for any objects $A,B$ such that 
we have
 a composition map
\[
 \Hom_\Cc(B,C)\times\Hom_\Cc(A,B)\to\Hom_\Cc(A,C),\quad (f,g)\mapsto f\circ g
\]
(where $C\in\Cc$)
that
satisfies the usual 
conditions. In particular, we have a unique identity morphism 
$\id_A\in\Hom_\Cc(A,A)$.
For $f\in \Hom_\Cc(A,B)$ we also write $f\colon A\to B$ and we call $f$ 
an isomorphism if there exists $f^{-1}\in \Hom_\Cc(B,A)$ such that
$f^{-1}\circ f=\id_A$ and $f\circ f^{-1}=\id_B$.
As a shorthand, we write $A\in\Cc$ instead of 
$A\in|\Cc|$. A \textit{small category} is a category whose objects form a 
set. 

We denote by $\Set$ the category whose objects are sets and whose morphisms 
are maps between sets. The category $\Top$ has topological spaces as objects 
and continuous maps between them as morphisms.
\footnote{Schubert, \cite{Schub}, uses the notations 
$[A,B]_\Cc$ for $\Hom_\Cc(A,B)$, $1_A$ for $\id_A$ and \textit{Ens} for 
$\Set$.} 
\subsubsection{Functors and Functor Categories}
Let $\Cc$ and $\Dc$ be categories. A \textit{functor} $T\colon\Cc\to\Dc$
assigns to each $A\in\Cc$ an object $T(A)\in\Dc$ and to each morphism 
$f\in\Hom_\Cc(A,B)$ a morphism $T(f)\in \Hom_\Dc(T(A),T(B))$ such that 
$T(\id_A)=id_{T(A)}$ and $T(f\circ g)=T(f)\circ T(g)$ hold for all 
$A,B,C\in\Cc$ and all $f\in\Hom_\Cc(B,C)$, $g\in\Hom_\Cc(A,B)$.
Let $S\colon \Cc\to\Dc$ be another functor. A \textit{natural transformation}
$\a\colon S\to T$ consists of morphisms $\a_A\colon S(A)\to T(A)$ for every 
$A\in\Cc$ such that for every $f\in\Hom_\Cc(A,B)$, we have
$T(f)\circ\a_A=\a_B\circ S(f)$. We always have the natural transformation 
$\id_T\colon T\to T$ defined by $(\id_T)_A=\id_{T(A)}$ and if 
$U\colon\Cc\to\Dc$ is another functor and $\b\colon T\to U$ is a natural 
transformation, then the object-wise composition
$\b_A\circ\a_A$ defines a natural transformation 
$\b\circ\a\colon S\to U$. 

If $\Cc$ is a small category, then the functors $\Cc\to\Dc$ are the objects 
and the natural transformations are the morphisms of a category which we 
denote by $\Dc^\Cc$ (see \cite[Proposition 3.4.3, p.19]{Schub}).

\section{Supermanifolds}\label{chap2}
\subsection{Open Subfunctors}
Open subfunctors of functors from $\Top^\Gr$ will play the same role as open 
subsets of topological spaces in ordinary differential geometry. The 
following definitions of intersections, restrictions, open covers and so on are 
intuitive and even provide one with a Grothendieck topology on $\Top^\Gr$ (see 
\cite[Definition 3.17, p.591f.]{AllLau}). 

Let $k\in\N_0\cup\{\infty\}$ and $\F\in\Top^{\Grk{k}}$. For $\L,\L'\in\Grk{k}$ 
and $\varrho\in\Hom_{\Grk{k}}(\L,\L')$, we set $\F_\L:=\F(\L)$ and 
$\F_\varrho:=\F(\varrho)$.
\begin{definition}
 Let $k\in\N_0\cup\{\infty\}$ and $\mathcal{F},\mathcal{F}'\in \Top^{\Grk{k}}$. 
We call $\mathcal{F}'$ a \textit{subfunctor}\index{Subfunctor} of $\mathcal{F}$ 
if for every $\Lambda\in\Grk{k}$, 
we 
have $\mathcal{F}_{\Lambda}'\subseteq\mathcal{F}_{\Lambda}$ and these 
inclusions 
define a natural transformation $\mathcal{F}'\to\mathcal{F}$. 
In this situation, we write 
$\mathcal{F}'\subseteq\mathcal{F}$\nomenclature{$\U\subseteq\V$}{}.
A subfunctor 
$\mathcal{F}'$ of $\F$ is called \textit{open}\index{Subfunctor!open} if every 
$\mathcal{F}_{\Lambda}'$ is open in $\mathcal{F}_{\Lambda}$.
\end{definition}

\begin{lemma/definition}\label{lemdefopsubfun}
 Let $k\in\N_0\cup\{\infty\}$ and $\F\in\Top^{\Grk{k}}$. For an open subset 
$U\subseteq\F_\R$, we define the 
\textit{restriction}\index{Subfunctor!restricted} 
$\F|_U$\nomenclature{$\U"|_V$}{} by setting
\[
 \F|_U(\Lambda):=({\F}_{\ep_\L})^{-1}(U)\quad 
\text{for}\quad \Lambda\in\Grk{k}
\]
and 
$\F|_U(\varrho):=\F_\varrho|_{\F|_U(\Lambda)}$ 
for morphisms 
$\varrho\colon\Lambda\to\Lambda'$. 
Then $\F|_U$ is an open subfunctor 
of $\F$. 
\end{lemma/definition}
\begin{proof}
  Let $x\in\F|_U(\L)$ and $\L\in\Grk{k}$. Then 
$\F_{\ep_{\L'}}\circ\F_\varrho(x)=\F_{\ep_{\L'}\circ\varrho}(x)=\F_{\ep_\L}
(x)\in U$ holds for all morphisms $\varrho\colon\L\to\L'$ since 
$\ep_{\L'}\circ\varrho=\ep_\L$. Because $\F_{\ep_\L}$ is continuous,
${\F}_{\ep_\L}^{-1}(U)$ is open.
\end{proof}
\begin{lemma/definition}
 Let $k\in\N_0\cup\{\infty\}$, $\F\in\Top^{\Grk{k}}$ and $\F',\F''$ be open 
subfunctors of $\F$. 
 Then $(\F'\cap\F'')_\L:=\F_\L'\cap\F_\L''$ and 
$(\F'\cap\F'')_\varrho:=\F_\varrho|_{(\F'\cap\F'')_\L}$ for $\L,\L'\in\Grk{k}$ 
and $\varrho\in\Hom_{\Grk{k}}(\L,\L')$ defines an open subfunctor 
$\F'\cap\F''\subseteq\F$\nomenclature{$\U\cap\V$}{}.
\end{lemma/definition}
\begin{proof}
 By definition $\F_\L'\cap\F_\L''$ is open in $\F_\L$. If 
$x\in\F_\L'\cap\F_\L''$, then by functoriality 
$\F_\varrho(x)\in\F_{\L'}'\cap\F_{\L'}''$, which shows that $\F'\cap\F''$ is a 
functor and that the inclusion is a natural transformation.
\end{proof}
\begin{definition}
 Let $k\in\N_0\cup\{\infty\}$ and $\mathcal{F},\mathcal{F}'\in\Top^{\Grk{k}}$. 
A 
natural transformation $f\colon\mathcal{F}'\to\F$ is called an \textit{open 
embedding}\index{Natural transformation!open embedding} if 
$f_\L\colon\mathcal{F}_{\L}'\to\mathcal{F}_{\L}$ is an open embedding for every 
$\L\in\Grk{k}$. 
\end{definition}

\begin{lemma/definition}\label{lemdefimrestr}
 Let $k\in\N_0\cup\{\infty\}$, $\F,\F'\in\Top^{\Grk{k}}$ and $f\colon\F'\to\F$ 
be a natural transformation. Let $\L,\L'\in\Grk{k}$ and 
$\varrho\in\Hom_{\Grk{k}}(\L,\L')$ be arbitrary.
\begin{enumerate}
\item Let $\V\subseteq\F$ be an open subfunctor. Setting 
$f^{-1}\V_\L:=f^{-1}_\L(\V_\L)$ and 
$f^{-1}\V_\varrho:=\F_\varrho'|_{f^{-1}\V_\L}$ defines an open subfunctor
$f^{-1}\V\subseteq\F'$\nomenclature{$f^{-1}\U$}{}.
 \item If $f$ is an open embedding, then $f(\F')_\L:=f_\L(\F_\L')$ and 
	$f(\F')_\varrho:=\F_\varrho|_{f(\F')_\L}$ define an open subfunctor 
	$f(\F')\subseteq\F$\nomenclature{$f(\U)$}{}.
 \item Let $\U\subseteq\F'$ be an open subfunctor. Then 
$f|_{\U}(\L):=f_\L|_{\U_\L}$ defines a natural 
transformation $f|_{\U}\colon\U\to\F$\nomenclature{$f\vert_{\U}$}{}.
\end{enumerate}
\end{lemma/definition}
\begin{proof}
 (a) Because $f_\L$ is continuous, $f^{-1}\F_\L$ is open. For $x\in 
  f^{-1}\F_\L$, naturality of $f$ implies 
  $f_{\L'}(\F'_\varrho(x))=\F_\varrho(f_\L(x))$ and therefore
  $f^{-1}\F_\varrho(x)\in f^{-1}\F_{\L'}$.
 
 (b) Because $f$ is an open embedding, $f(\F')_\L$ is open. For $x\in 
  f(\F')_\L$, naturality of $f$ implies 
  $f(\F')_\varrho(x)\in f(\F')_{\L'}$.
 
 (c) This is obvious.
\end{proof}

\begin{definition}\label{defcovering}
We call a set
$\{f^\alpha\colon\mathcal{F}^\alpha\to\mathcal{F}\colon \alpha\in 
A\}$\footnote{As with atlases before, the index set $A$ is just used for the 
sake of an easier notation and not part of the data of a covering.} of open 
embeddings a \textit{covering}\index{Natural transformation!covering} if 
$\bigcup_{\alpha\in A}f^\alpha_\L(\mathcal{F}^\alpha_{\L})=\mathcal{F}_{\L}$ 
holds
for all $\L\in\Grk{k}$.
In this situation, we define for all pairs $\a,\b\in A$  
an open subfunctor $\F^{\a\b}\subseteq\F^\a$\nomenclature{$\U^{\a\b}$}{} by 
$\F^{\a\b}_\L:=(f_\L^\a)^{-1}(f_\L^\a(\F^\a_\L)\cap f_\L^\b(\F^\b_\L))$ and 
$\F^{\a\b}_\varrho:=\F^\a_\varrho|_{\F^{\a\b}_\L}$ as well as
natural transformations 
$f^{\a\b}\colon\F^{\a\b}\to\F^{\b\a}$\nomenclature{$f^{\a\b}$}{}
by
$f^{\a\b}_\L:=(f_\L^\b)^{-1}\circ f_\L^\a|_{\F^{\a\b}_\L}$
for 
all $\L,\L'\in\Grk{k}$  and all morphisms 
$\varrho\colon\L\to\L'$.
\end{definition}
\begin{definition}
For $k\in\N_0\cup\{\infty\}$
 a functor $\F\in\Top^{\Grk{k}}$ is called \textit{Hausdorff}\index{Hausdorff} 
if $\F_\L$ is Hausdorff for every $\L\in\Grk{k}$.
\end{definition}

\subsection{Superdomains}
Superdomains take the role of open subsets of vector spaces in ordinary 
analysis. Together with appropriately defined supersmooth morphisms between 
them, they enable us to define supermanifolds from local data much in the same 
way as for manifolds.
The main result in this section is the description of supersmooth 
morphisms through so called skeletons in Proposition \ref{propskel}.
Since skeletons will be our main tool for concrete calculations, other 
important results are a formula for the composition (see Proposition 
\ref{propskelmult})
and a formula for the inversion (see Lemma \ref{lemskelinv}) in terms of 
skeletons.
We follow \cite{AllLau} in this section, with only small additions to 
accommodate $k$-superdomains (i.e., certain functors $\Grk{k}\to\Top$). With 
the exception of a concrete inversion formula, these results have already been 
stated in \cite{MolICTP}.

\begin{lemma/definition}
 For every $E\in \SVec$ and $k\in\N_0\cup\{\infty\}$, we get a functor 
$\ol{E}^{(k)}\colon{\Grk{k}}\to\Top$
\nomenclature{$\ol{E}$, $\ol{E}^{(k)}$}{} 
by setting
\[
 \ol{E}^{(k)}_\L:=\ol{E}^{(k)}(\Lambda):=(E_0\otimes\Lz)\oplus (E_1\otimes\Lo),
\]
equipped with the natural locally convex topology induced by
$E_0\otimes\Lambda_{k,\ol{0}}=\prod_{I\in\Poe{k}}\lambda_IE_0$, resp.\ 
 $E_1\otimes\Lambda_{k,\ol{1}}=\prod_{I\in\Poo{k}}\lambda_IE_1$.
For a morphism $\varrho\colon\Lambda\to\Lambda'$ of Grassmann algebras, we let
$\ol{E}^{(k)}_\varrho:=\ol{E}^{(k)}(\varrho):=(\id_E\otimes\varrho)|
_{\ol{E}^{(k)}_\L }
$.
Then $\ol{E}^{(k)}_\L$ is a topological $\Lz$-module and $\ol{E}^{(k)}_\varrho$ 
is a morphism of modules (with a change of rings).
In other words, $\ol{E}^{(k)}$ is an $\ol{\R}^{(k)}$-module in the category 
$\Top^{\Grk{k}}$ (see \cite{Mol} for details).
We abbreviate 
$\ol{E}:=\ol{E}^{(\infty)}$, $\ol{E_0}^{(k)}:=\ol{E_0\oplus\{0\}}^{(k)}$ and
$\ol{E_1}^{(k)}:=\ol{\{0\}\oplus E_1}^{(k)}$ and let
$\ol{\R}^{(k)}:=\ol{\R\oplus\{0\}}^{(k)}$,
\nomenclature{$\ol{\R}$, $\ol{\R}^{(k)}$}{}  
i.e., $\ol{\R}^{(k)}_\L=\Lz$.

If 
$\Delta\subseteq\Lambda$ is an $\R$-vector subspace, we set 
$\ol{E}^{(k)}_\Delta:=(E_0\otimes\Delta_{\ol{0}})\oplus 
(E_1\otimes\Delta_{\ol{1}})$, where $\Delta_{\ol{0}}:=\Delta\cap\Lz$ and 
$\Delta_{\ol{1}}:=\Delta\cap \L_{\ol{1}}$.
For $n\leq m\leq k$, we will always consider the natural embedding 
$\ol{E}^{(k)}_{\L_n}\subseteq\ol{E}^{(k)}_{\L_m}$ via 
$\ol{E}^{(k)}_{\eta_{n,m}}$. 
\end{lemma/definition}
\begin{proof}
 It is easy to see that $\ol{E}^{(k)}$ is a functor.
Since every $\Lambda\in\Gr^{(k)}$ is a $\Lz$-algebra, $\ol{E}^{(k)}_\L$ is a 
$\Lz$-module with the obvious multiplication. This multiplication is continuous
because its components are simply finite linear combinations. For the same 
reason, $\ol{E}^{(k)}_\varrho$ is continuous and linear. That we have 
$\ol{\R}^{(k)}_\varrho(x)\cdot 
\ol{E}^{(k)}_\varrho(v)=\ol{E}^{(k)}_\varrho(x\cdot 
v)$
for all $x\in\ol{\R}^{(k)}_\L$ and $v\in\ol{E}^{(k)}_\L$, follows directly from 
the definition of the multiplication.
\end{proof}
In the definition of super manifolds the functors $\ol{E}$ will play the same 
role as vector spaces do for regular manifolds. Accordingly, we need a notion 
of open subfunctors and appropriate ``smooth morphisms'' between open 
subfunctors.
All open subfunctors of $\U\subseteq\ol{E}$ for $E\in\SVec$ are 
uniquely determined by $\U_{\R}$.
\begin{lemma}
 Let $k\in\N\cup\{\infty\}$ and $E\in\SVec$. 
 Recall the restriction from Lemma/Definition \ref{lemdefopsubfun}.
 Every open subfunctor $\U\subseteq \ol{E}^{(k)}$ arises as such a restriction,
i.e., we have 
$\ol{E}^{(k)}|_{\U_\R}=\U$.
\nomenclature{$\U$}{}
\end{lemma}
\begin{proof}
 For $k=\infty$ this is just \cite[Proposition 3.5.8, p. 
61]{SachseDiss}. The same proof holds for 
$k\in\N_0$ if one only considers $\Lambda\in\Grk{k}$ (see also \cite[Section 
3.1, p.388 f.]{Mol}).
\end{proof}
\begin{definition}
 Let $E,F\in\SVec$ and $k\in\N_0\cup\{\infty\}$. We call an open subfunctor 
$\U\subseteq\ol{E}^{(k)}$ a $k$\textit{-superdomain}\index{Superdomain!$k$-}. 
In the case of $k=\infty$ we simply call it a 
\textit{superdomain}\index{Superdomain}.
A natural transformation $f\colon\U\to\V$ of 
$k$-superdomains $\U\subseteq\ol{E}^{(k)}$ and 
$\V\subseteq\ol{F}^{(k)}$ is called \textit{supersmooth}\index{Supersmooth}
if for all $\Lambda\in\Grk{k}$ the map 
$f_\Lambda\colon\U_\L\to\V_\L$ is smooth and the derivative
\[
 df_\Lambda\colon\U_\L\times\ol{E}^{(k)}_\L\to\ol{F}^{(k)}_\L
\]
is $\Lz$-linear in the second component, i.e., for any $x\in\U_\L$,
the map 
\[
 df_\L(x,\bl)\colon\ol{E}^{(k)}_\L\to\ol{F}^{(k)}_\L,\quad v\mapsto 
df_\Lambda(x)(v)
\]
is $\Lz$-linear. We denote by $\SC(\U,\V)$\nomenclature{$\SC(\U,\V)$}{} the 
set 
of all supersmooth morphisms $f\colon \U\to\V$.
\end{definition}
It is obvious from the usual chain rule that the $k$-superdomains together with 
the supersmooth natural transformations form a category, which we denote by 
$\SDom^{(k)}$\nomenclature{$\SDom$, $\SDom^{(k)}$}{}. In the case of 
$k=\infty$, we also 
use the notation $\SDom$.

Note that for $\R$-linear maps, it suffices to check $\L_{\ol{0}}$-linearity on
the generators:
For $E,F\in\SVec$ and an $\R$-linear map 
$L\colon\ol{E}_{\L_n}\to\ol{F}_{\L_n}$ with $L(\oddgen_Ix)=\oddgen_IL(x)$ for 
all $x\in\ol{E}_\L$ and $\oddgen_I\in\L_{n,\ol{0}}$, we have
\[
\displaystyle
 L(t\cdot x)=L\big(\sum_{I\in\Poe{n}}\oddgen_It_I\cdot 
x\big)=\sum_{I\in\Poe{n}}\oddgen_It_I\cdot L(x)=t\cdot L(x),
\]
where $t=\sum_{I\in\Poe{n}}\oddgen_It_I\in\L_{n,\ol{0}}$, $t_I\in\R$.
As it turns out, even natural 
transformations that are merely ``smooth'' already have very convenient 
properties.
\begin{lemma}\label{lemdnfnat}
 Let $E,F\in\SVec$, $k\in\N_0\cup\{\infty\}$, $\U\subseteq\ol{E}^{(k)}$ be an 
open subfunctor and $f\colon\U\to\ol{F}^{(k)}$ be a natural transformation such 
that $f_\L$ is smooth for all $\L\in\Grk{k}$. Then for all $n\in\N_0$, the maps
 $d^nf_\L$ define a natural transformation 
 \[
  d^nf\colon\U\times\ol{E}^{(k)}\times\cdots\times \ol{E}^{(k)}\to \ol{F}^{(k)}.
 \]
\end{lemma}
\begin{proof}
 Let $\L,\L'\in\Grk{k}$ and let $\varrho\colon\L\to\L$ be a morphism. 
 Because we have $\ol{F}^{(k)}_\varrho\circ f_\L=f_{\L'}\circ\U_\varrho$ 
 and $\ol{E}^{(k)}_\varrho|_{\U_\L}=\U_{\varrho}$,
Corollary \ref{corafgb} implies that
 \begin{align*}
  \ol{F}^{(k)}_\varrho\circ d^nf_\L=d^nf_{\L'}\circ(\U_\varrho\times 
\ol{E}^{(k)}_\varrho\times\cdots\times \ol{E}^{(k)}_\varrho).
 \end{align*}
 Thus $d^nf$ is a natural transformation.
 Compare also \cite[Lemma 3.6.5, p.812f.]{NeSa} and \cite[Lemma 2.15, 
p.577]{AllLau}.
\end{proof}
In the situation of the lemma, we write $\de f$\nomenclature{$\de f$}{} for the 
natural transformation defined by $d f_\L$.
\begin{lemma}\label{lemdnmulti}
 Let $E,F\in\SVec$, $k\in\N_0\cup\{\infty\}$, $\U\subseteq\ol{E}^{(k)}$ be an 
open subfunctor and $f\colon\U\to\ol{F}^{(k)}$ be a natural transformation such 
that $f_\L$ is smooth for all $\L\in\Grk{k}$. For $n,m\in\N$, $\L_m\in\Grk{k}$ 
let $x\in\U_\R\subseteq\U_{\L_m}$ and 
$y_i\in \oddgen_{I_i}E_{\ol{|I_i|}}\subseteq
\ol{E}^{(k)}_{\L_m}$, where $I_i\in \Po_+^m$ and $1\leq i\leq n$.
Then, we have 
\[
 d^nf_{\L_m}(x)(y_1,\ldots,y_n)\in 
\oddgen_{I_1}\cdots\oddgen_{I_n}F_{\ol{\ell}}\subseteq\ol{F}^{(k)}_{\L_m},
\]
for $\ell:=|\bigcup_{i=1}^nI_i|$. If the sets $I_i$ are not pairwise disjoint, 
we have that $d^nf_{\L_m}(x)(y_1,\ldots,y_n)=0$.
\end{lemma}
\begin{proof}
 Consider $d^nf_{\L_m}$ as a map into $\prod_{I\in 
\Po^m}\oddgen_IF_{\ol{|I|}}$.
 Let $I:=\bigcup_{i=1}^nI_i$, $p\in I$ 
and define $\varrho\colon\L_m\to\L_m$ by $\varrho(\oddgen_p)=0$ and 
$\varrho(\oddgen_j)=\oddgen_j$ for $j\neq p$. By Lemma \ref{lemdnfnat}, we have
\[
0=d^nf_{\L_m}(\U_\varrho(x))\big(\ol{E}^{(k)}_\varrho(y_1),\ldots,\ol{E}^{
(k)}_\varrho(y_n)\big)=\ol{F}^{(k)}_\varrho(d^nf_{\L_m}(x)(y_1,\ldots, y_n)).
\]
In other words, all components 
that do not contain $\oddgen_p$ are 
zero. Conversely, let $p'\notin I$ and let
$\varrho'\colon\L_m\to\L_m$ be a morphism given by $\varrho'(\oddgen_{p'})=0$ 
and 
$\varrho(\oddgen_j)=\oddgen_j$ for $j\neq p'$. Then, we have 
\[
d^nf_{\L_m}(\U_{\varrho'}(x))\big(\ol{E}^{(k)}_{\varrho'}(y_1),\ldots,\ol{E}^{
(k)}_{\varrho'}(y_n)\big)=d^nf_{\L_m}(x)(y_1,\ldots, y_n),
\]
but all components of $\ol{F}^{(k)}_{\varrho'}(d^nf_{\L_m}(x)(y_1,\ldots, 
y_n))$ 
that contain $\oddgen_{p'}$ vanish. It follows that 
$d^nf_{\L_m}(x)(y_1,\ldots,y_n)\in 
\oddgen_{I_1}\cdots\oddgen_{I_n}F_{\ol{\ell}}$. 
Finally, assume that the sets $I_i$ are not pairwise disjoint, for instance let 
$p''$ occur in $r>1$ sets. For $c\in\R$, we define a morphism 
$\varrho''\colon\L_m\to\L_m$ by $\varrho''(\oddgen_{p''}):=c\oddgen_{p''}$ and 
$\varrho''(\oddgen_j):=\oddgen_j$ for $j\neq p''$. We have
\[
d^nf_{\L_m}(\U_{\varrho''}(x))\big(\ol{E}^{(k)}_{\varrho''}(y_1),\ldots,\ol{E}^{
(k)}_{\varrho''}(y_n)\big)=c^rd^nf_{\L_m}(x)(y_1,\ldots, y_n).
\] 
But we also have $\big(\ol{F}^{(k)}_{\varrho''}(d^nf_{\L_m}(x)(y_1,\ldots, 
y_n))\big)_I=c\big(d^nf_{\L_m}(x)(y_1,\ldots, y_n)\big)_I$, which implies 
$d^nf_{\L_m}(x)(y_1,\ldots, y_n)=0$.
\end{proof}
The next lemma, a variation of \cite[Proposition 2.16, 
p.578]{AllLau}, is one of the rare cases where the proof for superdomains 
does not automatically translate to $k$-superdomains. In a sense, it shows the 
infinitesimal character of the generators $\oddgen_i$.
\begin{lemma}\label{lemdftaylor}
 Let $E,F\in\SVec$, $k\in\N_0\cup\{\infty\}$, $\U\subseteq\ol{E}^{(k)}$ be an 
open subfunctor and $f\colon\U\to\ol{F}^{(k)}$ be a natural transformation such 
that $f_\L$ is smooth for all $\L\in\Grk{k}$. Let $1\leq p\leq k$, 
$x\in\U_\L\setminus \ol{E}^{(k)}_{\oddgen_p\L}$ and 
$y\in\ol{E}^{(k)}_{\oddgen_p\L}$. Then, we have
\[
 f_\L(x+y)=f_\L(x)+df_\L(x)(y).
\]
\end{lemma}
\begin{proof}
 Let $c\in\R$. We define a morphism 
$\varrho_c\colon\Lambda\to\Lambda$ by 
$\varrho_c(\oddgen_p):=c\oddgen_p$ and $\varrho_c(\oddgen_i):=\oddgen_i$ for 
$i\neq p$. Then $\U_{\varrho_c}(x)=x$ and 
$\ol{E}^{(k)}_{\varrho_c}(y)=cy$.
Therefore, we have
\[
 f_\L(\ol{E}^{(k)}_{\varrho_0}(x+y))-
 f_\L(\ol{E}^{(k)}_{\varrho_0}(x))=0=\ol{F}^{(k)}_{\varrho_0}(f_\L(x+y)-f_\L(x))
\]
and thus $f_\L(x+y)-f_\L(x)\in\ol{F}^{(k)}_{\oddgen_p\L}$.
It follows that
\begin{align*}
c\cdot f_\L^{[1]}(x,y,c)&=
f_\L\big(\ol{E}^{(k)}_{\varrho_c}(x+y)\big)-f_\L\big(\ol{E}^{(k)}
_{\varrho_c}(x)\big)\\
&=\ol{F}^{(k)}_{\varrho_c}(f_\L(x+y)-f_\L(x))=c\cdot f_\L^{[1]}(x,y,1).
\end{align*}
Taking the limit $c\to 0$, we see that $f_\L^{[1]}(x,y,0)=f_\L^{[1]}(x,y,1)$ 
or in other words
 $f_\L(x+y)-f_\L(x)=df_\L(x,y)$
(compare \cite[Proposition 2.16, p.578]{AllLau}).
\end{proof}
Accordingly, we get the following variation of \cite[Corollary 2.17, 
p.579]{AllLau}.
\begin{proposition}\label{propfdecomp}
Let $E,F\in\SVec$, $k\in\N_0\cup\{\infty\}$, $\U\subseteq\ol{E}^{(k)}$ be an 
open subfunctor and $f\colon\U\to\ol{F}^{(k)}$ be a natural transformation such 
that $f_\L$ is smooth for all $\L\in\Grk{k}$. 
For $x:=x_0+\sum_{I\in \Po^n_{+}}x_I\in\U_{\L_n}$, where $n\leq k$, 
$x_0\in\U_\R$ and $x_I\in \oddgen_IE_{\ol{|I|}}$, we have
\[
f_{\L_n}(x)=f_{\L_n}(x_0)+\sum_{I\in 
\Po^{n}_{+}}\sum_{\omega\in\Part(I)}d^{\ell(\omega)}f_{\L_n}(x_0)(x_{\omega_1}
, \ldots , x_ { \omega_{ \ell(\omega) } } ).
\]
\end{proposition}
\begin{proof}
We first define a suitable partition of 
$\Po^n_+$. Let $\I_1:=\{\{1\}\}$ and  
$\I_j:=\Po^j_+\setminus\Po_+^{j-1}$ for $1<j\leq n$, i.e., $\I_j$ contains all 
subsets that contain $j$ but no larger index. 
Set $x_{\I_j}:=\sum_{I\in\I_j}x_I$; then we can write 
$x=x_0+\sum_{j=1}^nx_{\I_j}$. We prove the proposition by induction on 
the largest index of an odd generator appearing in $x$.
Lemma \ref{lemdftaylor} gives us the induction basis. Assume  
that the formula holds for $1\leq m<n$, i.e., assume that
\[
\textstyle
f_{\L_n}\big(x_0+\sum_{j=1}^mx_{\I_j}\big)=
\displaystyle
f_{\L_n}(x_0)+
\sum_{I\in 
\Po^m_+}\sum_{\omega\in\Part(I)}d^{\ell(\omega)}f_{\L_n}(x_0)(x_{\omega_1}
, \ldots , x_ { \omega_{ \ell(\omega) } } ).
\]
With this, differentiating in the direction of $x_{\I_{m+1}}$ gives us
\begin{align*}
df_{\L_n}\big(x_0+&\textstyle \sum_{j=1}^mx_{\I_j}\big)\big({x_{\I_{m+1}}}\big)=
df_{\L_n}(x_0)(x_{\I_{m+1}})+\\
\displaystyle
&\sum_{I\in 
\Po^m_+}\sum_{\omega\in\Part(I)}d^{\ell(\omega)+1}f_{\L_n}(x_0)(x_{\omega_1}
, \ldots , x_ { \omega_{ \ell(\omega) } },x_{\I_{m+1}} )\\
=&
\sum_{I\in 
\I_{m+1}}\sum_{\omega\in\Part(I)}d^{\ell(\omega)}f_{\L_n}(x_0)(x_{\omega_1}
, \ldots , x_ { \omega_{ \ell(\omega) } } ).
\end{align*}
It follows from Lemma \ref{lemdftaylor} that the addition of both equations 
results in 
the desired formula for $f_{\L_n}(x_0+\sum_{j=1}^{m+1}x_{\I_j})$
(compare \cite[Section 10.2, p.421]{Mol}).
\end{proof}
The proposition can be rewritten in the following way.
\begin{lemma}\label{lemtaylor}
 Let $E,F\in\SVec$, $k\in\N_0\cup\{\infty\}$, $\U\subseteq\ol{E}^{(k)}$ be an 
open subfunctor and $f\colon\U\to\ol{F}^{(k)}$ be a natural transformation such 
that $f_\L$ is smooth for all $\L\in\Grk{k}$. 
For $\L\in\Grk{k}$ fix $x\in\U_\R$, $n_0\in\ol{E}^{(k)}_{\L^\nil_{\ol{0}}}$ and 
$n_1\in\ol{E}^{(k)}_{\L_{\ol{1}}}$. Then
\begin{align*}
 f_\L(x+n_0+n_1)&=\sum_{m,l=0}^\infty\frac{1}{m!l!}\cdot 
d^{m+l}f_\L(x)(\underbrace{n_0,\ldots,n_0}_{m\ 
\mathrm{times}},\underbrace{n_1,\ldots,n_1}_{l\ \mathrm{times}})\\
 &=\sum_{i=0}^\infty\frac{1}{i!}\cdot 
d^{i}f_\L(x)({n_0+n_1,\ldots,n_0+n_1}).
\end{align*}
\end{lemma}
\begin{proof}
Let $\L=\L_n$. By Lemma \ref{lemdnmulti} the sums are finite and after 
multilinear expansion 
we only need to consider the summands that consist of partitions.
For any partition $\I\in\Part(I)$, $I\in\Po^n_+$  in graded lexicographic 
order containing $m$ even 
and $l$ odd sets,
there appear exactly $m!l!$ copies of the term
\[
d^{m+l}f_\L(x)(n_{0,\I_1},\ldots,n_{0,\I_m},n_{1,\I_{m+1}},\ldots,n_{1,\I_{m+l}}
)
\]
in the first sum because we must consider all permutations of 
$n_{0,\I_1},\ldots, n_{0,\I_m}$, resp.\ of 
$n_{1,\I_{m+1}},\ldots,n_{1,\I_{m+l}}$. The first equality follows then 
from 
Proposition \ref{propfdecomp}
(see also \cite[Proposition 2.21, p.582]{AllLau}).
The second equality holds because multilinear expansion of
$d^{m+l}f_\L(x)({n_0+n_1,\ldots,n_0+n_1})$ leads to 
$\binom{m+l}{l}$ copies of 
$d^{m+l}f_\L(x)(n_0,\ldots,n_0,n_1,\ldots,n_1)$ ($m$ times $n_0$ and $l$ times 
$n_1$)
and
$\binom{m+l}{l}\cdot\frac{1}{(m+l)!}=\frac{1}{m!l!}$.
\end{proof}

\begin{corollary}\label{corssinbas}
 Let $E,F\in\SVec$, $k\in\N_0\cup\{\infty\}$, $\U\subseteq\ol{E}^{(k)}$ be an 
open subfunctor and $f\colon\U\to\ol{F}^{(k)}$ be a natural transformation such 
that $f_\L$ is smooth for all $\L\in\Grk{k}$. If additionally
 $df_{\L}(x_0)\colon\ol{E}^{(k)}_{\L}\to\ol{F}^{(k)}_{\L}$ is 
$\L_{\ol{0}}$-linear for all $x_0\in\U_\R$, then $f$ is 
supersmooth.
\end{corollary}
\begin{proof}
  Let $\L=\L_n$. Because $df_{\L}(x_0)$ is $\L_{\ol{0}}$-linear it follows by 
symmetry of 
the higher derivatives that 
$d^mf_{\L}({x_0})\colon\ol{E}^{(k)}_{\L}\times\cdots\times 
\ol{E}^{(k)}_{\L}\to 
\ol{F}^{(k)}_{\L}$ is $\L_{\ol{0}}$-$m$-multilinear for all $m\in\N$. Let  
$x=x_0+\sum_{I\in\Po^n_+}x_I$ and $y=y_0+\sum_{I\in\Po^n_+}y_I$ where 
$x_I,y_I\in\oddgen_IE_{\ol{|I|}}$ and $t\in\L_{\ol{0}}$.
With 
Proposition \ref{propfdecomp}, we calculate
\begin{align*}
 df_{\L}(x)&(ty)=d\Big(f_{\L}(x_0)+\sum_{I\in 
\Po^{n}_{+}}\sum_{\omega\in\Part(I)}d^{\ell(\omega)}f_{\L}(x_0)(x_{\omega_1}
, \ldots , x_ { \omega_{ \ell(\omega) } } )\Big)(ty)\\
&=df_\L(x_0)(ty)+
\sum_{I\in 
\Po^{n}_{+}}\sum_{\omega\in\Part(I)}d^{\ell(\omega)+1}f_{\L}(x_0)(x_{\omega_1
}
, \ldots , x_ { \omega_{ \ell(\omega) } },ty )\\
&=t\Big(df_\L(x_0)(y)+
\sum_{I\in 
\Po^{n}_{+}}\sum_{\omega\in\Part(I)}d^{\ell(\omega)+1}f_{\L}(x_0)(x_{\omega_1
}
, \ldots , x_ { \omega_{ \ell(\omega) } },y )\Big).
\end{align*}
\end{proof}
This was already stated in \cite[Theorem 
3.3.2, p.391]{Mol} without proof. 
The corollary simplifies some calculations considerably. A small example is the 
next lemma.
\begin{lemma}\label{lemdfss}
 Let $k\in\N_0\cup\{\infty\}$, $E,F\in\SVec$ and $\U\subseteq\ol{E}^{(k)}$ be 
an open 
subfunctor. If $f\colon\U\to\ol{F}^{(k)}$ is supersmooth, then $\de 
f\colon\U\times\ol{E}^{(k)}\to\ol{F}^{(k)}$ is supersmooth as well.
\end{lemma}
\begin{proof}
 By Corollary \ref{corssinbas}, it suffices to calculate
 \begin{align*}
  &d\big(\de f_\L(x_0,y_0)\big)(t\cdot u, t\cdot 
v)=d^2f_\L(x_0)(y_0,t\cdot u)
  +df_\L(x_0)(t\cdot v)\\  
  &=
  t\cdot d^2f_\L(x_0)(y_0, u)
  +t\cdot df_\L(x_0)( v)
  =t\cdot \big(d\big(\de f_\L(x_0,y_0)\big)(v,u)\big),
 \end{align*}
 for $(x_0,y_0)\in\U_\R\times\ol{E}^{(k)}_\R$, 
$(v,u)\in\ol{E}^{(k)}_\L\times\ol{E}^{(k)}_\L$ and $t\in\Lz$.
\end{proof}
By induction, it follows that all higher derivatives of supersmooth maps 
are supersmooth again. A more general but also more involved version was proved 
in \cite[Proposition 2.18, p.580]{AllLau}.

We will now give an explicit description of supersmooth morphisms as so 
called skeletons, which is 
essential for almost all applications. It was already stated in 
\cite[Proposition 3.3.3, p.391]{Mol} and proofs can be found in
\cite[Theorem 4.11, p.20]{Sachse2}
or in higher generality in
\cite[Proposition 3.4, p.584]{AllLau}. 
\begin{definition}\label{defclcalt}
 Let $n\in\N$, let $E_0,\ldots E_n$ and $F$ be locally convex spaces and 
$U\subseteq E_0$ open.
Denote by 
$\Cs(U,\Lc^n(E_1,\ldots, 
E_n;F))$
the set of maps $f\colon U\to\Lc^n(E_1,\ldots,E_n;F)$ such that
\[
f^{\wedge}\colon U\times (E_1\times\cdots\times E_n)\to 
F,\quad  f^{\wedge}(x,v):=f(x)(v)
\]
is smooth.
In this situation, we define 
\[
d^mf(x)(w,v):=\partial_{(w_m,0)}\ldots\partial_{(w_1,0)}f^{\wedge }(x,v),
\]
for $m\in\N$, $x\in U$, $v\in E_1\times\cdots\times E_n$ and
$w=(w_1,\ldots,w_m)\in E_0^m$.
Analogously, we define
$\Cs(U,\Alt^n(E_1;F))$\nomenclature{$\Cs(U,\Alt^n(E_1;F))$}{}
as the set of maps $f\colon U\to\Alt^n(E_1;F)$ that are smooth in the above 
sense.
\end{definition}
\begin{definition}\label{defskel}
Let $k\in\N_0\cup\{\infty\}$, $E,F\in\SVec$ and $U\subseteq E_0$ open.
 A 
\textit{($k$-)skeleton}\index{Skeleton} is a family of maps 
$(f_n)_{0\leq n<k+1}$\nomenclature{$(f_n)_n$}{} such that
$f_n\in \Cs(U,\Alt^n(E_1;F_{\ol{n}}))$.
It will be convenient to set $d^0f_n:=f_n$ and let 
$d^0f_n(x)(w_1,\ldots,w_m,v):=d^0f_n(x)(v)$ as well as 
$d^mf_0(x)(w,v):=d^mf_0(x)(w)$ for $x\in U$, $v\in E_1^n$ and 
$w=(w_1,\ldots,w_m)\in E^m_0$.
\end{definition}

\begin{proposition}[{\cite[Proposition 3.4, p.584]{AllLau}}]\label{propskel}
 Let $E,F\in\SVec$, $k\in\N_0\cup\{\infty\}$, $\U\subseteq\ol{E}^{(k)}$, 
$\V\subseteq\ol{F}^{(k)}$ be 
open subfunctors and $f\in\SC(\U,\V)$. Then the 
equation
\[
 \textstyle
 f_{\L_k}\big(x+\sum_{l=1}^k\oddgen_l y_l\big)=
 f_{0}(x)+
\displaystyle
\sum_{l=1}^k\sum_{\{i_1,\ldots,i_l\}\in\Po^k}
\oddgen_If_l(x)(y_{i_1},\ldots,y_{i_l}),
\]
where $x\in\U_\R$ and $y_l\in E_1$, defines a $k$-skeleton 
$(f_n)_n$\index{Skeleton!of a supersmooth map}. For this skeleton, we have
\begin{align}\label{formulaskel}
 f_{\L_N}(x+n_0+n_1)=\sum_{m,l=0}^\infty\frac{1}{m!l!}\cdot 
d^mf_l(x)(\underbrace{n_0,\ldots,n_0}_{m\ 
\mathrm{times}},\underbrace{n_1,\ldots,n_1}_{l\ \mathrm{times}}),
\end{align}
where $x\in\U_\R$, $n_0\in\ol{E}^{(k)}_{\L^\nil_{N,\ol{0}}}$,  
$n_1\in\ol{E}^{(k)}_{\L_{N,\ol{1}}}$ and $N\leq k$. Here it is understood that
\[
d^mf_l(x)(\oddgen_{I_1}v_1,\ldots,\oddgen_{I_{l+m}}v_{{l+m}})=
\oddgen_{I_1}\cdots\oddgen_{I_{l+m}} 
d^mf_l(x)(v_{1},\ldots,v_{{l+m}})
\]
for $v_1,\ldots,v_m\in E_0$, $v_{m+1},\ldots, v_{m+l}\in E_1$ and $|I_j|$ even 
if $1\leq j\leq m$ and odd if $m+1\leq j\leq m+l$. Conversely, every 
$k$-skeleton defines a supersmooth map via formula (\ref{formulaskel}) and the 
skeleton of this map is the original one. 
\end{proposition}
\begin{proof}
 Using Lemma \ref{lemtaylor} instead of \cite[Proposition 2.21, p.582]{AllLau}, 
the proof follows in the same way as \cite[Proposition 3.4, p.584]{AllLau}. For 
the reader's convenience, we will sketch the steps using our notation. Let 
$k\leq N$
By Proposition \ref{propfdecomp} we have 
\[
 \textstyle
 f_{\L_N}\big(x+\sum_{l=1}^k\oddgen_l y_l\big)=
\displaystyle
\sum_{l=0}^k\sum_{\{i_1,\ldots,i_l\}\in\Po^N}
d^lf_{\L_N}(x)(\oddgen_{i_1}y_{i_1},\ldots,\oddgen_{i_l}y_{i_l}).
\]
The maps on the right-hand side are symmetric in 
$\oddgen_{i_j}y_{i_j}$ but swapping two odd generators leads to a sign change 
by the natural transformation property. With Lemma \ref{lemdnmulti} one sees 
that this determines 
alternating maps in $y_{i_j}$, where it is understood that the odd generators 
can be pulled out in order of their appearance. Now, one applies Proposition 
\ref{propfdecomp} to derive formula 
(\ref{formulaskel}). Note that by supersmoothness, the alternating maps defined 
above determine
$d^{m+l}f_{\L_k}$ completely.

To see that the right-hand side of formula (\ref{formulaskel}) defines a 
natural transformation for a given skeleton is straightforward and 
supersmoothness then follows directly or with Corollary 
\ref{corssinbas}. This 
supermooth map has the original skeleton by a combinatorial 
argument similar to the one used for Lemma \ref{lemtaylor}.
\end{proof}
\begin{remark}\label{remskelpart}
 In the situation of the proposition above, we can use Proposition 
\ref{propfdecomp} instead of Lemma \ref{lemtaylor}
 to get 
 \[
 f_{\L_k}(x+n_0+n_1)=
\sum_{I\in\Po^k_+}\sum_{\omega\in\Part(I)}
\oddgen_{\omega}d^{(e(\omega))}f_{o(\omega)}(x)(n_{\omega}),
\]
where the partitions $\omega$ are in graded lexicographic order, 
$\oddgen_{\omega}=\oddgen_{\omega_{1}}\cdots\oddgen_{\omega_{\ell(\omega)}}$ 
and
$n_\omega:=(n_{0,\omega_1},\ldots,n_{0,\omega_{e(\omega)}},n_{1,\omega_{
e(\omega)+1 } } ,
\ldots,n_{1,\ell(\omega)})$.
\end{remark}
\begin{remark}\label{remdf}
 Let $f\colon\U\to\V$ be as in Proposition \ref{propskel}. We have already 
seen that $\de f\colon\U\times\ol{E}^{(k)}\to\ol{F}^{(k)}$ is supersmooth. For 
$\Lambda\in\Grk{k}$, $x\in\U_\R$, $y\in E_0$ and $x_i,y_i\in 
E_i\otimes\Lambda_i^\nil$ set $u:=x+x_0+x_1$ and $v:=y+y_0+y_1$. Then use the 
proposition to calculate
\begin{align*}
df_\Lambda(u)(v)=\sum_{m,l=0}^\infty\frac{1}{m!l!}\cdot&\Big(d^{m+1}f_l(x
)(y,x_0,\ldots,x_0,x_1,\ldots,x_1)\\
&+m\cdot d^m f_l(x)(y_0,x_0,\ldots,x_0,x_1,\ldots,x_1)\\
&+l\cdot d^m f_l(x)(x_0,\ldots,x_0,y_1,x_1\ldots,x_1)\Big)\\
=\sum_{m,l=0}^\infty\frac{1}{m!l!}\cdot&\Big(d^{m+1}f_l(x
)(y+y_0,x_0,\ldots,x_0,x_1,\ldots,x_1)\Big)+\\
\phantom{=}\sum_{m,l=0}^\infty\frac{1}{m!l!}\cdot&\Big(d^m 
f_{l+1}(x)(x_0,\ldots,x_0,y_1,x_1\ldots,x_1)\Big).
\end{align*}
We see that the skeleton of $\de f$ is given by 
\[
 (\de f)_n=d f_n(\pr_{\U_\R},\pr_{E_0})(\pr_1,\ldots,\pr_1)+ 
n\cdot \alt{n} f_n(\pr_{\U_\R})(\pr_2,\pr_1,\ldots,\pr_1),
\]
with the projections $\pr_{\U_\R}\colon\U_\R\times E_0\to\U_\R$, 
$\pr_{E_0}\colon\U_\R\times E_0\to E_0$, the projection to the 
first component $\pr_1\colon E_1\times E_1\to E_1$ and the projection to the 
second argument 
$\pr_2\colon E_1\times E_1\to E_1$.
\end{remark}
In the sequel, we will not differentiate between supersmooth morphisms and 
their skeletons. In other words, if $\U,\V,\W$ are $k$-superdomains and 
$f\in\SC(\U,\V)$ has the skeleton $(f_n)_n$, we will write 
$(f_n)_n\colon\U\to\V$. If additionally $g\in\SC(\V,\W)$ 
has the skeleton $( g_n)_n$ we let $(g_n)_n\circ(f_n)_n$
\nomenclature{$(f_n)_n\circ( g_n)_n$}{} be the skeleton of $g\circ 
f$.
For this composition the concrete formula is given as follows.
\begin{proposition}[{compare 
\cite[Proposition 3.7, p.586]{AllLau}}]\label{propskelmult}
 Let $k\in\N_0\cup\{\infty\}$, $E\in \SVec$, $\U\subseteq\ol{E}^{(k)}$ be an 
open 
subfunctor and $\V,\W\in\SDom^{(k)}$.
For two supersmooth morphisms $(f_r)_r\colon\U\to\V$, 
$(g_r)_r\colon\V\to\W$ the skeleton
$(h_n)_n:=( g_r)_r\circ(f_r)_r$ is given by
$h_0:= g_0\circ f_0$ for $n=0$ and otherwise by
\index{Skeleton!composition}
 \begin{align}\label{formulaskelmul}
  h_n(x)(v)=\sum_{\substack{ m,l;\sigma\in
\Sy_n,\\(\alpha,\beta)\in
I_{m,l}^n }}\nomenclature{$I_{m,l}^n$}{}
\frac{\mathrm{sgn}(\sigma)}{m!l!\alpha!\beta!}d^m g_l(f_0(x))
  \big((f_\alpha\times f_\beta)(x)(v^{\sigma})\big)
 \end{align}
for $x\in \U_\R$ and $v=(v_1,\ldots,v_n)\in E^n_{1}$, where 
$v^{\sigma}:=(v_{\sigma(1)},\ldots,v_{\sigma(n)})$,
\begin{gather*}
I^n_{m,l}:=\left\{(\alpha,
\beta)\in(2\N)^m\times(2\N_0+1)^l\big|\ |\alpha|+|\beta|=n\right\},\\
 f_\alpha:=f_{\alpha_1}\times\cdots\times f_{\alpha_m},\
f_\beta:=f_{\beta_1}\times\cdots\times f_{\beta_l}\ 
\text{and}\\
 \alpha!=\alpha_1!\cdots\alpha_m!,\quad \beta!=\beta_1!\cdots\beta_l!.
\end{gather*}
\end{proposition}
\begin{proof}
By Proposition \ref{propskel} $(h_n)_n$ is defined by
\[
 g_\Lambda(f_\Lambda(x+y))=\sum_{l=0}^\infty\frac{1}{l!}h_l(x)(y,\ldots,y)
\ \text{for all}\ \L=\L_n\in\Grk{k},
\]
where $x\in\U_\R$, and
$y=\sum_{j=1}^n\oddgen_jy_j\in \ol{E}^{(k)}_\L$. For $i\in\{0,1\}$, we let
\[
n_i:=\sum_{l\in2\N-i}\frac{1}{l!}f_l(x)(y,\ldots,y).
\] 
 Together with Proposition \ref{propskel}, this implies
\begin{align}\label{formulasumalldridge}
 g_\Lambda(f_\Lambda(x+y))=\sum_{m,l=0}^\infty\frac{1}{m!l!}
d^m g_l(f_0(x))(n_0,\ldots,n_0,n_1,\ldots,n_1).
\end{align}
Since in formula (\ref{formulaskelmul}), $h_n$ only depends on 
$(f_r)_{r\leq n}$ and $( g_r)_{r\leq n}$, it suffices to compare  
the component containing all odd generators of $\L=\L_n$, i.e., the component 
$I:=\{1,\ldots,n\}$. The formula follows then by trivial induction.
Multilinear expansion of the $n_i$ in formula (\ref{formulasumalldridge}) 
shows that exactly those summands contribute, where the indices of all 
occurring $f_i$ add up to $n$. In other words exactly those
containing $f_\alpha\times f_\beta$ with
$(\alpha,\beta)\in I_{m,l}^n$.
Applying multilinear expansion to $y$, we see that for every 
$(\alpha,\beta)\in I^n_{m,l}$ exactly all
permutations 
$\oddgen_{
\sigma(1)}\cdots\oddgen_{\sigma(n)}\tfrac{1}{\alpha!\beta!}
(f_\alpha\times f_\beta)(y_ {
\sigma(1)},\ldots,y_{\sigma(n)})$ for $\sigma\in \Sy_n$ appear in formula 
(\ref{formulasumalldridge}) since
equal indices cancel each other. The sign in the formula is explained by
$\oddgen_{\sigma(1)}\cdots\oddgen_{\sigma(n)}=\sgn(\sigma)\oddgen_I$.
\end{proof}
\begin{remark}
  Formula (\ref{formulaskelmul}) was already stated in 
\cite[Proposition 3.3.3, p.91 f.]{Mol} but the first proof in the literature 
was \cite[Proposition 3.7, p.586]{AllLau}. Unfortunately, the proof is 
incomplete and there is a small mistake in the formula (the 
original one in \cite{Mol} is 
correct), which is why we decided to give the proof in its entirety.
 To see that our formula differs from the one proposed 
in
\cite{AllLau}, consider that in the situation of Proposition \ref{propskelmult} 
the latter leads to
\[
 \sum_{\sigma\in 
\Sy_2}\frac{1}{2}d g_1(f_0(x))\big(f_2(x)(\bl^\sigma),
 f_1(x)(\bl)\big)(v_1 , v_2 , v_3)=0,
\]
while in general
\[
 \sum_{\sigma\in
\Sy_3}\frac{\sgn(\sigma)}{2}d g_1(f_0(x))\big(f_2(x)(\bl),
 f_1(x)(\bl)\big)(v_1 , v_2 , v_3)^\sigma\neq 0.
\]
\end{remark}
\begin{lemma}\label{lemskelinv}
 Let $k\in\N\cup\{\infty\}$, $E,F\in\SVec$ and 
$\U\subseteq\ol{E}^{(k)},\V\subseteq\ol{F}^{(k)}$ be open subfunctors. A 
supersmooth morphism 
$f\colon\U\to\V$ is an isomorphism in $\SDomk{k}$ if and only if 
$f_{\L_1}\colon\U_{\L_1}\to\V_{\L_1}$ is a diffeomorphism. In this case, using 
the same notation as in formula (\ref{formulaskelmul}), the inverse $g$ has the 
skeleton
\index{Skeleton!inversion}
\begin{align*}
 &g_0\colon\V_\R\to\U_\R,\quad g_0(x'):=f_0^{-1}(x'),\\
 &g_1:\V_\R\to\Alt^1(F_1;E_1),\quad g_1(x'):= 
f_1(g_0(x'))^{-1}\quad\text{and}\\
 &g_n\colon\V_\R\to\Alt^n(F_1;E_{\ol{1}}),\\
 & g_n(x')(v'):=
 \ \quad-\smashoperator{\sum_{\substack{ m,l<n,(\alpha,\beta)\in
I_{m,l}^n,\\ 
\sigma\in \Sy_n
}}}\qquad
\frac{\mathrm{sgn}(\sigma)}{m!l!\alpha!\beta!}d^m g_l(x')
\big((f_\alpha\times f_\beta)\left( g_0(x')\right)(v^{
\sigma } )\big),
\end{align*}
where $n>1$, $v'=(v_1',\ldots,v_n')\in F_1^n$ and
$v:=( g_1(x')(v_1'),\ldots, g_1(x')(v_n'))\in E_1^n$.
\end{lemma}
\begin{proof}
 If a supersmooth morphism $f\colon\U\to\V$ is invertible, then clearly 
 $f_\L$ is a diffeomorphism for every $\L\in\Grk{k}$. Conversely, let 
$f_{\L_1}$ be a diffeomorphism. 
Then $f_{\L_1}(x+\oddgen_1v)=f_0(x)+\oddgen_1f_1(x)(v)$ for all 
$x\in\U_\R$ and $v\in E_1$. 
A direct calculation shows that 
$g_{\L_1}(x'+\oddgen_1v'):=g_0(x')+\oddgen_1g_1(x')(v')$
is the inverse of $f_{\L_1}$.
With the supersmooth 
morphism $(g_n)_n\colon\V\to\U$, we calculate
\begin{align*}
 (( g_r)_r\circ(f_r)_r)_n&(x)(v)=\\
 &\sum_{\substack{ m,l<n,(\alpha,\beta)\in
I_{m,l}^n,\\ 
\sigma\in \Sy_n
}}\frac{\mathrm{sgn}(\sigma)}{m!l!\alpha!\beta!}d^m g_l(f_0(x))
\big((f_\alpha\times f_\beta)\left(x\right)(v^{
\sigma } )\big)\\
&+\sum_{\sigma\in\Sy_n}\frac{1}{n!}
 g_n(f_0(x))\big((f_1\times\cdots\times f_1)(x)(v^\sigma)\big
),
\end{align*}
for $n>1$, $x\in\U_\R$, $v\in E_1^n$. Note that in the second summand the sum 
over 
$\Sy_n$ together with the factor $\frac{1}{n!}$ can be omitted because the 
expression is already alternating. With 
$(f_1(x)(v_1),\ldots,f_1(x)(v)):=v'$ and $f_0(x):=x'$ it follows from the 
definition 
of $ g_n$ that $(( g_r)_r\circ(f_r)_r)_n=0$. This implies
$( g_r)_r\circ(f_r)_r=(\id_{\U_\R},c_{\id_{E_1}},0,0,\ldots)$, which is 
the skeleton of the identity $\id_\U\colon\U\to\U$. Thus, $(f_n)_n$ has a 
left inverse. Since the same construction also works for $( g_n)_n$, the left 
inverse of $(f_n)_n$ also has a left inverse. Therefore, $( g_n)_n$ is 
the inverse of $(f_n)_n$ and $f$ is invertible in $\SDom^{(k)}$.
\end{proof}
In general, it is quite difficult to check that smooth bijective maps between 
locally convex spaces are diffeomorphisms. However, if the map has the form of
$f_{\L_1}$ in the above lemma, a result of Hamilton (\cite[Theorem 5.3.1, 
p.102]{Ham}) can be directly 
generalized to the locally convex case. We do not need this result in the 
sequel but since it might be of interest for inverting supersmooth maps, we 
state it nevertheless.
\begin{lemma}[{\cite[Lemma 2.3, p.11]{GloNeeb17}}]
 Let $E_0, E_1$ and $F_1$ be locally convex spaces, $U\subseteq E_0$ open and
$f\colon U\times E_1\to F_1$ be smooth such that $f_x:=f(x,\bl)\colon E_1\to
F_1$ is linear for all $x\in U$.
If $f_x$ is invertible for all $x\in U$ and $g\colon U\times F_1\to E_1,\
(x,v)\mapsto f_x^{-1}(v)$ is continuous, then $g$ is smooth. Moreover, we have
\[
 d_1g(x,v)(u)=-g\Big(x, d_1 f\big(x,g(x,v)\big)(u)\Big)
\]
for $x\in U$, $v\in F_1$ and $u\in E_0$.
\end{lemma}
It is easy to generalize this to the situation where additionally a 
diffeomorphism 
$f_0\colon U\to V$ between open sets of locally convex spaces is involved.

\subsubsection{Generalizations}\label{subsectgendom}
One obvious generalization is to consider a differential calculus for other 
base fields (or even rings) than $\R$. A robust framework for this is 
provided by \cite{BGN} and then further developed for the super case in 
\cite{AllLau}. In the most general case, one has a unital commutative Hausdorff 
topological ring $R$ such that the group of units $R^\times$ is dense, i.e., 
integers need not necessarily be invertible.
For simplicity's sake, we formulated our results over $\R$ but we made a 
conscious effort to make them easily adaptable to more general situations.

In this way
Lemma \ref{lemdnfnat} through Proposition \ref{propfdecomp} can easily be shown 
to hold in the most general case.
While Corollary \ref{corssinbas} and Lemma \ref{lemdfss} also translate, 
our definition of supersmoothness just means $\mathcal{C}^1_{MS}$ (together 
with smoothness over $\R$) in the 
terminology of \cite{AllLau}. Note however that $\mathcal{C}^1_{MS}$ is 
equivalent to $\mathcal{C}^\infty_{MS}$ if $R$ is an $\Q$-algebra and one 
has smoothness over $R$ (see 
\cite[Proposition 2.18, p.580]{AllLau}). In this case Lemma \ref{lemtaylor}, 
Proposition \ref{propskel}, Proposition \ref{propskelmult} and Lemma 
\ref{lemskelinv} carry over as well.

It should be noted that Remark \ref{remskelpart} enables us to
show an analog to Proposition \ref{propskel} if not all 
integers are invertible in $R$, i.e., supersmooth maps are given by 
something like skeletons 
even in the most general case. The resulting analog to 
the composition formula from Proposition \ref{propskelmult}
can be obtained with general results about multilinear bundles (compare Remark 
\ref{remgenorder}) and a similar induction as in Lemma \ref{lemskelinv} leads 
to an inversion formula (compare \cite[Theorem MA.6(2), p.172]{Bert}).

The second apparent generalization is to define morphisms of finite 
differentiability order
$n\in\N_0$. Given only $k$-superdomains with $k\leq n$, one can simply 
define 
$k$-skeletons where the differentiability class of the components is 
appropriately chosen. For a more detailed discussion see \cite[10.1, 
p.420f.]{Mol}. 

\subsection{Supermanifolds}
The construction of supermanifolds from superdomains is conceptually very 
close to the respective construction of manifolds. In the categorical 
approach 
proposed by Molotkov in \cite{Mol}, one defines a Grothendieck topology on 
$\Top^\Gr$ that takes the same role as the usual topology in the manifold case. 
As model space one uses functors of the form $\ol{E}$ for $E\in\SVec$ with 
open subfunctors $\U$ as the open subsets (respectively functors isomorphic 
to such functors). A supermanifold is then a functor 
$\M\in\Man^\Gr$\nomenclature{$\M$}{} 
together with an atlas consisting of
natural transformations $\varphi\colon\U\to\M$, such that the 
change of charts is supersmooth. Here  a technical problem arises. In 
this approach, the intersection of two chart domains in $\M$ is defined as a 
 fiber product in the category $\Man^\Gr$, which is not guaranteed to be a 
superdomain. This has to be demanded in the definition. We avoid this and other 
technicalities by using concrete definitions of the  
model spaces. For a 
concise version of the categorical approach see \cite[p.591 
ff.]{AllLau}. 

We introduce $k$-supermanifolds in the same
way by considering functors $\Man^{\Grk{k}}$ and obtain 
respective categories 
$\SMank{k}$ for $k\in\N_0\cup\{\infty\}$.
One has the restriction functors 
$\pi^m_n\colon\SMank{m}\to\SMank{n}$ for $n\leq m$ and the inclusion functors
$\i^0_k\colon\SMank{0}\to\SMank{k}$ and $\i^1_k\colon\SMank{1}\to\SMank{k}$, 
which play an important part in understanding the structure of supermanifolds.
Note in particular that $\SMank{0}\cong\Man$ and $\SMank{1}\cong\VBun$.

These statements are not particularly difficult to prove and were already 
stated in \cite{MolICTP}. Noteworthy new results include the following. For any 
supermanifold $\M$, we show that $\M_{\L_n}$ has the natural structure of a so 
called multilinear bundle of degree $n$ over $\M_\R$.
What is more, $(\M_{\L_n})_{n\in\N_0}$ forms an inverse system of multilinear 
bundles which enables us to obtain a functor
\[\textstyle
 \SMan\to\Man,\quad \M\mapsto\varprojlim_n\M_{\L_n}
\]
in Theorem \ref{thrmsmanmbun}. As already mentioned, this functor has good 
properties such as respecting products.
Another important result is the characterization
of purely even supermanifolds in 
terms of higher tangent bundles of the base manifold in Proposition 
\ref{proppeventk}.

\begin{definition}
 Let $k\in\N_0\cup\{\infty\}$, $E\in\SVec$ and $\M\in\Top^{\Grk{k}}$ Hausdorff. 
Recall 
Definition \ref{defcovering}. A 
covering $\A:=\{\varphi^\alpha\colon\U^\alpha\to\M\colon\alpha\in A\}$ of $\M$ 
such 
that all $\U^\a$ are open subfunctors of $\ol{E}^{(k)}$
is called an \textit{atlas}\index{Supermanifold!atlas} of $\M$ if
the natural transformations 
\[
\varphi^{\a\b}:=(\varphi^\b)^{-1}\circ\varphi^\a|_{\U^{\a\b}}\colon\U^{\a\b}
\to\U^{\b\a}
\]
are supersmooth for all $\a,\b\in A$. Two atlases $\A$ and $\B$ are called 
\textit{equivalent} if their union $\A\cup\B$ is again an atlas. As with 
ordinary 
manifolds, this clearly defines an equivalence relation and we call the pair 
$(\M,[\A])$ a $k$\textit{-supermanifold 
modelled on}\index{Supermanifold!$k$-} $E$. 
If $k=\infty$ we also simply call $\M$ a 
\textit{supermanifold}\index{Supermanifold}.
We 
will usually omit $[\A]$ from our notation and if we talk about an atlas of a 
supermanifold, it is meant to belong to this equivalence class.
An element of any of the equivalent atlases will be called a
\textit{chart}\index{Supermanifold!chart} of $\M$. For any two charts 
$\varphi^\a$ and $\varphi^\b$, we call $\varphi^{\a\b}$ the \textit{change of 
charts}\index{Supermanifold!change of charts}. 

A \textit{morphism}\index{Supermanifold!morphism of} $f\colon\M\to\Nc$ of 
$k$-supermanifolds $\M$ and $\Nc$ is a natural transformation $f\colon\M\to\Nc$ 
such that for any chart $\varphi\colon\U\to\M$ and any chart 
$\psi\colon\V\to\Nc$ 
\[
 \psi^{-1}\circ f\circ\varphi|_{(f\circ\varphi)^{-1}(\psi(\V))}\colon
 (f\circ\varphi)^{-1}(\psi(\V))\to \V
\]
is supersmooth.
\end{definition}
Note that the definition of morphisms between $k$-supermanifolds is independent 
of 
the atlases, because change of charts satisfies the cocycle condition. As with 
ordinary manifolds, one sees that the composition of two morphisms of 
supermanifolds is again a morphism by inserting charts between them. Thus, we 
get for every $k\in\N_0\cup\{\infty\}$ the category 
$\SMank{k}$
of $k$-supermanifolds. As always, 
we set $\SMan:=\SMank{\infty}$\nomenclature{$\SMan$, $\SMank{k}$}{}. 
For two $k$-supermanifolds $\M,\Nc$, we denote by $\SC(\M,\Nc)$
\nomenclature{$\SC(\M,\Nc)$}{}
the set of 
supersmooth morphisms $f\colon\M\to\Nc$.

\begin{definition}
 A $k$-supermanifold $\M$ modelled on $E\in\SVec$ is
 a \textit{finite-dimensional}\index{Supermanifold!finite-dimensional}, 
\textit{Banach}\index{Supermanifold!Banach} or 
\textit{\Frechet}\index{Supermanifold!\Frechet} $k$-supermanifold if $E$ is so.
If $E_1=\{0\}$, 
then $\M$ is \textit{purely even}\index{Supermanifold!purely even} and if 
$E_0=\{0\}$, then $\M$ is \textit{purely odd}\index{Supermanifold!purely odd}.
 We call $\M_\R$ 
the \textit{base manifold}\index{Supermanifold!base manifold of a} of $\M$ and 
say that $\M$ is 
$\sigma$\textit{-compact}\index{Supermanifold!$\sigma$-compact} if $\M_\R$ is 
$\sigma$-compact. 
\end{definition}
\begin{remark}\label{remhausd}
If one allows non-Hausdorff 
manifolds in the definition, it is easily seen that a supermanifold $\M$ is 
Hausdorff if and only if its base manifold is Hausdorff. In fact, this 
follows because $\M_\L$ is a fiber bundle over $\M_\R$ whose typical fiber 
is Hausdorff by Theorem \ref{thrmsmanmbun} below.
\end{remark}
To get some intuition for supermanifolds, we start with several simple 
observations.
\begin{lemma}
 Let $k\in\N_0\cup\{\infty\}$ and $\M\in\SMank{k}$ with atlas 
$\{\varphi^\a\colon\U^\a\to\M\colon\a\in A\}$.
\begin{enumerate}
 \item For every $\L\in\Grk{k}$, 
$\{(\varphi_\L^\a|^{\varphi^\a_\L(\U^\a_\L)}_{\U^\a_\L})^{-1}\colon
\varphi^\a_\L(\U^\a_\L)\to\U^\a_\L
\colon\a\in A\}$ is an atlas of $\M_\L$.
 \item For $n\leq m< k+1$, the inclusions 
$\M_{\eta_{n,m}}\colon\M_{\L_n}\to\M_{\L_m}$ are topological embeddings and 
$\M_{\L_n}$ is a closed submanifold of $\M_{\L_m}$.
 \item For $n\leq m< k+1$, the projections 
$\M_{\ep_{m,n}}\colon\M_{\L_m}\to\M_{\L_n}$ are surjective.
\end{enumerate}
\end{lemma}
\begin{proof}
 (a) This is obvious from the definition of a supermanifold, since the sets
$\varphi^\a_\L(\U^\a_\L)$ form an open cover of $\M_\L$, 
$\varphi_\L^\a|^{\varphi^\a_\L(\U^\a_\L)}_{\U^\a_\L}$ is a homeomorphism 
and the change of charts is smooth.
 
 (b) Let $\M$ be modelled on $E\in\SVec$. In the charts defined by 
$\varphi^\a_{\L_n}$ and $\varphi^\a_{\L_m}$ as in (a), the map 
$\M_{\eta_{n,m}}$ has the form $\U_{\eta_{n,m}}$ and we have
$\U^\a_{\L_n}\cong\U^\a_{\eta_{n,m}}(\U^\a_{\L_n})=
\U^\a_{\L_m}\cap\ol{E}^{(k)}_{\L_n}$. By naturality, we have
$\varphi_{\L_m}^\a(\U^\a_{\eta_{n,m}}(\U^\a_{\L_n}))=\M_{\L_m}\cap\M_{\eta_{n,m}
} (\varphi^\a_{\L_n}(\U^\a_{\L_n}))$.
 
 (c) In the charts defined by 
$\varphi^\a_{\L_n}$ and $\varphi^\a_{\L_m}$ as in (a), the map 
$\M_{\ep_{m,n}}$ has the form $\U_{\ep_{m,n}}$ which clearly defines a 
surjective map.
\end{proof}
Part (c) of this lemma already suggests that $\M_{\L_m}$ is some kind of fiber 
bundle over $\M_{\L_n}$. As we discuss below, this fiber bundle structure can be
accurately described via multilinear bundles.
Like ordinary manifolds, supermanifolds and morphisms thereof arise 
from local data.
\begin{proposition}[{see \cite[Proposition 3.23, 
p.593]{AllLau}}]\label{proplocaldescsman}
 Let $k\in\N_0\cup\{0\}$ and $E\in\SVec$. Let further $(\U^\alpha)_{\a\in A}$ 
be a family of open subfunctors of $\ol{E}^{(k)}$ and 
$\U^{\a\a'}\subseteq\U^\a$ be open subfunctors for $\a,\a'\in A$ such that 
$\U^{\a\a}=\U^\a$. 
Further, let
$\varphi^{\a\a'}\colon\U^{\a\a'}\rightarrow\U^{\a'\a}$ be isomorphisms in 
$\SDomk{k}$ such that we have $\varphi^{\a\a}=\id_{\U^\a}$ 
and $\varphi^{\a\a''}=\varphi^{\a'\a''}\circ\varphi^{\a\a'}$ on 
$\U^{\a\a'}\cap\U^{\a\a''}$ for all $\a,\a',\a''\in A$.
Finally, for all $\a,\b\in A$ and any two points $x\in\U^{\a}_\R$, 
$y\in\U^{\b}_\R$ such that $x\notin\U^{\a\b}_\R$ or $\varphi^{\a\b}_\R(x)\neq 
y$, let there exist open 
neighbourhoods $V\subseteq\U^{\a}_\R$ of $x$ and $V'\subseteq\U^{\b}_\R$ of 
$y$, such that $\varphi^{\a\b}_\R\big(\U^{\a\b}_\R\cap V\big)\cap V'=\emptyset$.
Then there exists a, 
up to unique isomorphism, unique $k$-supermanifold $\M$ with an atlas 
$\{\varphi^\a\colon\U^\a\rightarrow\M\colon\a\in A\}$ such that that the change 
of 
charts coincides with the $\varphi^{\a\a'}$ defined above.

 Moreover, let $\Nc\in\SMank{k}$ have the atlas 
$\{\psi^\b\colon\V^\b\to\Nc\colon\b\in B\}$ and 
let $\widetilde{\U}^{\a\b}\subseteq\U^\a$ for $\a\in A$ and $\b\in B$ such that 
$\bigcup_{\b\in B}\widetilde{\U}^{\a\b}_\R=\U^\a_\R$.
If
$f^{\a\b}\colon\widetilde{\U}^{\a\b}\to\V^\b$ is a family 
of supersmooth maps 
such that $\psi^{\b\b'}\circ 
f^{\a\b}\circ\varphi^{\a'\a}=f^{\a'\b'}$ on 
$(\varphi^{\a'\a})^{-1}(\widetilde{\U}^{\a\b})\cap 
(f^{\a'\b'})^{-1}(\V^{\b'\b})$ for all $\a,\a'\in A$, $\b,\b'\in B$, then there 
exists a unique supersmooth morphism $f\colon\M\rightarrow\Nc$ with 
$f^{\a\b}=(\psi^\b)^{-1}\circ f\circ \varphi^\a|_{\widetilde{\U}^{\a\b}}$.
\end{proposition}
\begin{proof}
 This follows exactly as in \cite[Proposition 3.23, p.593]{AllLau}. 
Essentially, 
we use the well-known equivalent statement for ordinary manifolds for every 
$\L\in\Grk{k}$ to construct $\M_\L$, resp.\ $f_\L$, and the rest follows from 
naturality. Note that 
$\bigcup_{\b\in B}\widetilde{\U}^{\a\b}_\R=\U^\a_\R$ implies
$\bigcup_{\b\in B}\widetilde{\U}^{\a\b}_{\L}=\U^\a_\L$ for all $\L\in\Grk{k}$.
Moreover, if $\M_\R$ is Hausdorff then $\M_\L$ is Hausdorff for all $\L\in\Gr$ 
(compare Remark \ref{remhausd}).
\end{proof}
\begin{lemma}[{\cite[Corollary 6.2.2, p.409]{Mol}}]\label{lemsmaninvert}
 Let $k\in\N\cup\{\infty\}$ and $\M,\Nc\in\SMank{k}$. A supersmooth morphism 
$f\colon \M\to\Nc$ is an isomorphism in $\SMank{k}$ if and only if 
$f_{\L_1}\colon 
\M_{\L_1}\to\Nc_{\L_1}$ is a diffeomorphism.
\end{lemma}
\begin{proof}
Clearly, 
$f\colon \M\to\Nc$ is an isomorphism if and only if $f_\L\colon\M_\L\to\Nc_\L$ 
is bijective and the maps $f_\L^{-1}$ define a supersmooth natural 
transformation for every $\L\in\Grk{k}$.  In particular, $f_{\L_1}$ is a 
diffeomorphism in this situation.
Let $\{\varphi^\a\colon\U^\a\to\M\colon\a\in A\}$ be an atlas of $\M$ and 
$\{\psi^\b\colon\V^\b\to\Nc\colon\b\in B\}$ be an atlas of $\Nc$.
Let $f\colon \M\to\Nc$ be supersmooth such that $f_{\L_1}$ is a diffeomorphism.
For all $\a\in A$ and $\b\in B$ we define
$\widetilde{\U}^{\a\b}:=(f\circ\varphi^\a)^{-1}(\psi^\b(\V^\b))\subseteq \U^\a$
and 
$\widetilde{\V}^{\b\a}:=f^{\a\b}\big((f\circ\varphi^\a)^{-1}
(\psi^\b(\V^\b))\big)=(\psi^\b)^ { -1
} (f(\varphi^\a(\U^\a)))\subseteq\V^\b$
and let
\[
f^{\a\b}:=(\psi^\b)^{-1}\circ f\circ 
\varphi^\a|_{\widetilde{\U}^{\a\b}}\colon\widetilde{\U}^{\a\b}\to\widetilde{\V}^
{\b\a}.
\]
Since $f_\R$ is also a diffeomorphism, the sets
$\widetilde{\V}^{\b\a}_\R$
cover $\Nc_\R$ and
because every $f^{\a\b}_{\L_1}$ is a diffeomorphism, there exist unique 
supersmooth inverse morphisms 
$(f^{\a\b})^{-1}\colon\widetilde{\V}^{\b\a}\to\widetilde{\U}^{\a\b}$
by 
Lemma \ref{lemskelinv}. 
For every $\a,\a'\in A$ and $\b,\b'\in B$, we have 
$(\psi^{\b'\b})^{-1}\circ f^{\a'\b'}\circ\varphi^{\a\a'}=f^{\a\b}$
on $\widetilde{\U}^{\a\b}\cap(\varphi^{\a\a'})^{-1}(\widetilde{\U}^{\a'\b'})$.
Therefore,
$(f^{\a\b})^{-1}=(\varphi^{\a\a'})^{-1}\circ (f^{\a'\b'})^{-1}\circ\psi^{\b\b'}$
on $\widetilde{\V}^{\b\a}\cap (\psi^{\b\b'})^{-1}(\widetilde{\V}^{\b\b'})$
and the morphisms lead to a unique supersmooth morphism $f^{-1}\colon\Nc\to\M$ 
by
Proposition \ref{proplocaldescsman}. That it is inverse to $f$ follows from the 
local description of $f^{-1}\circ f$ and $f\circ f^{-1}$.
\end{proof}
\begin{definition}
 Let $k\in\N_0\cup\{\infty\}$ and $\M\in\SMank{k}$ be modelled on $E\in\SVec$. 
A subfunctor $\Nc$ of $\M$ 
is called a \textit{sub-supermanifold}\index{Supermanifold!sub-} of $\M$ if 
there exist sequentially closed vector subspaces $F_0\subseteq E_0$ and 
$F_1\subseteq 
E_1$ such that
for every $x\in\Nc_\R$ there exists a chart $\varphi^\a\colon\U^\a\to\M$ of 
$\M$ with $x\in\varphi^\a_\R(\U^\a_\R)$ such that
$\varphi^\a(\U^\a\cap\ol{F}^{(k)})=\varphi^\a(\U^\a)\cap\Nc$, where 
$F:=F_0\oplus F_1\in\SVec$.

We call 
$\varphi^\a|_{\U^\a\cap\ol{F}^{(k)}}$ a \textit{sub-supermanifold 
chart}\index{Supermanifold!sub-supermanifold chart} of $\Nc$. Taking all 
sub-supermanifold 
charts of $\Nc$ as the atlas turns $\Nc$  into a supermanifold and we always 
give $\Nc$ this structure.
\end{definition}
\begin{lemma}
 Let $k\in\N_0\cup\{\infty\}$, $\M\in\SMank{k}$ and $\Nc$ be a 
sub-super\-manifold 
of $\M$. Then the inclusion $i\colon\Nc\to \M$ is supersmooth.
\end{lemma}
\begin{proof}
 By definition of a subfunctor, the inclusion is a natural 
transformation. Let $\M$ be modelled on $E\in\SVec$, $\Nc$ be modelled on 
$F\subseteq E$ and $\{\varphi^\a\colon\U^\a\to\M\colon\a\in A\}$ be a 
collection of 
charts such that $\{\varphi^\a|_{\U^\a\cap\ol{F}^{(k)}}\colon\a\in A\}$ is an 
atlas 
of $\Nc$. In these charts the inclusion is just the inclusion 
$\U^\a\cap\ol{F}^{(k)}\to\U^\a$, which is obviously supersmooth.
\end{proof}
\begin{lemma/definition}\label{lemopensubsuper}
 Let $k\in\N_0\cup\{\infty\}$ and $\M\in\SMank{k}$. For every open subfunctor 
of $\U\subseteq\M$, we have $\U=\M|_{\U_\R}$. In this case $\U$ is a 
sub-supermanifold of $\M$ and if $f\colon\M\to\Nc$ is a supersmooth morphism to 
$\Nc\in\SMank{k}$, then so is $f|_{\U}\colon\U\to\Nc$.
We call such sub-supermanifolds \textit{open 
sub-supermanifolds}\index{Supermanifold!open sub-supermanifold}.
\end{lemma/definition}
\begin{proof}
 That $\U=\M|_{\U_\R}$ holds for $k=\infty$ follows directly from 
\cite[Corollary 
3.5.9, p. 62]{SachseDiss} and the same proof works for $k\in\N_0$ if one only 
considers $\L\in\Grk{k}$. Let $\{\varphi^\a\colon\U^\a\to\M\colon \a\in A\}$ be 
an atlas of $\M$. Then 
$\{\varphi^\a|_{(\varphi^\a)^{-1}(\varphi^\a(\U^\a)\cap\U)}\colon\a\in A\}$ is 
an atlas of $\U$. With these charts, the supersmoothness of $f|_\U$ is obvious.
\end{proof}

\begin{definition}
 Let $k\in\N_0\cup\{\infty\}$ and $\M,\Nc\in\SMank{k}$ be modelled on 
$E,F\in\SVec$ with atlases 
$\{\varphi^\a\colon\U^\a\to\M\colon \a\in A\}$ and 
$\{\psi^\b\colon\V^\b\to\Nc\colon\b\in B\}$. We define the 
\textit{product}\index{Supermanifold!product of} $\M\times\Nc$ of $\M$ and 
$\Nc$ as the functor $\L\mapsto\M_\L\times\Nc_\L$, resp.\ $\varrho\mapsto 
\M_\varrho\times\Nc_\varrho$, for $\L,\L'\in\Grk{k}$ and 
$\varrho\in\Hom_{\Grk{k}}(\L,\L')$. We will always give $\M\times\Nc$ the 
structure of a $k$-supermanifold modelled on $E\times F$ defined by the 
atlas
$\{\varphi^\a\times\psi^\b\colon\U^\a\times\V^\b\to\M\times\Nc\colon(\a,\b)\in 
A\times B\}$.
\end{definition}
Clearly, the projections $\pi_{\M}\colon\M\times\Nc\to\M$ and 
$\pi_{\Nc}\colon\M\times\Nc\to\Nc$ are supersmooth morphisms.

Recall the definition of multilinear bundles and inverse systems of multilinear 
bundles from Appendix \ref{chapmullin}. The following theorem shows 
that for a supermanifold $\M$, the manifolds $\M_{\L_n}$ are multilinear 
bundles 
of degree $n$ over $\M_\R$ and that $(\M_{\L_m},\M_{\ep_{m,n}})$ is an 
inverse system of multilinear bundles. This lets us consider supermanifolds as 
ordinary manifolds.
\begin{theorem}\label{thrmsmanmbun}
 Let $k\in\N_0\cup\{\infty\}$, $\M,\Nc\in\SMank{k}$ and $f\colon\M\to\Nc$ be 
supersmooth. If $\M$ is modelled on $E\in\SVec$ with the atlas 
$\{\varphi^\a\colon\a\in A\}$, 
 then $\M_{\L_n}$ is a multilinear bundle of degree $n$ over $\M_\R$ with the 
fiber $\ol{E}_{\L_n^+}$ and the
bundle atlas $\{\varphi^\a_{\L_n}\colon\a\in A\}$
 for every $\L_n\in\Grk{k}$. Moreover, $f_{\L_n}\colon\M_{\L_n}\to\Nc_{\L_n}$ 
is a 
morphism of multilinear bundles of degree $n$.
With this, we obtain a faithful functor
 \[
  \SMank{k}\to\MBunk{k},
 \]
  defined by $\M\mapsto\M_{\L_k}$ and $f\mapsto f_{\L_k}$ for $k\in\N_0$.
Furthermore, if $k=\infty$, then $(\M_{\L_m},\M_{\ep_{m,n}})$ is an 
inverse system of multilinear bundles with the adapted atlas 
$\{(\varphi^\a_{\L_n})^{-1}\colon n\in\N_0,\a\in A\}$
and
\[
 \varprojlim\colon\SMan\to\MBunk{\infty},
\]
defined by $\M\mapsto\varprojlim_n\M_{\L_n}$ and $f\mapsto\varprojlim_n 
f_{\L_n}$, is a faithful functor. Along the forgetful functor, we have thus 
constructed faithful 
functors 
\[
 \SMank{k}\to\Man
\]
for $k\in\N_0\cup\{\infty\}$. All these functors respect products.
\end{theorem}
\begin{proof}
 Let $\M$ be modelled on $E\in\SVec$. We start by showing that 
$\{\varphi_{\L_n}^\a\colon\U^\a_{\L_n}\to\M_{\L_n}\colon\a\in A\}$ is indeed a 
bundle atlas of a multilinear bundles of degree $n$. 
Let the 
change of charts $\varphi^{\a\b}$ be defined by the skeleton 
$(\varphi^{\a\b}_n)$. 
We consider $\U^{\a}_{\L_n}=\U_\R^{\a}\times\prod_{I\in\Po^n_+}\oddgen_I 
E_{\ol{|I|}}$ as a trivial multilinear bundle over the $n$-multilinear space 
$(E_I)$ with $E_I:=\oddgen_IE_{\ol{|I|}}$.
By naturality, we have 
$(\varphi_{\L_n}^\a)^{-1}\big(\M_{\ep_{\L_n}}^{-1}(\{x\})\big)=
(\U^\a_{\L_n})^{-1}\big(\varphi^{-1}_\R(\{x\})\big)$ for all 
$x\in\varphi^\a_\R(\U^\a_\R)$. In other words, the projection 
$\M_{\ep_{\L_n}}\colon\M_{\L_n}\to\M_\R$ turns $\M_{\L_n}$ into a fiber bundle
with typical fiber $\ol{E}_{\L_n^+}$.
Recall the sign of a partition defined in Remark \ref{defsgnpart}. 
Then 
$\oddgen_{\omega_1}\cdots\oddgen_{\omega_{\ell(\omega)}}=\sgn(\omega)\oddgen_I$ 
for all $I\in\Po^n_+$ and $\omega\in\Part(I)$.
With this, we use Remark \ref{remskelpart} to calculate
\[
  \textstyle
 \varphi^{\a\b}\big(x+\sum_{I\in\Po^n_+}\oddgen_I x_I\big)
 =
 \displaystyle\varphi^{\a\b}_0(x)+
 \sum_{I\in\Po^n_+}\sum_{\omega\in\Part(I)}\oddgen_I \sgn(\omega)d^{(e(\omega))}
  \varphi^{\a\b}_{o(\omega)}(x)(x_\omega)
\]
for $\omega$ in graded lexicographic order $x_I\in E_{\ol{|I|}}$ for 
$I\in\Po^n_+$ and
\[
x_\omega:=(\oddgen_{\omega_1}x_{\omega_1},\ldots,\oddgen_{\omega_{\ell(\omega)}}
x_{\omega_{\ell(\omega)}}).
\]
In the notation of multilinear bundles, the change of chart is thus given by 
the sum of maps of the form 
$(\varphi^{\a\b})_x^\omega:=\sgn(\omega)d^{(e(\omega))}
  \varphi^{\a\b}_{o(\omega)}(x)$, which define an isomorphism of 
$n$-multilinear 
spaces for every $x\in\U^{\a\b}_\R$. Thus, $\M_{\L_n}$ is a multilinear bundle 
over $\M_\R$ of degree $n$ with typical fiber $(E_I)$.

For a morphism $f\colon\M\to\Nc$, we first note that by naturality 
$f_\R\circ\M_{\ep_{\L_n}}=\Nc_{\ep_{\L_n}}\circ f_{\L_n}$ and therefore 
$f_{\L_n}$ 
is a fiber bundle morphism over $f_\R$. In bundle charts, we can make the exact 
same argument as above to see that $f_{\L_n}$ is a morphism of multilinear 
bundles.

It follows that we have a functor $\SMank{k}\to\MBunk{k}$ as 
described in the theorem for $k\in\N_0$. Next, we show that 
$(\M_{\L_m},\M_{\ep_{m,n}})$ 
defines an inverse system of multilinear bundles 
if $\M\in\SMan$. We have 
$\M_{\ep_{m,n}}\circ\varphi^\a_{\L_m}=\varphi^\a_{\L_n}\circ\U^\a_{\ep_{m,n}}$
for all $n\leq m$ and therefore $\U^\a_{\ep_{km,n}}$ is the chart 
representation of $\M_{\ep_{m,n}}$. Hence, in terms of multilinear bundles, 
$\M_{\ep_{m,n}}$ is exactly the projection defined in Lemma \ref{lemqkn}. 
It follows that $\M_{\L_m}|_{\Po^n_+}=\M_{\ep_{m,n}}(\M_{\L_m})=\M_{\L_n}$
and 
$\varphi^\a_{\L_m}|_{\Po^n_+}=\varphi^\a_{\L_m}\circ\U^\a_{\eta_{n,m}}
=\M_{\eta_{n,m}}\circ\varphi^\a_{\L_n}$ shows that 
$\{(\varphi^\a_{\L_m})^{-1}\colon m\in\N_0,\a\in\A\}$ is indeed an adapted 
atlas.  

On morphisms $f\colon\M\to\Nc$, $\Nc\in\SMan$, we have likewise
$f_{\L_n}\circ\M_{\ep_{m,n}}=\Nc_{\ep_{m,n}}\circ f_{\L_m}$ which shows that 
$(f_{\L_m})_{m\in\N_0}$ is a morphism of inverse systems of multilinear bundles.

It is clear from the definitions that products of supermanifolds correspond to 
products of inverse systems. 
\end{proof}
\begin{remark}
 In \cite[Remark 3.3.1, p.392]{Mol} Molotkov constructs a functor 
$\Man^\Gr\to\Man$ by taking the disjoint union of the $\M_\L$ for 
$\M\in\Man^\Gr$ and $\L\in\Gr$. He also considers this as a functor 
$\SMan\to\Man$ along the forgetful functor. For one, this functor relies on a 
more general definition of manifolds where the model spaces of different 
connected components may be non-isomorphic. More critically, this functor does 
not respect products, leading Molotkov to state that ``Lie supergroups (groups 
of 
the category $\SMan$) are not groups at all (considered in $\Set$'' 
\cite[Ibid.]{Mol}.
We hope to have convinced the reader with the above theorem that Lie 
supergroups can be seen not only as groups but even as Lie groups in a natural 
way.
\end{remark}
\begin{remark}
 Consider the following type of fiber bundles. For $E\in\SVec$ let the base 
manifold $M$ be modelled on $E_0$, let the typical fiber be
 $\varprojlim_n \ol{E}_{\L^+_n}$ and let the transition functions  
come from the limit of skeletons as in Theorem \ref{thrmsmanmbun}. The 
morphisms of such bundles shall locally also come from limits of
skeletons. Obviously, these bundles are elements of $\MBunk{\infty}$ and 
restricting to this subcategory turns the functor $\varprojlim$ into an 
equivalence of categories.
If $E$ is finite-dimensional, we have that $\varprojlim_n \ol{E}_{\L^+_n}$ is a 
\Frechet space. Consequently, non-trivial finite-dimensional supermanifolds are 
mapped to 
\Frechet manifolds under $\varprojlim$.
\end{remark}
 One can reconstruct the original supermanifold $\M$ from 
$\varprojlim\M$ if one keeps track of any atlas of $\varprojlim\M$ coming from 
the limit of an atlas of $\M$. An interesting problem is whether one can at 
least recover the isomorphism class of a supermanifold without a specific atlas.
\begin{problem}\label{probliminj}
 Is the functor $\varprojlim\colon\SMan\to\MBunk{\infty}$ injective on 
isomorphism 
classes, i.e., do we have $\varprojlim\M\cong\varprojlim\Nc$ in 
$\MBunk{\infty}$ if and 
only if we have $\M\cong\Nc$ in $\SMan$?
\end{problem}
If $\M_\R$ admits a smooth partition of unity, then it follows from 
Batchelor's Theorem \ref{thrmbatch} below that this is the case because 
$\varprojlim\M\cong\varprojlim\Nc$ in $\MBunk{\infty}$ implies 
$\M_{\L_1}\cong\Nc_{\L_1}$ in $\VBun$. 
The functor
$\varprojlim\colon\SMan\to\Man$
is not injective on isomorphism classes. For example
\[
\varprojlim\ol{\R^{1|0}}\cong\prod_{\substack{I\subseteq\N,|I|<\infty,\\
|I|\ \text{even}}}\R\cong\prod_
{ n\in\N } \R\cong\R\times\prod_{\substack{I\subseteq\N,|I|<\infty,\\
|I|\ \text{odd}}}\R\cong\varprojlim\ol{\R^{0|1}}
\]
in the category $\Man$.

We have seen in Theorem \ref{thrmsmanmbun} how to embed the category of 
supermanifolds into the category of manifolds. Conversely, one can also embed 
the category $\Man$ into the category $\SMan$.
For this, let $\Dom\nomenclature{$\Dom$}{}$ denote the category consisting of 
pairs $(U,E_0)$ where $E_0$ is a Hausdorff locally convex space and 
$U\subseteq E_0$ is open and where the morphisms are smooth maps between the 
open subsets.
\begin{proposition}[{\cite[cf. Proposition 4.2.1, p.396]{Mol}}]\label{propisman}
  Let $k\in\N_0\cup\{\infty\}$. We 
define a functor 
  \[
   \i^0_k\colon\Dom\to\SDomk{k}
  \]
  by setting $\i^0_k(U):=\ol{E}^{(k)}|_U$ and $\i^0_k(f_0):=(f_0,0,0,\ldots)$
  for $(U,E_0)\in\Dom$ and $E:=E_0\oplus\{0\}\in\SVec$. This functor 
extends to 
a 
   fully faithful functor
  \[
   \i^0_k\colon\Man\to\SMank{k}.
  \]
  In case of $k=\infty$ we also write 
$\i\colon\Man\to\SMan$\nomenclature{$\i$, $\i^0_k$}{}. The functor 
$\i^0_0\colon\Man\to\SMank{0}$ is an equivalence of categories. All of these 
functors respect 
products.
\end{proposition}
\begin{proof}
 It follows from the composition formula in Proposition \ref{propskelmult}
 that $\i^0_k\colon\Dom\to\SDomk{k}$ is a functor.
 Let $M$ be a manifold modelled on $E_0$ with atlas $\{\varphi_\a\colon 
V_\a\to U_\a\colon\a\in A\}$. Applying this functor to the change of charts 
$\varphi_{\a\b}\colon U_{\a\b}\to U_{\b\a}$,
 defines an (up to unique isomorphism) unique supermanifold $\M$
 modelled on $E_0\oplus\{0\}$ with the atlas 
$\{\i^0_k\big((\varphi_\a)^{-1}\big)\colon\i^0_k(U_\a)\to\M\colon\a\in A\}$
 by
 Proposition \ref{proplocaldescsman}.
 If $N\in\Man$ has the atlas $\{\psi_\b\colon V'_\b\to U'_\b\colon\b\in B\}$ 
and $f\colon M\to N$ is a smooth map then the same proposition applied to
$f_{\a\b}=\psi_\b\circ f\circ (\varphi_\a)^{-1}$ leads to a unique 
morphism 
$\i^0_k(f)\colon \i^0_k(M)\to\i^0_k(N)$ such that
$\big(\i^0_k(f)\big)^{\a\b}=\i^0_k(f_{\a\b})$. Functoriality follows again 
by the local definition of the composition of supersmooth morphisms.

The uniqueness of this construction shows that $\i^0_k$ is faithful.
On the other hand, every supersmooth map $g\colon\i^0_k(M)\to \i^0_k(N)$ is 
determined by its local chart descriptions $g^{\a\b}$, whose skeletons 
have the form $(g^{\a\b}_0,0,0\ldots)$ since $\i^0_k(M)$ is purely even. 
Clearly, the maps $g^{\a\b}_0$ define a unique smooth map $M\to N$ whose image 
under $\i^0_k$ is $g$.
We already know from Theorem \ref{thrmsmanmbun} that
 $\M\mapsto\M_\R$ and $f\mapsto f_\R$ defines a functor 
$\pi_0^0\colon\SMank{0}\to\Man$ and the above shows that 
$\pi_0^0\circ\i^0_0\cong\id_{\Man}$ and 
that $\i^0_0\circ\pi_0^0\cong\id_{\SMan}$.

It is obvious that the functor $\i^0_k\colon\Dom\to\SDomk{k}$ preserves 
products and from this it 
follows immediately that 
$\i^0_k\colon\Man\to\SMank{k}$ also preserves products.
\end{proof}
\begin{lemma}
 Let $k\in\N_0\cup\{\infty\}$. For every supermanifold 
$\M\in\SMank{k}$, we have that $\i^0_k(\M_\R)$ is a sub-supermanifold of $\M$. 
If $\M$ is purely even, we have
 $\i^0_k(\M_\R)\cong\M$.
\end{lemma}
\begin{proof}
  Let $\M$ be modelled on $E\in\SVec$ and let 
$\{\varphi^\a\colon\U^\a\to\M\colon\a\in A\}$ be an atlas of $\M$. 
If the changes of charts $\varphi^{\a\b}$ have the skeletons 
$(\varphi^{\a\b}_n)$, then the skeletons 
$(\varphi^{\a\b}_0,0,\ldots)$ define
$\i^k_0(\M_\R)$ by Proposition \ref{propisman}. Because 
$\varphi^{\a\b}|_{\ol{E_0\oplus\{0\}}^{(k)}}$ has the 
skeleton $(\varphi^{\a\b}_0,0,\ldots)$, it follows that $\i^0_k(\M_\R)$ is a 
sub-supermanifold of $\M$. If $\M$ is purely even, then changes of charts 
have 
the form $(\varphi^{\a\b}_0,0,\ldots)$ to begin with and it follows 
$\i^0_k(\M_\R)\cong\M$ by Proposition \ref{proplocaldescsman}.
\end{proof}
Purely even supermanifolds $\M$ can be described in terms of higher tangent 
bundles 
of $\M_\R$. This will be particularly important for the theory of Lie 
supergroups.
\begin{proposition}\label{proppeventk}
 Let $M$ be a manifold. Recall Example \ref{extinfty}. Using Lemma 
\ref{lemmbunminus} and Theorem \ref{thrmsmanmbun}, there are isomorphisms
 \[
  \Gamma^k_n\colon\i^0_k(M)_{\L_n}\to T^kM|_{\Poee{n}}^-
 \]
 of multilinear bundles of degree $n$ for every $n\leq k<\infty$. These 
isomorphisms 
are natural in $k$ and $n$ in the sense that 
\[
 (\i^0_k(M)_{\L_k},\i^0_k(M)_{\ep_{k,n}})\cong \big(T^kM|_{\Poee{k}}^-, 
\pi^k_n|_{\Poee{k}}^-\big)
\]
holds as inverse systems of multilinear bundles. 
It follows that $\Lambda_k\mapsto T^kM|_{\Poee{k}}^-$ can be made into a 
supermanifold isomorphic to $\i(M)$.
\end{proposition}
\begin{proof}
To show that $\i^k_0(M)_{\L_n}\cong T^nM|_{\Poee{n}}^-$ holds, we simply 
compare the 
change of charts. Let $M$ be modelled on $E_0$ and $\{\varphi_\a\colon V_\a\to 
U_\a\colon\a\in A\}$ be an atlas of $M$. For a change of charts 
$\varphi^{\a\b}\colon 
U_{\a\b}\to U_{\b\a}$, we have  
$\i^k_0(\varphi^{\a\b})=(\varphi^{\a\b},0,0,\ldots)$ and thus
\begin{align*}
\textstyle
 \i^k_0(\varphi^{\a\b})_{\L_n}&\big(x+\sum_{I\in\Poee{n}}\oddgen_I x_I\big)
 =\\
 \displaystyle
 &\varphi^{\a\b}(x)+
 \sum_{I\in\Poee{n}}\sum_{\omega\in\Parte(I)}
\oddgen_I 
\sgn(\omega)d^{(e(\omega))}
  \varphi^{\a\b}_{o(\omega)}(x)(x_\omega),
\end{align*}
where $x\in U_{\a\b}$ and
$x_\omega=(x_{\omega_1},\ldots,x_{\omega_{\ell(\omega)}})\in 
E_0^{e(\omega)}$. On the other hand, we know from Example \ref{extangent}(b) 
and 
the Lemma \ref{lemmbunminus} that the change of charts for 
$T^kM|_{\Poee{k}}^-$ is given by
\begin{align*}
\textstyle
 T^k{\varphi^{\a\b}}|_{\Poee{n}}^-&\big(x+\sum_{I\in\Poee{n}}\ep_I x_I\big)
 =\\
 \displaystyle
 &\varphi^{\a\b}(x)+
 \sum_{I\in\Poee{n}}\sum_{\omega\in\Parte(I)}
\ep_I 
\sgn(\omega)d^{(e(\omega))}
  \varphi^{\a\b}_{o(\omega)}(x)(x_\omega)
\end{align*}
for the same $x\in U_{\a\b}$ and $x_\omega\in E_0^{e(\omega)}$. Therefore,
there exists an isomorphism $\Gamma^k_n\colon\i^k_0(M)_{\L_n}\to 
T^kM|_{\Poee{n}}^-$ such that 
$(\Gamma^k_n)^\a:= 
T^k{\varphi^{\a}}|_{\Poee{n}}^-\circ\Gamma^k_n\circ\i^k_0(\varphi^{\a
} )_{\L_n}^{-1}$ is given by the
obvious isomorphisms of trivial 
$k$-multilinear bundles 
\[
(\Gamma^k_n)^\a\colon\i^k_0(U_{\a})_{\L_n}=U_\a\times\prod_{I\in\Poee{n}}
\oddgen_IE_{0} \to
U_\a\times\prod_{I\in\Poee{n}}\ep_IE_{0}= T^k|_{\Poee{n}}(U_\a)^-.
\]
 Note that for any $k$-multilinear bundle $F$, we have
 $(F|_{\Poee{k}})|_{\Po^n_+}=F|_{\Poee{n}}$ for  $n\leq k$
 and it follows  from the local description in Lemma \ref{lemqkn} that  
 $(q^k_n)^-\colon F|_{\Poee{k}}^-\to F|_{\Poee{n}}^-$ is just the respective 
projection of 
 $F|_{\Poee{k}}^-$. This shows that $\big(T^kM|_{\Poee{k}}^-, 
\pi^k_n|_{\Poee{k}}^-\big)$ is indeed an inverse system of multilinear 
bundles.
It is clear from the local description that 
$\Gamma^n_n\circ 
\i(M)_{\ep_{k,n}}=\pi^k_n|_{\Poee{k}}^-\circ\Gamma^k_k$, which shows 
$(\i(M)_{\L_k},\i(M)_{\ep_{k,n}})\cong \big(T^kM|_{\Poee{k}}^-, 
\pi^k_n|_{\Poee{k}}^-\big)$.

To turn $\Lambda_k\mapsto T^kM|_{\Poee{k}}^-$ into a supermanifold, one simply 
defines 
\[(T^kM|_{\Poee{k}}^-)_\varrho\colon T^kM|_{\Poee{k}}^-\to 
T^nM|_{\Poee{n}}^-\]
for every morphism $\varrho\colon\L_k\to\L_n$ via 
$\i(M)_{\varrho}\colon\i(M)_{\L_k}\to\i(M)_{\L_n}$ and the above isomorphisms. 
The charts are then given by
$(\Gamma^k_k \circ\i(\varphi_\a^{-1})_{\L_k}\big)_{\L_k}$.
\end{proof}

\begin{proposition}[{\cite[cf. Proposition 4.2.1, 
p.396]{Mol}}]\label{propvbunman}
 Let $k\in\N\cup\{\infty\}$. There is a faithful functor
 \[
  \i^1_k\colon\VBun\to\SMank{k}.\nomenclature{$\i^1_k$}{}
 \]
  The functor $\i^1_1\colon\VBun\to\SMank{1}$ is an equivalence of categories. 
All these functors respect products.
\end{proposition}
\begin{proof}
 The proof is very similar to the proof of Proposition \ref{propisman}.
 Let $\pi\colon F\to M$ be a vector bundle with typical fiber $E_1$ and bundle 
atlas $\{\varphi_\a\colon V_\a\to U^\a\times E_1\colon\a\in A\}$. The 
change of bundle charts 
$\varphi_{\a\b}\colon U_{\a\b}\times E_1\to U_{\b\a}\times E_1$ has the form 
$(\varphi^{\a\b}_0,\varphi^{\a\b}_1)$, where $\varphi^{\a\b}_0\colon 
U_{\a\b}\to 
U_{\b\a}$ is a smooth and $\varphi^{\a\b}_1\colon U_{\a\b}\times E_1\to 
E_1$ is smooth and linear in the second component. Note that there exists an 
atlas of $M$ such that the change of charts is given by $\varphi^{\a\b}_0$.
Let $M$ be modelled on $E_0$. We define the super vector space $E:=E_0\oplus 
E_1$ and let $\i^1_1(U_\a\times E_1):=\U^{\a}:=\ol{E}^{(k)}|_{U_\a}$, as well 
as 
$\U^{\a\b}:=\ol{E}^{(k)}|_{U_{\a\b}}$, for all $\a,\b\in A$.
Then 
$i^1_1(\varphi^{\a\b}):=(\varphi^{\a\b}_0,\varphi^{\a\b}_1,0,0,\ldots)\colon 
\U^{\a\b}\to\U^{\b\a}$ defines isomorphisms that satisfy the conditions of 
Proposition \ref{proplocaldescsman} because by the composition formula from 
Proposition \ref{propskelmult}, we have
\begin{align*}
 &(\varphi^{\b\g}_0,\varphi^{\b\g}_1,0,0,\ldots)\circ 
(\varphi^{\a\b}_0,\varphi^{\a\b}_1,0,0,\ldots)\\
&=
(\varphi^{\b\g}_0\circ\varphi^{\a\b}_0,\varphi^{\b\g}_1\circ\varphi^{\a\b}
_0(\varphi^{\a\b}_1),0,0,\ldots)\\
 &=(\varphi^{\a\g}_0,\varphi^{\a\g}_1,0,0,\ldots)=\i^1_1(\varphi^{\a\g})
\end{align*}
where defined. We let $\i^1_k(F)$ be the supermanifold $\M$ defined in 
Proposition \ref{proplocaldescsman} by the given change of charts. 

Morphisms $f\colon F\to F'$ of vector bundles have the local form 
$(f^{\a\b}_0,f^{\a\b}_1)$, where $f^{\a\b}_1$ is linear in the second component 
and define skeletons of the form $(f^{\a\b}_0,f^{\a\b}_1,0,0,\ldots)$ that 
satisfy 
Proposition 
\ref{proplocaldescsman}. In this way, we obtain a unique supersmooth morphism 
$\i^1_1(f)\colon 
\i^1_1(F)\to \i^1_1(F')$. By the same argument as above, this construction is 
functorial and, by uniqueness, the resulting functor is faithful.
\end{proof}

\begin{lemma}[{\cite[cf. Proposition 4.2.1, p.396]{Mol}}]\label{lemsmankproj}
 Let $k,n\in\N_0\cup\{\infty\}$ and $n\leq k$. The restriction of functors 
 $\Man^{\Grk{k}}$ to functors $\Man^{\Grk{n}}$ leads to functors
 \[
  \pi^k_n\colon\SMank{k}\to\SMank{n},\quad \pi^k_n(\M)=:\M^{(n)}.
  \nomenclature{$\pi^k_n$}{}
  \nomenclature{$\M^{(n)}$}{}
 \]
 On morphisms, we write $\pi^k_n(f)=:f^{(n)}$. These functors respect products 
and $\pi^k_m=\pi^k_n\circ\pi^n_m$ holds for all $m\leq n$.
 Identifying $\SMank{0}$ with $\Man$ and $\SMank{1}$ with $\VBun$ via 
Proposition \ref{propisman} and Proposition \ref{propvbunman}, we have 
$\pi^k_0\circ\i^0_k\cong\id_{\Man}$ and $\pi^k_1\circ \i^1_k\cong\id_{\VBun}$ 
if $k>0$.
\end{lemma}
\begin{proof}
 Let $\M,\Nc\in\SMank{k}$. It 
follows directly from the definition that $\L\to\M_\L$ for $\L\in\Grk{n}$ 
defines an $n$-supermanifold $\M^{(n)}$ with the obvious restricted atlas.
Likewise, for morphisms $f\colon\M\to\Nc$, one defines $f^{(n)}$ by 
$f^{(n)}_\L:=f_\L$ for $\L\in\Grk{n}$. This construction is clearly functorial, 
respects products and satisfies $\pi^k_m=\pi^k_n\circ\pi^n_m$ for all $m\leq 
n$. 
To see 
$\pi^k_0\circ\i^0_k\cong\id_{\Man}$ and $\pi^k_1\circ \i^1_k\cong\id_{\VBun}$, 
one simply checks that on the level of skeletons this composition does not 
change anything.
\end{proof}
One can understand the projections 
$\pi^k_n\colon\SMank{k}\to\SMank{n}$ and the embeddings 
$\i^0_k\colon \Man\to\SMank{k}$ and $\i^1_k\colon\VBun\to\SMank{k}$ completely 
in terms of skeletons.
The former simply cuts skeletons $(f_0,\ldots)$ 
down to $(f_0,\ldots, f_n)$. The latter two extend skeletons $(f_0)$, resp.\ 
$(f_0,f_1)$, to $(f_0,0,\ldots)$, resp.\ $(f_0,f_1,0,\ldots)$. Proposition 
\ref{propskelmult} ensures that the composition of two such skeletons is 
again of this form, which is why these embeddings are well-defined. 

A natural question is now whether 
two $k$-supermanifolds $\M^{(k)}$ and $\Nc^{(k)}$ such that 
$\M^{(n)}\cong\Nc^{(n)}$ holds
for $1<n<k$ are automatically isomorphic as well. In other words, whether a 
supersmooth isomorphism $f^{(n)}\colon\M^{(n)}\to\Nc^{(n)}$ can be lifted to an 
isomorphism $f^{(k)}\colon\M^{(k)}\to\Nc^{(k)}$.
This will be discussed in the following section on Batchelor's Theorem.
\begin{definition}
 We denote by 
$\po$\nomenclature{$\po$}{} the supermanifold modelled on 
$\{0\}\oplus\{0\}$ that consists for every $\L\in\Gr$ of a single point.
Let $k\in\N_0\cup\{\infty\}$. A 
\textit{point}\index{Supermanifold!point}\index{Point} of a 
$k$-supermanifold $\M$ is a morphism $x\colon\po^{(k)}\to\M$.
We also write $x_\L:=x_\L(\po^{(k)}_\L)$.
\end{definition}
\begin{lemma}
 Let $k\in\N_0\cup\{\infty\}$ and $\M\in\SMank{k}$. For every point 
 $x\colon\po^{(k)}\to\M$ and every $\L\in\Grk{k}$, we have
$x_\L=\M_{\eta_{\L}}(x_\R)$. Conversely, for every $x_\R\in\M_\R$ the 
assignment 
$x_\L:=\M_{\eta_\L}(x_\R)$ defines a point.
\end{lemma}
\begin{proof}
 For every $\L\in\Grk{k}$, we have 
 \[
  \M_{\eta_\L}(x_\R)=x_\L\circ 
\po^{(k)}_{\eta_\L}(\po^{(k)}_\R)=x_\L.
 \]
 Conversely, let $x_\R\in\M_\R$ be given and $x_\L:=\M_{\eta_\L}(x_\R)$.  
 Then $\varrho\circ\eta_\L=\eta_{\L'}$ and 
therefore $\M_\varrho(x_\L)=\M_{\eta_{\L'}}(x_\R)=x_{\L'}$ holds for 
$\varrho\in\Hom_{\Grk{k}}(\L,\L')$.
\end{proof}
Hence, the points of a supermanifold can be identified with the usual 
points of 
the base manifold. 
\subsubsection{Connection to the Sheaf Theoretic Approach}
The full subcategory of finite-dimensional supermanifolds in the categorical 
approach is equivalent to the category of supermanifolds in the sheaf theoretic 
approach. This was already discussed in \cite{Vor} and \cite{MolICTP} but a 
more thorough and general proof can be found in \cite{AllLau}. Let us briefly 
sketch the idea behind the equivalence.

Let $p,q\in\N_0$ and $\U\subseteq\ol{\R^{p|q}}$ be an open subfunctor.
In terms of skeletons, we have 
\[
\SC(\U,\ol{\R^{1|1}})=\Cs\Big(\U_\R,\bigoplus_{i=0}^q\Alt^i(\R^q;\R)\Big)
\cong\Cs(\U_\R,\R)\otimes\L_q.
\]
Therefore, for any supermanifold $\M$ modelled on $\R^{p|q}$, the sheaf 
\[
 U\mapsto \SC(\M|_{U},\ol{\R^{1|1}}),\quad U\subseteq\M_\R\ \text{open},
\]
is locally isomorphic to the sheaf $\Cs_{\R^p} \otimes\L_q$ as needed. One then 
checks that morphisms of supermanifolds lead to appropriate morphisms of these 
sheaves along the pullback.
\subsection{Generalizations}\label{subsectgenman}
Many of the generalizations for $k$-superdomains mentioned in 
\ref{subsectgendom} can be applied to supermanifolds without much difficulty 
such that the results in this section carry over.
Additional care is 
necessary only if the base ring is not a field because  
in that case locality of smooth maps is not guaranteed
(see \cite[2.4, 
p.21]{Bert}).
As already mentioned, one can consider non-Hausdorff supermanifolds by simply 
extending the category $\Man$ to non-Hausdorff manifolds. Analytic 
supermanifolds can be defined by demanding that the skeletons are analytic in 
an appropriate sense.

One should also note that many structural results do not rely on 
supersmoothness. In view 
of Proposition \ref{propfdecomp} and Lemma \ref{lemdnmulti}, one can define a 
subcategory of $\Man^{\Grk{k}}$ of functors locally isomorphic to some 
$\ol{E}^{(k)}$, 
$E\in\SVec$
where the changes of charts are simply natural transformations such that every 
component is smooth. Then an analog to Theorem \ref{thrmsmanmbun} still 
holds and one obtains a geometry combining commuting and anticommuting 
coordinates with less stringent symmetry conditions than for supermanifolds.
\subsection{Batchelor's Theorem}
The classical version of Batchelor's Theorem (see \cite{Batch}) states that any 
supermanifold, defined as a sheaf $(M,\O_M)$, is isomorphic to the 
supermanifold 
($M,\Gamma(\bigwedge F))$, where $\bigwedge F$ is the exterior bundle of a 
vector bundle $F$ that is determined by $\O_M$.
The isomorphism is not canonical because its construction involves a 
partition of unity.

Molotkov transfers this result to supermanifolds in our sense  
and generalizes it to infinite-dimensional supermanifolds $\M$ in \cite{Mol2}.
In his version, the vector bundle $\M_{\L_1}$ takes the role of $F$ in the 
classical version. Molotkov only considers Banach supermanifolds, but as we 
will 
see, his methods generalize to locally convex supermanifolds.
It appears that \cite{Mol2} is not well-known and since it is not readily 
available, we describe his arguments in detail below.
A closer look is also worthwhile because the techniques employed are close to 
the ones used in \cite{Batch} and might make it easier to translate between the 
sheaf theoretic and the categorical approach.
\begin{theorem}[{\cite[Corollary 4, p.279]{Mol2}}]\label{thrmbatch}
 \index{Batchelor's Theorem}
 Let $k\in\N\cup\{\infty\}$ and $\M\in\SMank{k}$ be such that $\M_\R$ admits  
smooth partitions of unity.
 If $\M'\in\SMank{k}$ is a $k$-supermanifold such that 
$\M^{(1)}$ and $\M'^{(1)}$ are isomorphic, then $\M$ is isomorphic to $\M'$.
In particular $\M\cong\i^1_k(\M^{(1)})$.
\end{theorem}
In other words, if one restricts the categories $\VBun$ and $\SMank{k}$ to 
the respective subcategories over finite-dimensional paracompact bases, the 
restricted functor $\i^1_k$ from Proposition \ref{propvbunman} becomes 
essentially surjective.
\begin{definition}
 Let $k\in\N$. We call $k$-supermanifolds of the form $\i^1_k(\M^{(1)})$, 
where
$\M^{(1)}$ is a vector bundle, \textit{supermanifolds of Batchelor 
type}\index{Supermanifold!of Batchelor type}. An isomorphism $f\colon 
\Nc\to\i^1_k(\M^{(1)})$ is called a \textit{Batchelor model of} 
$\Nc$\index{Batchelor model}.
We say an atlas $\A:=\{\varphi^\a\colon\a\in A\}$ of a supermanifold is of 
\textit{Batchelor type}\index{Supermanifold!atlas of Batchelor type}
if all changes of charts have the form 
$\varphi^{\a\b}=(\varphi^{\a\b}_0,\varphi^{\a\b}_1,0,\ldots)$. 
\end{definition}
\begin{remark}
 It follows from Proposition \ref{proplocaldescsman} that a supermanifold is of 
Batchelor type if and only if it has an 
atlas of Batchelor type.
For a supermanifold of Batchelor type, the union of two atlases of Batchelor 
type 
is again of Batchelor type because the atlases define the same vector bundle.
This does not need to be the case for arbitrary supermanifolds, which implies 
that there is no 
canonical choice of a Batchelor model in general.
\end{remark}
One can reformulate this result as follows. In the situation of the theorem, 
any 
isomorphism 
$f^{(n)}\colon\M^{(n)}\to\M'^{(n)}$ can be lifted to an isomorphism
$f^{(n+1)}\colon\M^{(n+1)}\to\M'^{(n+1)}$ for $1\leq n<k$ (see \cite[Theorem 
1(a), p.273]{Mol2}).
It is not difficult to see that one may assume $\M'=\M$, which we will do in 
the 
sequel to simplify our explanations (compare 
\cite[Proposition 2, p.277]{Mol2}).

Let us introduce some notation for this section. For $\M\in\SMank{k}$ and 
$k\in\N_0\cup\{\infty\}$ consider the group $\Aut_{\id_\R}(\M)$ of 
automorphisms 
$f\colon\M\to\M$ such that $f_\R=\id_{\M_\R}$. 
We denote by $\Shf(\M)$ the sheaf of groups over $\M_\R$ defined by
$U\mapsto \Aut_{\id_\R}(\M|_{U})$ for every open $U\subseteq\M_\R$. The 
restriction morphisms are given in the obvious way by Lemma/Definition 
\ref{lemdefimrestr}(c). The restrictions are morphisms of groups because we 
only consider automorphisms over the identity on the base manifold.
The functor $\pi^n_m\colon\SMank{n}\to\SMank{m}$ from 
Lemma \ref{lemsmankproj} leads to morphisms 
$\varphi^n_m\colon\Shf(\M^{(n)})\to\Shf(\M^{(m)})$ of sheaves of groups for 
$m\leq n<k+1$. 
Locally, 
$(\varphi^n_m)_U\colon\Aut_{\id_\R}(\M^{(n)}|_{U})\to\Aut_{\id_\R}(\M^{(m)}|_{U}
)$ 
just maps skeletons $(\id,f_1,\ldots,f_n)$ to $(\id,f_1,\ldots,f_m)$.
We define
\[
 \Shf^n_m(\M):=\ker\varphi^n_m.
\]
The elements of $\Shf^n_m(\M)$ are exactly those which locally have the 
form $(\id,c_{\id},0,\ldots,0,f_{m+1},\ldots,f_{n})$. In particular, we get a 
short 
exact sequence of sheaves of groups
\[
1\to\Shf^{n+1}_{n}(\M)\hookrightarrow\Shf^{n+1}_0(\M)\to\Shf^n_0(\M
)\to 1
\]
(see \cite[Theorem 1, p.273f.]{Mol2}). Note that 
$\Shf^{n+1}_0(\M)=\Shf(\M^{(n+1)})$. We sum up the most relevant results from 
\cite{Mol2} about the structure of $\Shf^{n+1}_{n}(\M)$ in the next lemma and 
give a sketch of the proof.
\begin{lemma}[{compare \cite[Theorem 1(d), p.274]{Mol2}}]\label{lemshfn+1}
 Let $k\in\N\cup\{\infty\}$, $n<k$ and $\M\in\SMank{k}$. 
  Then $\Shf^{n+1}_{n}(\M)$ is a sheaf of abelian groups and a 
$\mathcal{C}^\infty_{\M_\R}$-module. If $\M$ is a Banach supermanifold, then 
there exist canonical isomorphisms of $\mathcal{C}^\infty_{\M_\R}$-modules
\[
 \Shf^{n+1}_n(\M)\cong\Gamma(\Alt^{n+1}(\M_{\L_1};T\M_\R))\quad\text{if}\ n+1\ 
\text{is even and} 
\]
\[
 \Shf^{n+1}_n(\M)\cong\Gamma(\Alt^{n+1}(\M_{\L_1};\M_{\L_1}))\quad\text{if}\ 
n+1\ 
\text{is odd}. 
\]
\end{lemma}
\begin{proof}
 Let $\{\varphi^\a\colon\U^\a\to\M^{(n+1)}\colon\a\in A\}$ be an atlas of 
$\M^{(n+1)}$ and $U\subseteq\M_\R$ be open. For $f\in\Shf^{n+1}_n(\M)_U$,
we set $f^\a:=(\varphi^\a)^{-1}\circ f\circ \varphi^\a$, where we may assume 
after restriction that $\varphi^\a$ is a chart of $\M^{(n+1)}|_U$.
Locally, $f$ has the form 
$f^\a=(\id_{\U^\a_\R},c_{\id},0,\ldots,0,f^\a_{n+1})$. 
For $g\in\Shf^{n+1}_n(\M)_U$, we use Proposition \ref{propskelmult} to 
calculate
\[
 (f^\a\circ g^\a)=(f\circ g)^\a=(\id_{\U^\a_\R},c_{\id},0,\ldots,0, 
f^\a_{n+1}+g^\a_{n+1}).
\]
Thus, $\Shf^{n+1}_n(\M)$ is a sheaf of abelian groups.
Let $\varphi^{\b\a}$ be a change of 
charts of $\M^{(n+1)}|_U$. 
We again use Proposition \ref{propskelmult} to get
\[
 (\varphi^{\b\a})^{-1}\circ 
f^\a\circ\varphi^{\b\a}=(\id_{\U^{\b\a}_\R},c_{\id},0,\ldots,0,f^\b_{n+1}),
\]
where 
\[
 f^\b_{n+1}(\varphi^{\a\b}_0(x))(\bl)=
d\varphi^{\a\b}_0(x)\big(f^\a_{k+1}(x)\big(\varphi^{\b\a}_1(x)(\bl)
, \ldots, \varphi^{\b\a}_1(x)(\bl)\big)\big)
\]
for $n+1$ even, $x\in\U^{\b\a}_\R$ and
\[
 f^\b_{n+1}(\varphi^{\a\b}_0(x))(\bl)=
\varphi^{\a\b}_1(x)\big(f^\a_{k+1}(x)\big(\varphi^{\b\a}_1(x)(\bl),
\ldots, \varphi^{\b\a}_1(x)(\bl)\big)\big)
\]
for $n+1$ odd, $x\in\U^{\b\a}_\R$. If $\M_{\L_1}$ is a Banach vector bundle, 
this describes the change of charts for a section   
$f_{n+1}\colon\M_\R|_U\to\Alt^{n+1}(\M_{\L_1}|_U;T\M_\R|_U))$, resp.\ 
$f_{n+1}\colon\M_\R|_U\to\Alt^{n+1}(\M_{\L_1}|_U;\M_{\L_1}|_U)$.
It is easy to see from these formulas that for a smooth
map $h\colon\M_\R|_U\to\R$  with the local description 
$h^\a:=h\circ\varphi^\a_0$, the multiplication $h\cdot f$, defined by
\[
 (h\cdot f)^\a=(\id_{\U^\a_\R},c_{\id},0,\ldots,0,h^\a\cdot f^\a_{n+1}),
\]
leads to a $\mathcal{C}^\infty_{\M_\R}$-module structure on $\Shf^{n+1}_n(\M)$.
This structure 
corresponds to the $\mathcal{C}^\infty_{\M_\R}$-module structure of the 
sheaves of sections defined above in the Banach case.
\end{proof}
We will now return to finding a lift for an isomorphism 
$f^{(n)}\colon\M^{(n)}\to\M^{(n)}$ to an isomorphism 
$f^{(n+1)}\colon\M^{(n+1)}\to\M^{(n+1)}$. Let $\M\in\SMank{k}$ be 
modelled on $E\in\SVec$ with atlas $\{\varphi^\a\colon\U^\a\to\M\colon\a\in 
A\}$ and 
$n<k$.
Locally, we have isomorphisms $f^{(n),\a}\colon\U^{\a,(n)}\to\U^{\a,(n)}$ of 
the form $(\id_{\U^{\a}_\R},f^{(n),\a}_1,\ldots, f^{(n),\a}_n)$. By Lemma 
\ref{lemskelinv}, these can be lifted to isomorphisms 
$\tilde{f}^{\a,(n+1)}\colon\U^{\a,(n+1)}\to\U^{\a,(n+1)}$
in $\SDomk{n+1}$
of the form 
\[(
\id_{\U^\a_\R},f^{(n),\a}_1,\ldots, 
f^{(n),\a}_n,\tilde{f}^{(n+1),\a}_{n+1}),
\] 
with
$\tilde{f}^{(n+1),\a}_{n+1}\colon\U^\a_\R\to\Alt^{n+1}(E_1;E_{\ol{n+1}})$ an 
arbitrary map that is smooth in the sense of skeletons.

The morphisms $\tilde{f}^{\b,(n+1)}$ and 
$\varphi^{\a\b,(n+1)}\circ\tilde{f}^{\a,(n+1)}\circ\varphi^{\b\a,(n+1)}$ differ 
on $\U^{\b\a,(n+1)}$ only in the $(n+1)$-th components of their skeletons 
because higher components do not affect the composition of any 
lower components. The difference is given by a unique element 
$h^{\b\a}\in\Shf^{n+1}_n(\U^{\b\a})$ such that
\[
\tilde{f}^{\b,(n+1)}=\varphi^{\a\b,(n+1)}\circ\tilde{f}^{\a,(n+1)}\circ\varphi^{
\b\a,(n+1)}\circ h^{\b\a}
\]
and one checks that these $h^{\b\a}$ define a cocycle in $\Shf^{n+1}_n(\M)$ via
$\tilde{h}^{\b\a}:=\varphi^{\b,(n+1)}\circ 
h^{\b\a}\circ(\varphi^{\b,(n+1)})^{-1}$ on 
$\varphi^{\b,(n+1)}(\U^{\b\a,(n+1)})$.
This cocycle describes the obstacle to lifting $f^{(n)}$ to $f^{(n+1)}$ (see 
\cite[Theorem 3, p.277]{Mol2}). If $\M_\R$ admits smooth partitions of unity, 
then $\Shf^{n+1}_n(\M)$ is a fine sheaf by Lemma \ref{lemshfn+1}
and thus acyclic.
Therefore, the cocycle constructed above vanishes and a lift exists.
\begin{remark}
 Mirroring the proof of the fact that for fine sheaves the higher \v{C}ech 
cohomologies vanish,
  one can directly construct the lift via a partition of unity.
In the situation above, we assume that $(\varphi^\a_\R(\U^\a_\R))_{\a\in 
A}$ is a locally finite cover of $\M_\R$ and that $(\rho_\a)_{\a\in A}$ is a 
partition 
of unity subordinate to this cover. With the module structure from Lemma 
\ref{lemshfn+1}, we define
\[
 f^{(n+1),\a}_{n+1}:=\Big((\varphi^{\a})^{-1}\circ\Big(
      \sum_{\b\in A} \rho_\b\cdot \tilde{h}^{\b\a}
 \Big)\circ\varphi^{\a}\Big)_{n+1},
\]
where $\rho_\b\cdot \tilde{h}^{\b\a}$ is 
continued to $\varphi^{\a,(n+1)}(\U^{\a,(n+1)})$ 
by zero. It is elementary to check that the change of charts is well-defined 
for the resulting local descriptions of $f^{(n+1)}$.
\end{remark}

\subsection{Super Vector Bundles}
In analogy to our definition of supermanifolds, we give a definition of super 
vector bundles as supermanifolds with a particular kind of bundle 
atlas. In this, we follow \cite[Definition 
5.1, p.29]{Mol}. See also \cite{SaWo}.
While a bit cumbersome, it will be useful to describe the change of bundle 
charts and the local form of bundle morphisms in terms of skeletons.
\begin{definition}[{compare \cite[Subsection 1.3, p.5]{Mol}}]\label{ufamily}
 Let $E,F,H\in\SVec$, $k\in\N_0\cup\{\infty\}$ and $\U\subseteq \ol{H}^{(k)}$ 
be an open subfunctor. 
A 
supersmooth morphism $f\colon\U\times\ol{E}^{(k)}\to\ol{F}^{(k)}$ is called an 
\textit{ 
$\U$-family of $\ol{\R}$-linear morphisms}\index{$\U$-family of 
$\ol{\R}$-linear morphisms} if for every $\Lambda\in\Grk{k}$ and 
every $u\in\U_\L$, the map
\[
 f_\L(u,\bl)\colon\ol{E}^{(k)}_{\L}\to\ol{F}^{(k)}_{\L}
\]
is $\Lz$-linear. 
\end{definition}
\begin{lemma}\label{lemufamily}
 In the situation of Definition \ref{ufamily}, let 
$f\colon\U\times\ol{E}^{(k)}\to\ol{F}^{(k)}$ be a supersmooth morphism. Then 
$f$ is an 
$\U$-family of $\ol{\R}$-linear 
morphisms if and only if for all $\L\in\Grk{k}$, we have
\[
 df_\Lambda\big((u,0)\big)\big((0,v)\big)=f_\Lambda(u,v)\ \text{for}\ 
u\in\U_\L,\ v\in\ol{E}^{(k)}_\L.
\]
\end{lemma}
\begin{proof}
 If the equality holds, then $f$ is a $\U$-family of $\ol{\R}$-linear 
morphisms because
 the derivative $df_\L$ is $\Lz$-linear at every 
$u\in\U_\L$. The converse is true because any 
$\Lz$-linear map is in particular $\R$-linear and thus the derivative 
of such a map is the map itself.
\end{proof}
\begin{lemma}\label{lemufamily2}
 In the situation of Definition \ref{ufamily}, let 
$f\colon\U\times\ol{E}^{(k)}\to\ol{F}^{(k)}$ be a supersmooth morphism with the 
skeleton
$(f_n)_n$.
We set $U:=\U_\R$.
Then $f$ is a $\U$-family of $\ol{\R}$-linear 
morphisms if and only if
every $f_n$ has the form
\begin{align*}
 f_n=&f_n(\pr_U,\pr_{E_0})((\pr_1,0),\ldots,(\pr_1,0))\\
&+n\cdot \alt{n} f_n(\pr_U,0)((0,\pr_2),(\pr_1,0),\ldots,(\pr_1,0)),
\end{align*}
with $f_n(\pr_U,\pr_{E_0})((\pr_1,0),\ldots,(\pr_1,0))$ linear in the second 
component and
where $\pr_U\colon U\times E_0\to U$, $\pr_{E_0}\colon U\times E_0\to E_0$, 
$\pr_1\colon H_1\times E_1\to H_1$ and $\pr_2\colon H_1\times E_1\to E_1$ are 
the respective projections.
\end{lemma}
\begin{proof}
Let $f\colon\U\times\ol{E}^{(k)}\to\ol{F}^{(k)}$ be an $\U$-family of 
$\ol{\R}$-linear 
morphisms. Choosing $u:=x+\sum_{i=1}^k \oddgen_iy_i$, $x\in U$ and $y_i\in 
H_1$, Proposition 
\ref{propskel} implies that $f_n(x,\bl)((y_1,\bl),\ldots,(y_n,\bl))$ must 
be linear in $E_0\oplus E_1\oplus\cdots\oplus E_1$. We use the multilinearity 
of $f_n(x,v)(\bl)$ to calculate
\begin{align*}
&f_n(x,v)((y_1,w_1),\ldots,(y_n,w_n))=\\
&f_n(x,v)((y_1,0),\ldots,(y_n,0))+\sum_{i=1}^n f_n(x,v)((y_1,0),
\ldots,(0,w_i),\ldots,(y_n,0))\\
 &=\big(f_n(x,v)((\pr_1,0),\ldots,(\pr_1,0))+\\
 &\qquad n\cdot 
\alt{n} 
f_n(x,0)((0,\pr_2),(\pr_1,0),\ldots,(\pr_1,0))\big)((y_1,w_1),\ldots,(y_n,
w_n))
\end{align*}
for $v\in E_0$ and $w_i\in E_1$.
The second equality follows because for $v'\in E_0$, we have
\begin{align*}
f_n(x,v+v')&((y_1,0),
\ldots,(0,w_i),\ldots,(y_n,0))=\\
 &f_n(x,v)((y_1,0),
\ldots,(0,w_i),\ldots,(y_n,0))\\
&+
f_n(x,v')((y_1,0),
\ldots,(0,0),\ldots,(y_n,0)).
\end{align*}

Conversely, let $(f_n)_n$ have the 
aforementioned form. Let $\L\in\Grk{k}$, $(x,y)\in U\times E_0$ and 
$(x_i,y_i)\in (H_i\oplus E_i)\otimes\L^{\nil}_{\ol{i}}$, $i\in\{0,1\}$. 
To simplify our notation, we consider 
$\ol{H}^{(k)}_\L\subseteq\ol{H\oplus E}^{(k)}_\L$ and 
$\ol{E}^{(k)}_\L\subseteq\ol{H\oplus E}^{(k)}_\L$ in the obvious way. 
One sees
\begin{align*}
 d^mf_l(x,y)((x_0,y_0),\ldots,&(x_0,y_0),(x_1,y_1),\ldots,(x_1,y_1))=\\
 &d^m 
f_l(x,y)(x_0,\ldots,x_0,x_1,\ldots,x_1)\\
&+m\cdot 
d^{m-1}f_l(x,y_0)(x_0,\ldots,x_0, x_1,\ldots,x_1)\\
 &+l\cdot d^mf_l(x)(x_0,\ldots,x_0,y_1,x_1\ldots x_1),
\end{align*}
where the last two summands are understood to be zero for $m=0$ and $l=0$, 
respectively.
For  
$u=x+x_0+x_1$ and $v=y+y_0+y_1$, we use Remark \ref{remdf} to get
\begin{align}\label{formparitychange}
df_\L(u)(v)=\sum_{m,l=0}^\infty\frac{1}{m!l!}\cdot&\Big(d^mf_l(x,
y)(x_0,\ldots,x_0,x_1,\ldots,x_1)\nonumber\\
&+m\cdot d^{m-1}f_l(x,y_0)(x_0,\ldots,x_0,x_1,\ldots,x_1)\\
&+l\cdot d^m f_l(x)(x_0,\ldots,x_0,y_1,x_1\ldots,x_1)\Big).\nonumber
\end{align}
Comparing the terms, Proposition 
\ref{propskel} implies that $df_\L(u)(v)=f_\L(u,v)$ and the result 
follows from Lemma \ref{lemufamily}.
\end{proof}
\begin{definition}\label{defsvectorbundle}
 Let $E,F\in\SVec$, $k\in\N_0\cup\{\infty\}$ and let $\E,\M\in\SMank{k}$
 be such that $\E$ is modelled on $E\oplus F$ and $\M$ is modelled on $E$
 together with a supersmooth morphism $\pi\colon\E\to\M$ such that 
$\pi_\L\colon\E_\L\to\M_\L$ is a vector bundle with fiber 
$\ol{F}^{(k)}_\L$ for all $\L\in\Grk{k}$.
 A \textit{bundle atlas of $\E$}\index{Super vector bundle!bundle atlas} is
 an atlas 
$\A:=\{\Psi^\alpha\colon\U^\alpha\times\ol{F}^{(k)}\to\E\colon
\U^\a\subseteq\ol{E}^{(k)},\alpha\in A\}$ such 
that $\{\Psi^\a_\L\colon\a\in A\}$ is a bundle atlas of $\E_\L$ and
the change of two charts $\Psi^\a$ and $\Psi^\b$ has the form
$\Psi^{\alpha\beta}\colon\U^{\alpha\beta}\times\ol{F}^{(k)}\to\U^{\beta\alpha
}
\times\ol{F}^{(k)}$
with $\Psi^{\alpha\beta}=(\phi^{\alpha\beta},\psi^{\alpha\beta})$, where
\begin{enumerate}
\item[(i)] $\phi^{\alpha\beta}\colon\U^{\alpha\beta}\to 
\U^{\beta\alpha}$ and
\item[(ii)] $\psi^{\alpha\beta}\colon\U^{\alpha\beta}\times\ol{F}^{(k)}\to 
\ol{F}^{(k)}$ is an $\U^{\alpha\beta}$-family of $\ol{\R}$-linear maps.
\end{enumerate}
The elements of $\A$ are called \textit{bundle charts}\index{Super vector 
bundle!bundle chart}. Two bundle atlases are \textit{equivalent} if their union 
is again 
a bundle atlas. We call $\pi\colon\E\to\M$ together with an equivalence class 
of bundle atlases a \textit{$k$-super vector bundle over the 
base $\M$ with typical fiber $F$}.\index{Super vector bundle}
\index{Super vector bundle!base}\index{Super vector bundle!typical fiber}
The morphism $\pi$ is called the \textit{projection to the base}.

Let $\E'$ be another $k$-super vector bundle with typical fiber 
$F'\in\SVec$ and base $\Nc$. A supersmooth morphism 
$f\colon\E\to\E'$ is a \textit{morphism of super vector bundles}\index{Super 
vector bundles!morphism of}, if in bundle charts $\Psi^\a$ of $\E$ and 
$\Psi'^{\a'}$ of $\E'$, it 
has the form $(h^{\a\a'},g^{\a\a'})$, where
\begin{enumerate}
\item[(i)] $h^{\alpha\a'}\colon\U^{\alpha\a'}\to 
\U^{\a'}$, $\U^{\a\a'}\subseteq\U^\a$ and
\item[(ii)] $g^{\alpha\a'}\colon\U^{\alpha\a'}\times\ol{F}^{(k)}\to 
\ol{F'}^{(k)}$ is an $\U^{\alpha\a'}$-family of $\ol{\R}$-linear maps.
\end{enumerate}
Clearly, the $h^{\a\a'}$ define a supersmooth morphism $h\colon\M\to\Nc$ such 
that
$\pi_{\Nc}\circ f=h\circ \pi_\M$. We say that \textit{$f$ is a morphism over 
$h$}. The $k$-super vector bundles and their 
morphisms form a category, which we denote by 
$\SVBunk{k}$\nomenclature{$\SVBunk{k}$}{}, resp.\ 
$\SVBun$
\nomenclature{$\SVBun$}{} 
if $k=\infty$.
\end{definition}
By this definition, a $k$-super vector bundle can be 
seen as a functor $\Grk{k}\to\VBun$. This point of view is taken by Molotkov in 
\cite{Mol}.
\begin{remark}\label{remsvbunlocal}
 It follows from Proposition \ref{proplocaldescsman} that, if one has a 
collection of change of charts that satisfy the conditions of Definition 
\ref{defsvectorbundle}, then one gets a (up to unique isomorphism) unique super 
vector bundle.
In the notation of the definition, the $\phi^{\a\b}$ then define the base 
supermanifold $\M$ and the bundle projection is locally given by
\[
 (\phi^\alpha)^{-1}\circ\pi\circ\Psi^\alpha:=\pr_{\U^\alpha}\colon
 \U^\alpha  \times\ol{F}^{(k)}\to\U^\alpha.
\]
In the same way, morphisms of super vector bundles are determined 
by their local description.
\end{remark}
\begin{lemma/definition}\label{lemdeffiber}
 Let $k\in\N_0\cup\{\infty\}$, let $\E$ be a $k$-super vector bundle 
with typical fiber $F\in\SVec$ over $\M$ modelled on $E\in\SVec$ and let 
$x\colon\po^{(k)}\to\M$ be a point of $\M$. We define $\E_x$, the \textit{fiber 
of $\E$ at $x$}\index{Super vector bundle!fiber}, by setting 
$(\E_x)_\L:=(\pi_\M)^{-1}(\{x_\L\})$ for every $\L\in\Grk{k}$. Then $\E_x$ is a 
sub-supermanifold of $\E$ and $\E_x$ is, in a canonical way, an $\ol{\R}$-module
such that $\E_x\cong\ol{E}^{(k)}$.
\end{lemma/definition}
\begin{proof}
 Let $\{\Psi^\a\colon\U^\a\times\ol{F}^{(k)}\to\E\colon \a\in A\}$ be a bundle 
atlas of $\E$ with corresponding atlas $\{\phi^\a\colon\U^\a\to\M\colon \a\in 
A\}$ of $\M$. Let $\phi^\a\colon\U^\a\to\M$ be such that 
$x_\R\in\phi^\a_\R(\U^\a_\R)$. We may assume that $0\in\U^\a_\R$ and that
$\phi^\a_\R(0)=x_\R$, because the translation defined by 
$\ol{E}^{(k)}_\L\to\ol{E}^{(k)}_\L,\ v\mapsto v-(\varphi^\a_\R)^{-1}(x_\R)$ is 
clearly an isomorphism in $\SMank{k}$. Let $\L\in\Grk{k}$. We have
$y_\L\in(\E_x)_\L$ if and only if $(\pi_\M)_\L(y_\L)=x_\L$ holds and thus if 
and only if 
$y_\L\in\Psi^\a_\L(\ol{\{0\}}_\L^{(k)}\times\ol{F}_\L^{(k)})$ holds. Therefore, 
$\E_x$ 
is a sub-supermanifold of $\E$. We define an $\ol{\R}$-module structure on 
$\E_x$ via the isomorphism
\[
 \Psi^\a\circ(0,\id_{\ol{F}^{(k)}})\colon\ol{F}^{(k)}\to\E_x.
\]
The $\ol{\R}$-module structure on $\E_x$ does not depend on $\Psi^\a$ because a 
change of bundle charts leads to an isomorphism of $\ol{\R}$-modules in the 
second component.
\end{proof}
\begin{lemma}
 Let $k\in\N_0\cup\{\infty\}$ and let $\E$ and $\E'$ be $k$-super vector 
bundles over $\M$ and $\Nc$. If $f\colon\E\to\E'$ is a morphism of $k$-super 
vector bundles over $g\colon\M\to\Nc$, then for every point $x$ of $\M$ 
the morphism
\[
 f|_{\E_x}\colon\E_x\to\E'_{g\circ x}
\]
is a well-defined morphism of $\ol{\R}$-modules (and in particular supersmooth).
\end{lemma}
\begin{proof}
  Let $\L\in\Grk{k}$ and $y_\L\in(\E_x)_\L$. We have
 \[
 \xymatrix{ y_\L \ar[r]^{f_\L} \ar[d]_{(\pi_\M)_\L} & *+[r]{(\E')_\L} 
\ar[d]^{(\pi_{\Nc})_\L} \\
               x_\L \ar[r]_{g_\L} & *+[r]{g_\L(x_\L)} }
\]
and $g_\L(x_\L)=(g\circ x)_\L$ implies that the morphism is 
well-defined. In charts the second component of $f$ is $\ol{\R}$-linear. Thus, 
$f|_{\E_x}$ is a morphism of $\ol{\R}$-modules.
\end{proof}
\begin{lemma}\label{lemsvbunfunct}
 The functors $\i^0_k$, $\i^1_k$ and 
$\pi^k_n$ from Proposition \ref{propisman}, Proposition \ref{propvbunman}
and Lemma \ref{lemsmankproj} extend to functors
\begin{align*}
 &\i^0_k\colon\SVBunk{0}\to\SVBunk{k}\ \text{for}\ k\in\N_0\cup\{\infty\},\\
 &\i^1_k\colon\SVBunk{1}\to\SVBunk{k}\ \text{for}\ k\in\N\cup\{\infty\}\
 \text{and}\\
 &\pi^k_n\colon\SVBunk{k}\to\SVBunk{n}\ \text{for}\ k\in\N_0\cup\{\infty\},\  
0\leq n\leq k.
\end{align*}
With these functors, we have $\pi^k_0\circ\i^0_k\cong\id_{\SVBunk{0}}$ and
$\pi^k_1\circ\i^1_k\cong\id_{\SVBunk{1}}$.
\end{lemma} 
\begin{proof}
Let us consider $\i^0_k$ and $\i^1_k$ as functors $\SMank{0}\to\SMank{k}$ and 
$\SMank{1}\to\SMank{k}$.
 Let $k\in\N_0\cup\{\infty\}$, $E\in\SVec$, $\U\subseteq\ol{E}^{(k)}$.
 We see from the concrete description in Lemma \ref{lemufamily2} that every 
 $\U$-family of $\ol{\R}$-linear morphisms is mapped to an $\U$-family of 
$\ol{\R}$-linear morphisms under the original functors. From this, the result 
follows.
\end{proof}
Note also that for 
any super vector bundle $\E$ with base $\M$ and typical fiber $F$, the 
inverse limit 
$\varprojlim\E$ is in a natural way a vector bundle over $\varprojlim\M$ with 
typical fiber $\varprojlim\ol{F}$.
\subsubsection{The Change of Parity Functor}
The space of sections of a super vector bundle can be turned into a vector 
space, as we will see below. However, in a sense this describes only the 
even sections. To incorporate odd sections and obtain a super vector space 
of sections, we need the so called change of parity functor.
On super vector spaces this functor simply swaps the even and odd parts. Doing 
this fiberwise, one gets the change of parity functor for super vector 
bundles.  
As before, it will be useful to express this functor in terms of skeletons.
\begin{definition}\label{defcop}
 Let $E,F\in\SVec$ and $f\colon E\to F$ be a morphism. We define 
a functor $\Pi\colon\SVec\to\SVec$ by setting 
$\Pi(E)_i:=E_{\ol{i+1}}$ and $\Pi(f)_i:=f_{\ol{i+1}}$ for $i\in\{0,1\}$ 
\nomenclature{$\Pi$}{}. 
Now, let 
$\L=\L_n\in\Gr$, $n\in\N$ and $g\colon \ol{E}_\L\to\ol{F}_\L$ be $\R$-linear 
such 
that there exist linear maps
\[
 g_{(0)}\colon E_0\to\ol{F}_\L\quad\text{and}\quad g_{(1)}\colon 
E_1\to\ol{\Pi(F)}_\L
\]
with $g(\oddgen_Iv_I)=\oddgen_Ig_{(i)}(v_I)$ for $I\in\Po^n_i$, $v_I\in E_i$, 
$i\in\{0,1\}$.
 We call such a map $g$ \textit{parity changeable}\index{Parity changeable}.
We define a parity changeable map
\[
\ol{\Pi}_\L(g)\colon 
\ol{\Pi(E)}_\L\to\ol{\Pi(F)}_\L
\] 
by setting
 $\ol{\Pi}_\L(g)_{(i)}:=g_{(\ol{i+1})}$ for $i\in\{0,1\}$.
\end{definition}
Note that in the above situation $g$ is automatically 
$\Lz$-linear and we have
$\ol{\Pi}_\L(\ol{\Pi}_\L(g))=g$. What is more, with $f_{(0)}:=f_0$ and 
$f_{(1)}:=f_1$, we see that $\ol{f}_\L$ is parity changeable and it follows
$\ol{\Pi}_\L(\ol{f}_\L)=\ol{\Pi(f)}_\L$.
\begin{lemma}\label{lemparchangefunc}
 Let $E,F,H\in\SVec$, $\L=\L_n\in\Gr$ with $n\in\N$ and let $f\colon E_\L\to 
F_\L$, $g\colon F_\L\to H_\L$ be parity changeable. Then $g\circ f$ is also 
parity changeable and we have
\[
 \ol{\Pi}_\L(g\circ f)=\ol{\Pi}_\L(g)\circ\ol{\Pi}_\L(f).
\]
\end{lemma}
\begin{proof}
 Let $f_{(0)}, f_{(1)}$ and $g_{(0)}, g_{(1)}$ be as in Definition 
\ref{defcop}. For $I\in \Po^n_i$, $v\in E_i$, $i\in\{0,1\}$ let
\[
 f(\oddgen_Iv)=\oddgen_If_{(i)}(v)=\oddgen_I\sum_{J\in 
\Po^{n}}\oddgen_Jw_J,
\]
where $w_J\in F_{\ol{|I|+|J|}}$.
It follows that
\[
 (g\circ f)(\oddgen_I v)=\sum_{J\in 
\Po^{n}}\oddgen_I\oddgen_J g_{(\ol{|I|+|J|})}(w_J).
\]
This implies that $g\circ f$ is parity changeable with $(g\circ f)_{(0)}=g\circ 
f_{(0)}$ and
$(g\circ f)_{(1)}=\ol{\Pi}_\L(g)\circ f_{(1)}$.
Thus, $\ol{\Pi}_\L(g\circ f)_{(0)}=\ol{\Pi}_\L(g)\circ f_{(1)}$ and
$\ol{\Pi}_\L(g\circ f)_{(1)}=g\circ f_{(0)}$.
Applying this to $\ol{\Pi}_\L(g)\circ \ol{\Pi}_\L(f)$, we get
\[
 (\ol{\Pi}_\L(g)\circ 
\ol{\Pi}_\L(f))_{(0)}=\ol{\Pi}_\L(g)\circ\ol{\Pi}_\L(f)_{(0)}=\ol{\Pi}
_\L(g)\circ f_{(1)}
\]
and
\[
 (\ol{\Pi}_\L(g)\circ \ol{\Pi}_\L(f))_{(1)}=\ol{\Pi}_\L(\ol{\Pi}_\L(g))\circ 
\ol{\Pi}_\L(f)_{(1)}=g\circ f_{(0)}
\]
and therefore
\begin{align*}
 &\ol{\Pi}_\L(g\circ f)=\ol{\Pi}_\L(g)\circ\ol{\Pi}_\L(f).\qedhere
\end{align*}
\end{proof}
\begin{lemma}\label{lemparchange}
 Let $E,F,H\in\SVec$, $k\in\N\cup\{\infty\}$, $\U\subseteq \ol{H}^{(k)}$ be an 
open subfunctor and let
$f\colon\U\times\ol{E}^{(k)}\to\ol{F}^{(k)}$ be an
$\U$-family of $\ol{\R}$-linear morphisms. For $n\in\N$, $\L=\L_n\in\Grk{k}$  
and $u\in\U_\L$, the map $f_\L(u,\bl)\colon\ol{E}_\L^{(k)}\to\ol{F}_\L^{(k)}$ 
is 
parity changeable.
Defining
 $(\ol{\Pi}(f))_\L(u,\bl):=\ol{\Pi}_\L(f_\L(u,\bl))$
leads to an $\U$-family of $\ol{\R}$-linear morphisms
\[
\ol{\Pi}(f)\colon\U\times\ol{\Pi(E)}^{(k)}\to\ol{\Pi(F)}^{(k)}.
\]
The skeleton of $\ol{\Pi}(f)$ has the components
\begin{align*}
 \tilde{f}_0&=f_1(\pr_U)(\pr_{\Pi(E)_0})
 \quad\text{and}\\
 \tilde{f}_l&=f_{l+1}(\pr_U)(\pr_{\Pi(E)_0},\pr_1,\ldots,\pr_1)+l\cdot 
\alt{l}f_{l-1}(\pr_U,\pr_2)(\pr_1,\ldots,\pr_1)
\end{align*}
for $l>0$, where we consider
\begin{align*}
 \alt{l}f_{l-1}(\pr_U,\pr_2)(\pr_1,\ldots,\pr_1)\colon 
U\times\Pi(E)_0\to\Alt^l(H_1\oplus\Pi(E)_1;\Pi(F)_{\ol{k}})
\end{align*}
and
\begin{align*}
 f_{l+1}(\pr_U)&(\pr_{\Pi(E)_0},\pr_1,\ldots,\pr_1)\colon\\ 
&U\times\Pi(E)_0\to\Alt^l(H_1\oplus\Pi(E)_1;\Pi(F)_{\ol{k}})
\end{align*}
with the projections $\pr_U\colon U\times \Pi(E)_0\to U$, 
$\pr_{\Pi(E)_0}\colon U\times\Pi(E)_0\to\Pi(E)_0$, 
$\pr_1\colon H_1\times \Pi(E)_1\to H_1$ and 
$\pr_2\colon H_1\times\Pi(E)_1\to \Pi(E)_1$.
\end{lemma}
\begin{proof}
 Let $U:=\U_\R$. We set 
$\ol{\Pi}(f)_\R:=\ol{\Pi}(f)_{\L_1}|_{U\times 
E_1}\colon\U_\R\times\ol{\Pi(E)}^{(k)}_\R\to\ol{\Pi(F)}^{(k)}_\R$ so that 
$\ol{\Pi}(f)_\L$ is defined for all $\L\in\Grk{k}$. 
To simplify our notation, we consider 
$\ol{H}^{(k)}_\L\subseteq\ol{H\oplus E}^{(k)}_\L$ and 
$\ol{E}^{(k)}_\L\subseteq\ol{H\oplus E}^{(k)}_\L$ in the obvious way.
Let $x\in U$, $x_0\in\ol{H_0}^{(k)}_{\L^\nil}$, $x_1\in\ol{H_1}^{(k)}_\L$ and
$y_0\in\ol{E_0}^{(k)}_\L$, $y_1\in\ol{E_1}^{(k)}_\L$.
For  
$u=x+x_0+x_1$ and $v=y_0+y_1$, we  
use formula (\ref{formparitychange}) 
to get
\begin{align*}
f_\L(u,v)=\sum_{m,l=0}^\infty\frac{1}{m!l!}&d^{m+1}f_l(x)(y_0,x_0,
\ldots , x_0,x_1,\ldots,x_1)+\\
\sum_{m,l=0}^\infty\frac{1}{m!l!}& d^m 
f_{l+1}(x)(x_0,\ldots,x_0,y_1,x_1\ldots,x_1).
\end{align*}
Therefore, $f_\L(u,\bl)$ is parity changeable with
\begin{align*}
&(f_\L(u,\bl))_{(0)}=\sum_{m,l=0}^\infty\frac{1}{m!l!} d^{m+1}
f_l(x)(\bl,x_0,\ldots , x_0,x_1,\ldots,x_1)\quad\text{and}\\
&(f_\L(u,\bl))_{(1)}=\sum_{m,l=0}^\infty\frac{1}{m!l!} d^m 
f_{l+1}(x)(x_0,\ldots,x_0,\bl,x_1\ldots,x_1).
\end{align*}
In the next step, we show that $(\tilde{f}_n)_n$ is the skeleton of $\Pi(f)$.
Let $\tilde{y}\in \Pi(E)_0$, $\tilde{y}_0\in\ol{\Pi(E)_0}^{(k)}_{\L^\nil}$ and
$\tilde{y}_1\in\ol{\Pi(E)_1}^{(k)}_{\L}$. We calculate
\begin{align*}
d^m\tilde{f}_l(x,\tilde{y})((x_0,\tilde{y}_0),\ldots,&(x_0,\tilde{y}_0),(x_1,
\tilde { y } _1),\ldots,(x_1,\tilde{y}_1))=\\
&d^mf_{l+1}(x)(x_0,\ldots,x_0,\tilde{y},x_1,\ldots,x_1)\\
&+m\cdot 
d^{m-1}f_{l+1}(x)(x_0,\ldots,x_0,\tilde{y}_0,x_1,\ldots,x_1)\\
&+l\cdot  d^{m} f_{l-1}(x,\tilde{y}_1)(x_0,\ldots,x_0,x_1,\ldots,x_1),
\end{align*}
where the last two 
summands are zero for $m=0$ and $l=0$, respectively. Note that 
\begin{align*}
 &l\cdot  d^{m} f_{l-1}(x,\tilde{y}_1)(x_0,\ldots,x_0,x_1,\ldots,x_1)\\
 &=
 l \cdot d^{m+1} f_{l-1}(x)(\tilde{y}_1,x_0,\ldots,x_0,x_1,\ldots,x_1)
\end{align*}
holds because of Lemma \ref{lemufamily2}.
If $\tilde{f}\colon\U\times\ol{\Pi(E)}^{(k)}\to\ol{\Pi(F)}^{(k)}$ is the 
morphism defined by $(\tilde{f}_n)_n$, then it follows
\begin{align*}
\tilde{f}_\L(u,\tilde{v})=\sum_{m,l=0}^\infty\frac{1}{m!l!}&d^{m}f_{l+1}(x)(x_0,
\ldots , x_0,\tilde{y}+\tilde{y}_0,x_1,\ldots,x_1)+\\
\sum_{m,l=0}^\infty\frac{1}{m!l!}& d^{m+1} 
f_{l}(x)(\tilde{y}_1,x_0,\ldots,x_0,x_1\ldots,x_1)
\end{align*}
for $\tilde{v}=\tilde{y}+\tilde{y}_0+\tilde{y}_1$. This is exactly 
$(\ol{\Pi}(f))_\L(u,\tilde{v})$.
\end{proof}
\begin{corollary}\label{corparchangefunc}
 Let $E,E',F,F',H,H'\in\SVec$, $k\in\N\cup\{\infty\}$ and $\U\subseteq 
\ol{H}^{(k)}$, $\V\subseteq\ol{H'}^{(k)}$ be open subfunctors.
Moreover,
 let
$f\colon\U\times\ol{E}^{(k)}\to\ol{F}^{(k)}$ be an
$\U$-family of $\ol{\R}$-linear morphisms, 
$g\colon\V\times\ol{E'}^{(k)}\to\ol{F'}^{(k)}$ be an
$\V$-family of $\ol{\R}$-linear morphisms and $h\colon\U\to\V$ be supersmooth.
Then $g\circ(h,f)\colon\U\times\ol{E}^{(k)}\to\ol{F'}^{(k)}$ is an $\U$-family 
of $\ol{\R}$-linear morphisms and we have
\[
 \ol{\Pi}(g\circ(h,f))=\ol{\Pi}(g)\circ(h,\ol{\Pi}(f)).
\]
In addition, we have $\ol{\Pi}(\ol{\Pi}(f))=f$.
\end{corollary}
\begin{proof}
 This follows from the pointwise definition of $\ol{\Pi}$ in Lemma 
\ref{lemparchange} and Lemma \ref{lemparchangefunc}.
\end{proof}
\begin{proposition}
 For $k\in\N\cup\{\infty\}$ let $\pi\colon\E\to\M$ be a $k$-super vector 
bundle with typical fiber $F\in\SVec$, bundle atlas 
$\{\Psi^\a\colon\U^\a\times\ol{F}^{(k)}\to\E\colon\a\in A\}$ and the 
respective change of 
charts $\Psi^{\a\b}=(\phi^{\a\b},\psi^{\a\b})$, $\a,\b\in A$. Then the 
morphisms
$(\phi^{\a\b},\ol{\Pi}(\psi^{\a\b}))$ define a $k$-super vector bundle 
$\ol{\Pi}(\E)$ over $\M$ with typical fiber $\Pi(F)$. 

Let $\pi'\colon\E'\to\M'$ be another $k$-vector bundle and
$f\colon\E\to\E'$ be a morphism of $k$-super vector bundles
over $h\colon\M\to\M'$.
If $f$ has the local form $(g^{\a\a'},\varphi^{\a\a'})$, then 
$(g^{\a\a'},\ol{\Pi}(\varphi^{\a\a'}))$ defines a morphism 
$\ol{\Pi}(f)\colon\ol{\Pi}(\E)\to\ol{\Pi}(\E')$ over $h$. This construction is 
functorial and defines an equivalence of categories
\[
 \ol{\Pi}\colon\SVBunk{k}\to\SVBunk{k}.\nomenclature{$\ol{\Pi}$}{}
\]
\end{proposition}
\begin{proof}
In light of Remark \ref{remsvbunlocal},
 it follows from Corollary \ref{corparchangefunc} that the 
morphisms $(\phi^{\a\b},\ol{\Pi}(\psi^{\a\b}))$ define a super vector bundle.
That $\ol{\Pi}$ is well-defined on morphisms and functorial follows by the same 
argument. The corollary also implies that $\ol{\Pi}(\ol{\Pi}(\E))\cong\E$ and 
$\ol{\Pi}(\ol{\Pi}(f))\cong f$ hold under this identification, which shows that 
$\ol{\Pi}$ is an equivalence of categories.
\end{proof}
\subsection{The Tangent Bundle of a Supermanifold}
In this section, we expand on the definition of the tangent functor $\T$ 
for supermanifolds given by Molotkov (see \cite[Section 5.3, p. 404f.]{Mol}).

Let $k\in\N_0\cup\{\infty\}$ and $\M\in\Man^{\Grk{k}}$. We define a functor 
$\T\M\in\Man^{\Grk{k}}$
by setting $(\T\M)_\L=\T\M_\L:=T\M_\L$ for 
all $\L\in\Grk{k}$ and 
$(\T\M)_\varrho=\T\M_\varrho:=T\M_\varrho\colon\T\M_\L\to\T\M_{\L'}$ 
for all $\varrho\in\Hom_{\Grk{k}}(\L,\L')$. It follows from the functoriality 
of $T\colon\Man\to\Man$ that $\T\M$ is indeed a functor.
By the same argument, the bundle projections 
$\pi^{\T\M}_\L\colon\T\M_\L\to\M_\L$ define a natural transformation 
$\pi^{\T\M}\colon\T\M\to\M$.

If $\Nc\in\Man^{\Grk{k}}$ and $f\colon\M\to\Nc$ is a natural transformation, 
then it is easy to see that setting $\T 
f_\L:=Tf_\L\colon \T\M_\L\to\T\Nc_\L$ for all $\L\in\Grk{k}$ defines a 
natural transformation $\T f\colon\T\M\to\T\Nc$ and that this gives us a functor
$\T\colon\Man^{\Grk{k}}\to\Man^{\Grk{k}}$.
We obviously have $\pi^{\T\Nc}\circ\T f=f \circ\pi^{\T\M}$.
\begin{lemma}
 Let $k\in\N_0\cup\{\infty\}$ and $\M\in\SMank{k}$ be modelled on $E\in\SVec$ 
with 
the atlas $\{\varphi^\a\colon\U^\a\to\M\colon \a\in A\}$. Then $\T\M$ is a 
$k$-super vector bundle over $\M$ with typical fiber $E$, the bundle atlas 
$\{\T\varphi^\a\colon\T\U^\a\to\T\M\colon \a\in A\}$ and the projection 
$\pi^{\T\M}$. If $f\colon\M\to\Nc$ is a 
morphism of $k$-supermanifolds, then
 $\T f\colon\T\M\to\T\Nc$
is a morphism of $k$-super vector bundles and the above defines a functor
\[
  \T\colon\SMank{k}\to\SVBunk{k}.\nomenclature{$\T$}{}
\]
\end{lemma}
\begin{proof}
That $\{\T\varphi^\a\colon\T\U^\a\to\T\M\colon \a\in A\}$ is a covering is 
obvious.
By functoriality, we have
\[
 \T(\varphi^\b)^{-1}\circ \T\varphi^\a=\T\varphi^{\a\b}
\]
on $\T\U^{\a\b}=\U^{\a\b}\times\ol{E}^{(k)}$ for all $\a,\b\in A$ and by 
definition, we have
\[
 \T\varphi^{\a\b}= (\varphi^{\a\b},\de\varphi^{\a\b})
 \colon \U^{\a\b}\times\ol{E}^{(k)}\to \U^{\b\a}\times\ol{E}^{(k)},
\]
which is a supersmooth morphism because of Lemma \ref{lemdfss}. 
Clearly, each $\pi^{\T\M}_\L\colon\T\M_\L\to\M_\L$ is a vector bundle and we 
have that
\[
(\varphi^\a)^{-1}\circ\pi^{\T\M}\circ\T\varphi^\a\colon\U^\a\times\ol{E}^{(k)
}\to\U^\a
\]
is simply the projection and thus supersmooth.
Since, by definition, $\de\varphi^{\a\b}$ is an $\U^{\a\b}$-family of 
$\ol{\R}$-linear morphisms, the above atlas is indeed a bundle 
atlas for $\T\M$.
In such charts, $\T f$ has locally the form $(f^{\a\b'},\de f^{\a\b'})$ and 
therefore is a morphism of $k$-super vector bundles for the same reason. 
Functoriality follows from the 
functoriality of $\T$ as a functor $\Man^{\Grk{k}}\to\Man^{\Grk{k}}$.
\end{proof}
In the situation of the lemma, we call $\T\M$ 
\textit{the tangent bundle of }$\M$\index{Supermanifold!tangent bundle}.
We will write $\pi^{\T\M}\colon \T\M\to\M$\nomenclature{$\pi^{\T\M}$}{} 
for the bundle projection and $\T_x\M$\nomenclature{$\T_x\M$}{} instead 
of $(\T\M)_x$ for the fiber of $\T\M$ at a point $x$ of $\M$.
\begin{lemma}\label{lemtlim}
 For every $\M\in\SMan$, we have
 \[
  \varprojlim\T\M\cong T\varprojlim\M
 \]
 in $\MBunk{\infty}$
 with the functor $\varprojlim$ from Theorem \ref{thrmsmanmbun}.
 Moreover, $\varprojlim \pi^{\T\M}=\pi^{T\varprojlim \M}$ holds for the 
bundle 
projections $\pi^{T\varprojlim \M}\colon 
T\varprojlim \M\to\varprojlim \M$ and $\pi^{\T\M}\colon\T\M\to\M$.
 For morphisms $f\colon\M\to\Nc$ of supermanifolds, we have
 \[
  \varprojlim \T f=T\varprojlim f
 \]
 under the above identification.
\end{lemma}
\begin{proof}
 This follows from the definition of $\T\M$, Lemma 
\ref{lemtmbun} and the definition of $\varprojlim$ in Theorem 
\ref{thrmsmanmbun}.
\end{proof}
\begin{remark}
 In view of Lemma \ref{lemtlim}, it seems likely that one can 
 describe higher tangent bundles, higher jet bundles and higher tangent Lie 
supergroups analogously to \cite{Bert}.
\end{remark}
\begin{lemma}\label{lemipit}
 Let $k\in\N_0\cup\{\infty\}$ and $\M\in\SMank{k}$. With the functors from 
Lemma \ref{lemsvbunfunct}, we have  
 $\T\i^n_k(\M)\cong\i^n_k(\T\M)$ for $n\in\{0,1\}, n\leq 
k$ in $\SVBunk{k}$ and
 $\T\pi^k_n(\M)\cong\pi^k_n(\T\M)$ for $0\leq n\leq k$ in $\SVBunk{n}$.
\end{lemma}
\begin{proof}
 With any atlas $\A:=\{\varphi^\a\colon\a\in A\}$ of $\M$ it is obvious that 
applying $\T\circ\i^n_k$ and $\i^n_k\circ\T$ to a change of charts leads to the 
same morphism. The same is true for $\T\circ\pi^k_n$ and $\pi^k_n\circ\T$.
\end{proof}

\appendix
\section{Multilinear Bundles}\label{chapmullin}
Multilinear bundles were introduced in \cite{Bert} to describe higher order 
tangent bundles. As it turns out, the structure of supermanifolds is closely 
related to the structure of multilinear bundles. One important addition 
introduced below, is the inverse limit of multilinear bundles.

For this section, we fix
the infinitesimal generators 
$\ep_k$\nomenclature{$\ep_k$}{}, $k\in\N$,  with the 
relations $\ep_i\ep_j=\ep_j\ep_i$ and $\ep_i\ep_i=0$.
As usual, 
we set 
$\ep_I:=\ep_{i_1}\cdots\ep_{i_r}$\nomenclature{$\ep_I$}{} for 
$I=\{i_1,\ldots,i_r\}\subseteq\{1,\ldots,k\}$.
\subsection{Multilinear Spaces}
\begin{definition}[{\cite[MA.2, p.169]{Bert}}]\label{defcube}
 Let $k\in\N$. A \textit{(locally convex) $k$-dimensional 
cube}\index{Cube}\index{Cube!$k$-dimensional} is a family 
$(E_I)_{I\in\Po^k_+}$ of (locally convex) $\R$-vector spaces with the 
\textit{total space}\index{Cube!total space of a}
\[
 E:=\bigoplus_{I\in 
\Po^k_+}E_I.
\]
We denote the elements of $E$ by $v=\sum_{I\in\Po^k_+}v_I$ or by 
$v=(v_I)_{I\in\Po^k_+}$ with $v_I\in E_I$. By abuse of notation, we will  
call $E$ a $k$-dimensional cube as well.
For convenience, we let a $0$-dimensional cube be defined by the total space 
$\{0\}$. The spaces $E_I$ are called the \textit{axes of $E$}\index{Cube!axes 
of a}.

Let $(E_I)$ and $(E'_I)$ be $k$-dimensional 
cubes.
For each partition $\nu\in\Part(I)$, $I\in\Po^k_+$, let $f^\nu$ be an 
$\R$-$\ell(\nu)$-multilinear map
\[
 f^\nu\colon E_\nu:=E_{\nu_1}\times\ldots\times 
E_{\nu_{\ell(\nu)}}\to 
E'_I, 
\]
\[
v_\nu:=(v_{\nu_1},\ldots,v_{\nu_{\ell(\nu)}})\mapsto 
f^\nu(v_\nu).
\]
A \textit{morphism}\index{Cube!morphism of} of (locally convex) $k$-dimensional 
cubes 
$E$ and $E'$ is a (continuous) map of 
the form
\[
 f\colon E\to E',\quad \sum_{I\in 
\Po^k_+}v_I\mapsto\sum_{I\in 
\Po^k_+}
\sum_{\nu\in \Part(I)}f^\nu(v_\nu ).
\]
The composition of two morphisms is simply the composition of maps.
We define the \textit{product}\index{Cube!product of} $E\times E'$ by
$(E\times E')_I:=E_I\times E'_I$.
\end{definition}
Clearly,
a morphism $f$ of $k$-multilinear bundles (that are also 
locally convex $k$-multilinear bundles) is continuous if and only if all 
$f^\nu$ are continuous.
\begin{theorem}\label{theomspace}
\
\begin{enumerate}
 \item[(a)]The (locally convex) $k$-dimensional cubes and their 
(continuous) morphisms  
form a category, 
which we will 
call the category of (locally convex) $k$-multilinear 
spaces.\index{$k$-multilinear space}
\index{$k$-multilinear space!locally convex}
\item[(b)] A morphism $f\colon E\to E'$ of 
$k$-dimensional cubes is 
invertible if and only if $f^\nu$ is a bijection for all partitions of the form
$\nu=\{I\}$, $I\in\Po^k_+$, i.e., for all partitions of length one. In this 
case 
$f^{-1}$ is again a 
morphism of $k$-dimensional cubes.
\item[(c)] If $f\colon E\to E'$ is a morphism of locally convex $k$-dimensional 
cubes such that 
$f^{\{I\}}$ is bijective with continuous inverse
for all $I\in\Po^k_+$, then $f$ is invertible.
\end{enumerate}
\end{theorem}
\begin{proof}
 Items (a) and (b) are just \cite[MA.6, p.172]{Bert} and
 (c) follows from the inductive construction in that 
proof.
\end{proof}
We denote the category of $k$-multilinear spaces by 
$\MSpacek{k}$\nomenclature{$\MSpacek{k}$}{}.
\begin{remark}\label{remgenorder}
 It is calculated in the proof of \cite[Theorem MA.6, p.172]{Bert}  that the 
composition of morphisms $f\colon E\to E'$, $g\colon E'\to E''$ of 
$k$-dimensional cubes $E, E'$ and $E''$ is given by
\begin{align}\label{formulacubemult}
 (g\circ 
f)^\nu(v_\nu)=\sum_{\omega\preceq\nu}g^\omega\left(f^{
\omega_1|\nu}(v_{\omega_1|\nu}),\ldots,f^{
\omega_{\ell(\omega)}|\nu}(v_{\omega_{\ell(\omega)}|\nu}) \right).
\end{align}
Of course the sets
$\underline{\omega_1|\nu},\ldots,\underline{\omega_{\ell(\omega)}|\nu}$ need 
not be in (graded) lexicographic order but
by abuse of notation, we also write $g^\omega$ for the map that arises from 
permuting the factors.
\end{remark}
\begin{definition}\label{defsubcube}
Let $(E_I)$ be a $k$-dimensional cube.
For $P\subseteq \Po^k_+$, we define the 
\textit{restriction}\index{Cube!restriction}
$((E|_P)_I)$ of $(E_I)$ by 
\[
(E|_P)_I:=\begin{cases}
           E_I&\quad\text{if}\ I\in P\\
           \{0\}&\quad\text{if}\ I\in\Po^k_+\setminus P.
          \end{cases}
\]  
It has the total space
\[
 E|_P=\bigoplus_{J\in P}E_J\oplus\bigoplus_{I\in\Po^k_+\setminus P}\{0\}.
\]
When convenient, we identify the restriction $E|_P$ with the respective 
$n$-dimensional cube in the obvious way if $P=\Po^n_+\subseteq\Po^k_+$ for 
$n\leq k$. 
If $E|_{\Poee{k}}=E$ holds, i.e., if $E_I=\{0\}$ for $I\in\Po^k_1$, we 
call $E$ \textit{purely even}\index{Cube!purely even}.
\end{definition}
\begin{lemma}\label{lemcubesubalg}
 Let $P\subseteq \Po_+^k$ be a subset such that 
$\{\sum_I\ep_Ia_I\colon I\in P,a_I\in\R\}$ is a subalgebra of 
$\R[\ep_1,\ldots,\ep_k]$. Then every morphism of $k$-multilinear spaces 
$f\colon E\to E'$ can be restricted in a natural way to a morphism $f|_P\colon 
E|_P\to E'|_P$.
This restriction defines a functor 
\[
\MSpacek{k}\to\MSpacek{k}
\]
that respects products.
\end{lemma}
\begin{proof}
 Let $E,E'$ be $k$-dimensional cubes and $f\colon E\to E'$ be a morphism 
given by
 \[
  f^\nu\colon E_{\nu_1}\times\cdots\times 
E_{\nu_{\ell(\nu)}}\to E'_I,
 \]
 for $\nu\in \Part(I)$ and $I\in\Po^k_+$. 
 If $I\notin P$, then there exists $1\leq i\leq\ell(\nu)$ such that 
$\nu_i\notin 
P$, which implies $f(E|_P)\subseteq E'|_P$.
 Therefore, we can define $f|_P$ by setting $(f|_P)^\nu:=f^\nu$ if 
$\nu_1,\ldots,\nu_{\ell(\nu)}\in P$ and $(f|_{P})^\nu:=0$ else.
Let $E''$ be another $k$-dimensional cube and $g\colon E'\to E''$ be a morphism.
Since $g(E'|_P)=g|_P(E'|_P)$ holds, functoriality follows. 
That the restriction respects products is obvious.
\end{proof}
The purely even $k$-dimensional cubes clearly form a full subcategory of 
$\MSpacek{k}$ which we denote by 
$\MSpaceke{k}$.\nomenclature{$\MSpaceke{k}$}{} It follows from Lemma 
\ref{lemcubesubalg} that we have an essentially surjective restriction functor
\[
 \MSpacek{k}\to\MSpaceke{k},\quad E\mapsto E|_{\Poee{k}}\quad\text{and}\quad 
f\mapsto f|_{\Poee{k}}
\]
for $E,E'\in\MSpacek{k}$ and $f\colon E\to E'$ a morphism.
\begin{definition}\label{defsgnpart}
 Let $I\in \Po^k_+$ and $\nu=\{\nu_1,\ldots,\nu_\ell\}\in \Part(I)$. We 
define a tuple $(\nu_1|\cdots|\nu_\ell)$ by concatenating the elements of 
$\nu_1,\ldots,\nu_\ell$ 
in ascending order and define $\sgn(\nu)$\nomenclature{$\sgn(\nu)$}{}, the 
\textit{sign 
of}\index{Partition!sign of a} $\nu$, as the 
sign of the permutation needed to bring this tuple into strictly ascending 
order.
\end{definition}
This definition depends on the order one chooses on the partitions. However, 
the sign of a partition taken with regard to the lexicographic order is the 
same as when one takes it with regard to the graded lexicographic order, 
because changing the position of sets with even cardinality does not change 
the sign. We will only use these two orders in the following.
\begin{example}
 Let $\nu=\{\{2\},\{1,3\}\}$. Then we get the tuple $(\nu_1|\nu_2)=(2,1,3)$ 
and 
to 
permutate this tuple to $(1,2,3)$, the permutation $\sigma=(1, 2)$ is needed. 
Thus $\sgn(\nu)=-1$.
\end{example}

\begin{lemma}\label{lemmspaceinus}
 Let $E, E'\in \MSpaceke{k}$ and $f\colon E\to E'$ be a morphism 
defined by the family $(f^\nu)_{\nu\in\Part(\{1,\ldots,k\})}$. Setting 
$E^-:=E$, 
we let the 
morphism $f^-\colon E^-\to E'^-$ be given by 
$(\sgn(\nu)f^\nu)_{\nu\in\Part(\{1,\ldots,k\})}$. This defines 
a functor
\[
 {}^-\colon\MSpaceke{k}\to\MSpaceke{k}.
\]
The functor is inverse to itself and respects products.
\end{lemma}
\begin{proof}
 Let $E,E',E''\in\MSpaceke{k}$
and let $f\colon E\to E'$, $g\colon E'\to E''$ be morphisms. 
To check functoriality, it suffices to assume 
$I=\{i_1,\ldots,i_s\}\in\Poee{k}$ and 
$\nu\in\Parte(I)$ because if any $|\nu_j|$ is odd, 
then 
$f^\nu=0$ holds.
Recall formula (\ref{formulacubemult}) from Remark \ref{remgenorder}. On the 
one 
hand, we have
\begin{align*}
 &\big((g\circ 
f)^-\big)^\nu(v_\nu)=\\
&\sgn(\nu)\sum_{\omega\in\Parte(I),\ \omega\preceq\nu}
g^\omega\left(f^ {
\omega_1|\nu}(v_{\omega_1|\nu}),\ldots,f^{
\omega_{\ell(\omega)}|\nu}(v_{\omega_{\ell(\omega)}|\nu}) \right).
\end{align*}
On the other hand, we calculate
\begin{align*}
 (g^-\circ f^-)^\nu(v_\nu)=&
 \sum_{\substack{\omega\in\Parte(I),\\ \omega\preceq\nu}}
 \sgn(\omega)
 \cdot\sgn(\omega_1|\nu)\cdots\sgn(\omega_{\ell(\omega)}|\nu)\\
 &\phantom{\sum}\quad
 g^\omega\left(f^{
\omega_1|\nu}(v_{\omega_1|\nu}),\ldots,f^{
\omega_{\ell(\omega)}|\nu}(v_{\omega_{\ell(\omega)}|\nu}) \right).
\end{align*}
The sign $\sgn(\nu)$ describes the reordering of the tuple 
$(\nu_1|\cdots|\nu_{\ell(\nu)})$ to $(i_1,\ldots, i_s)$.
For $\omega\preceq\nu$ let 
$\omega_j|\nu=\{\nu_1^j,\ldots,\nu^j_{\ell_j}\}\in\Part(\omega_j)$. Then 
\[
\sgn(\omega)
 \cdot\sgn(\omega_1|\nu)\cdots\sgn(\omega_{\ell(\omega)}|\nu)
\]
 gives the sign of the reordering of the tuple
$(\nu^1_1|\cdots|\nu^1_{\ell_1}|\cdots|\nu^{\ell}_1|\cdots|\nu^{\ell}_{
\ell_{\ell(\omega)} })$ to $(i_1,\ldots, i_s)$. Since we only need to 
consider 
$\nu_j$ with even cardinality, reordering 
$(\nu^1_1|\cdots|\nu^1_{\ell_1}|\cdots|\nu^{\ell}_1|\cdots|\nu^{\ell}_{
\ell_{\ell(\omega)} })$ to $(\nu_1|\cdots|\nu_{\ell(\nu)})$ does not change the 
sign and it follows $\sgn(\nu)=\sgn(\omega)
 \cdot\sgn(\omega_1|\nu)\cdots\sgn(\omega_{\ell(\omega)}|\nu)$. This implies 
$g^-\circ f^-=(g\circ f)^-$.
That the functor respects products is obvious.
\end{proof}
The motivation for the lemma is essentially to substitute the infinitesimal 
generators $\ep_i$ with $\oddgen_i$. For more details see Remark 
\ref{remminus} below.
\subsection{Multilinear Bundles}
\begin{definition}[{compare \cite[15.4, p.81]{Bert}}]\label{defmbun}
\
\begin{enumerate}
 \item[(a)] Let $E$ be a locally convex $k$-dimensional cube.
 A \textit{multilinear bundle (with base $M$, of degree 
$k$)}\index{Multilinear bundle} is a smooth fiber bundle 
$F$ over a manifold $M$ with typical fiber $E$ together with an equivalence 
class of bundle atlases
such that 
the change of charts leads to an isomorphism of 
locally convex $k$-dimensional cubes on the fibers.
 \item[(b)] Let $F$ and $F'$ be multilinear bundles of degree $k$ with base 
$M$, resp.\ $M'$. A \textit{morphism of multilinear 
bundles}\index{Multilinear bundle!morphism of} is a smooth fiber 
bundle morphism $f\colon F\to F'$ that locally (i.e., in 
bundle charts) leads to a morphism of the respective $k$-dimensional cubes 
in 
each fiber.
\end{enumerate}
We identify multilinear bundles of degree zero with their base manifold in the 
obvious way.
\end{definition}
It follows from Theorem \ref{theomspace} that the multilinear bundles form a 
category 
which we denote by $\MBun$\nomenclature{$\MBun$}{}. Multilinear bundles 
of degree $k$ form a full subcategory 
denoted by $\MBunk{k}$\nomenclature{$\MBunk{k}$}{}.
\begin{remark}\label{fmbunmorph}
 The above definition means that in bundle charts a morphism $f\colon F\to F'$ 
of multilinear bundles of degree $k$ with fiber $E$, resp.\ 
$E'$, has the form
\begin{align*}
 U\times E\to U'\times E',\quad 
 (x,(v_I))\mapsto
 \Big(\varphi(x),\sum_{I\in\Po^k_+}
\sum_{\nu\in \Part(I)}f_{x}
^\nu(v_\nu )\Big),
\end{align*}
where $\varphi\colon U\to U'$ is the local representation of the morphism 
induced by $f$ on the base manifolds and 
\[
f_{x}^\nu\colon 
E_{\nu_1}\times\ldots\times E_{\nu_{\ell(\nu)}}\to E'_I
\]
is a multilinear map for each $x\in U$, $I\in\Po^k_+$ and $\nu\in \Part(I)$. A 
function of this form is smooth if and only if $\varphi\colon U\to U'$ is 
smooth and the maps $(x,v_\nu)\mapsto 
f_{x}^\nu(v_\nu)$ are all smooth. This can be easily checked by 
restricting to the closed subspaces of $E$ defined by a given partition. 
\end{remark}
\begin{definition}
 Let $F$ and $F'$ be multilinear bundles of degree $k$ over $M$ and $M'$, 
with typical fiber $E$ and $E'$. Further, let $\{\varphi_\a\colon V_\a\to 
U_\a\times E\colon \a\in A\}$ and
$\{\psi_\b\colon V'_\a\to 
U'_\a\times E'\colon \b\in B\}$ be bundle atlases of $F$ and $F'$.
We let the \textit{product bundle} $F\times F'$\index{Multilinear 
bundle!product} be the multilinear bundle of 
degree $k$ over $M\times M'$ with typical fiber $E\times E'$ given by the 
bundle atlas 
\[
\{\varphi_\a\times\psi_b\colon V_\a\times V'_b\to (U_\a\times 
U'_b)\times (E\times E')\colon \a\in A,\b\in B\}.
\] 
\end{definition}
\begin{lemma/definition}\label{lemdefmlsubbundle}
 Let $F$ be a multilinear bundle of degree $k$ with 
typical fiber $E$ and bundle atlas 
$\{\varphi_\a\colon V_\a\to U^\a\times E\colon \a\in A\}$. Let 
$P\subseteq \Po_+^k$ be a subset 
such that 
$\{\sum_I\ep_Ia_I\colon I\in P,a_I\in\R\}$ is a subalgebra of 
$\R[\ep_1,\ldots,\ep_k]$. Each change of bundle charts
$\varphi_{\a\b}\colon U_{\a\b}\times E\to U_{\b\a}\times E$
of $F$ restricts to a map 
\[
\varphi_{\a\b}|_P\colon U_{\a\b}\times E|_P\to U_{\b\a}\times E|_P.
\]
We denote the multilinear bundle of degree $k$ that is defined by the charts 
$\varphi_\a|_{(\varphi_a)^{-1}(U_\a\times E|_P)}$ by $F|_P$ and call it a 
\textit{subbundle of $F$}\index{Multilinear bundle!subbundle of a}. If 
$P=\Po^n_+\subseteq \Po^k_+$ for $k\leq n$, we identify $F|_P$ with the 
respective multilinear bundle of degree $n$ in the obvious way (compare to 
Definition \ref{defsubcube}).
Any morphism $f\colon F\to F'$ of multilinear bundles of degree $k$ 
restricts 
to a morphism $f|_P\colon F|_P\to F'|_P$. This restriction defines a functor
\[
 \MBunk{k}\to\MBunk{k}\nomenclature{$"|_P$}{}
\]
that respects products.
\end{lemma/definition}
\begin{proof}
Applying Lemma \ref{lemcubesubalg} pointwise to transition maps and morphisms 
of 
multilinear 
bundles in their chart representation shows that $F|_P$ and $f|_P$ are 
well-defined and that the restriction is functorial.
That the restriction respects products is obvious.
\end{proof}
Note that in the situation of the definition, the identification of 
$F|_{\Po^n_+}$ with a bundle of degree $n$ is not a morphism of multilinear 
bundles but only a diffeomorphism of manifolds. 
There are cases where a subbundle is a multilinear bundle of lesser 
degree in a natural way that are not contained in the above definition. One 
important example is the following.
\begin{lemma/definition}
 Let $\pi\colon F\to M$ be a multilinear bundle of degree $k$ with 
typical fiber 
$E$. For each $I\in \Po^k_+$, we have a 
subbundle $F|_{\{I\}}$ which has the structure of a vector bundle with fiber 
$E_I$ in a natural way.
The $2^k-1$ vector bundles obtained in this way are called the 
\textit{axes of $F$}\index{Multilinear bundle!axes of a}.
\end{lemma/definition}
\begin{proof}
For any change of bundle charts $\varphi_{\a\b}\colon U_{\a\b}\times E\to 
U_{\b\a}\times E$ of $F$, we have
that the corresponding change of bundle charts 
\[
 \varphi_{\a\b}|_{\{I\}}\colon U_{\a\b}\times E_I\to 
U_{\b\a}\times E_I
\]
of $F|_{\{I\}}$ is linear in the second component. Thus the restricted charts 
define a vector bundle with typical fiber $E_I$.
\end{proof}
Bertram uses the above fact in \cite[15.4, p.81]{Bert} to 
\textit{define} multilinear 
bundles by letting these axes take an analogous role to the axes in cubes. It 
is 
easy to see that both definitions are equivalent but our 
definition via bundle charts makes the relation of multilinear bundles 
to supermanifolds more direct.
\begin{definition}\label{defpebundle}
 A \textit{purely even multilinear bundle}\index{Multilinear bundle!purely 
even} is a multilinear bundle $F$ of degree $k$ such 
that $F|_{\Poee{k}}=F$.
The purely even multilinear bundles form a full subcategory of 
$\MBun$ (resp.\ $\MBunk{k}$), which we denote by 
$\MBun_{\ol{0}}$\nomenclature{$\MBun_{\ol{0}}$}{} 
(resp.\ $\MBunk{k}_{\ol{0}}$\nomenclature{$\MBunk{k}_{\ol{0}}$}{}) and we have 
the essentially surjective restriction functor
\[
 \MBun\to\MBun_{\ol{0}},\quad F\mapsto F|_{\Poee{k}}\quad \text{and}\quad 
f\mapsto f|_{\Poee{k}}
\]
for $F,F'\in\MBunk{k}$ and $f\in\Hom_{\MBunk{k}}(F,F')$ (resp.\ 
$\MBunk{k}\to\MBunk{k}_{\ol{0}}$).
\end{definition}

\begin{example}\label{extangent}
Let $k\in\N$.
 \begin{enumerate}
  \item[(a)] Let $U\subseteq E$ be an open subset of a locally convex vector 
space $E$. Define inductively $TU:=U\times\ep_1 E$, $T^2U=T(U\times 
\ep_1E)=U\times\ep_1 E\times\ep_2 E\times\ep_1\ep_2 E$ and so on. 
Then $T^kU=U\times\bigoplus_{I\in 
\Po^k_+}\ep_IE$ is a 
trivial multilinear bundle over $U$ of degree $k$. The axes are the trivial 
vector 
bundles 
$U\times\ep_IE\to U$.
  \item[(b)] Let $M$ be a manifold with the atlas 
$\{\varphi_\alpha\colon V_\alpha\to U_\a\colon\alpha\in A\}$. Then $T^kM$ is a 
multilinear 
bundle over $M$ of degree $k$ with the bundle atlas 
$\{T^k\varphi_\alpha\colon T^kV_{\alpha}\to T^kU_\a\colon\alpha\in A\}$. Let 
$\varphi_{\a\b}$ 
be a 
change of charts. Using (a), the corresponding change of 
bundle charts is given by
\[
 T^k\varphi_{\a\b}\Big(x,\sum_{I\in\Po^k_+} 
\ep_Iv_I\Big)=\Big(\varphi_{\a\b}(x),\sum_{m=1}^k\sum_{|I|=m}\ep_I\sum_{\nu\in 
\Part(I)}d^m\varphi_{\a\b}(x)(v_\nu) \Big)
\]
(see \cite[Theorem 7.5, p.47]{Bert}). The axes of $T^kM$ are thus 
all 
isomorphic to $TM$ and we write $\ep_ITM$\nomenclature{$\ep_ITM$}{} to 
differentiate between them. 
It also follows from \cite[Theorem 7.5, 
p.47]{Bert} that for each smooth map $f\colon M\to N$ between manifolds, 
$T^kf$ is a 
morphism of multilinear bundles and we get a functor $T^k\colon \Man\to 
\MBunk{k}$ in this way.
 \end{enumerate}
\end{example}
\begin{lemma}\label{lemtaxes}
 Let $k\in\N$ and let $f\colon M\to N$ be a smooth map between manifolds.
 For each $I\in\Po^k_1$, we have
 \[
  T^kf(\ep_ITM)\subseteq \ep_I TN
 \]
 with the notation of Example \ref{extangent}(b).
\end{lemma}
\begin{proof}
This is obvious because $T^kf$ is a morphism of multilinear bundles.
\end{proof}

\begin{lemma}\label{lemmbunminus}
 Applying the functor from Lemma \ref{lemmspaceinus} pointwise to 
transition maps and local chart representations of morphisms, we get a functor
\[
 {}^-\colon\MBun_{\ol{0}}\to\MBun_{\ol{0}},\quad F\mapsto F^-\quad 
\text{and}\quad h\mapsto h^-,\nomenclature{$"|^{-}$}{}
\]
where $F,F'\in\MBunk{k}_{\ol{0}}$ and $h\in\Hom_{\MBunk{k}_{\ol{0}}}(F,F')$. 
This functor is an 
equivalence of categories and so are the restrictions
$\MBunk{k}_{\ol{0}}\to\MBunk{k}_{\ol{0}}$. All these functors respect products.
\end{lemma}
\begin{proof}
 Locally this is obvious in view of Remark \ref{fmbunmorph} and Lemma 
\ref{lemmspaceinus}.
 By functoriality, applying this to all the change of charts of $F$ leads to 
new 
cocycles that define a bundle $F^-$.
Likewise, applying it pointwise to the chart representation of a morphism 
$h\colon F\to F'$ of purely even multilinear bundles leads in a functorial 
way to a morphism $h^-\colon F^-\to F'^-$. Obviously $(F^-)^-\cong F$ and 
$(h^-)^-=h$ under this identification, which shows that the functor is 
an equivalence of categories. 
That these functors respect products also follows because it is true locally.
\end{proof}
\begin{remark}\label{remminus}
 The intuition behind the above equivalence of categories is as follows. One 
can take the case of higher tangent bundles as exemplary and define  
$k$-dimensional cubes as families $(\ep_IE_I)$ of vector spaces. A morphism 
$f\colon E\to E'$ of $k$-multilinear spaces consists then as before of maps
\[
 f^\nu\colon E_{\nu_1}\times\cdots\times E_{\nu_{\ell(\nu)}}\to E'_I
\]
for $I\in\Po^k_+$, $\nu\in\Part(I)$, where it is understood that 
\[
 f^\nu(\ep_{\nu_1}v_{\nu_1},\ldots,\ep_{\nu_{\ell(\nu)}}v_{\nu_{\ell(\nu)}})=
\underbrace{\ep_{\nu_1}\cdots\ep_{\nu_{\ell(\nu)}}}_{=\ep_I}f^\nu(v_{\nu_1},
\ldots , v_ { \nu_ { \ell(\nu) } } ).
\]
Because of the relations of the infinitesimal generators, this point of view 
also explains why one only considers partitions for the morphisms and why the 
order of the partitions can usually be disregarded.

We would like to substitute the generators $\ep_i$ with the odd generators 
$\oddgen_i$. One immediately sees that the order of the partition now plays a 
role, as a change of signs might occur. However, as we have shown in Lemma 
\ref{lemmspaceinus}, in the 
case of purely even multilinear bundles a consistent choice can be made such 
that this substitution leads to well-defined bundles and morphisms.
In general this is not the case. With the notation of Lemma \ref{lemmspaceinus} 
one could define a new composition law 
\begin{align*}
 &(g^-\circ f^-)^\nu:=\\
 &\sum_{\omega\preceq\nu}\sgn(\sigma_{\omega|\nu})
 \sgn(\omega)
 \sgn(\omega_1|\nu)\cdots\sgn(\omega_{\ell(\omega)}|\nu)
 g^\omega\left(f^{
\omega_1|\nu},\ldots,f^{
\omega_{\ell(\omega)}|\nu} \right),
\end{align*}
where $\sigma_{\omega|\nu}\in \Sy_{|I|}$ is the 
permutation that reorders the tuple $({\nu_1}|\cdots|{\nu_{\ell}})$ to
$\big({\nu^1_1}|\cdots|{\nu^{\ell(\omega)}_{\ell_{\ell(\omega)}}}\big)$. 
If all 
$\nu_i$ have even cardinality then $\sgn(\sigma_{\omega|\nu})=1$ and we get the 
same definition as above.
In general the formula does not appear to lead to 
natural manifold structures though there is one interesting case where it 
does: If only those $f^\nu$, where $\nu$ contains at most one set of odd 
cardinality, are not zero, the same argument as before applies. This means that 
for a supermanifold $\M$ of Batchelor type, at least $\M_\L^-$ would be 
well-defined. 
However, morphisms remain problematic.
\end{remark}
\subsubsection{The tangent bundle of a multilinear bundle}
Let $F$ be a multilinear bundle of degree $k$ over $M$ with typical 
fiber $E$. 
Assume that $M$ is modelled on $E_0$ and let $\varphi\colon U\times E\to 
V\times E$ be a change of bundle charts. Then by definition
\[
\varphi(x,(v_I)_{I\in\Po^k_+})=\varphi_0(x)+\sum_{I\in\Po^k_+}\sum_{\nu\in\Part(
I)}b^\nu(x,v_\nu),
\]
where $\varphi_0\colon U\to V$ is a diffeomorphism and $b^\nu(x,\bl)\colon 
E_{\nu_1}\times\cdots\times E_{\nu_{\ell(\nu)}}\to E_I$ are multilinear maps 
for $x\in U$, $\nu\in\Part(I)$. For $y\in E_0$ and $(w_I)_{I\in\Po^k_+}\in E$, 
we calculate
\begin{align*}
 &d\varphi\big((x,(v_I)_{I\in\Po^k_+}), (y,(w_I)_{I\in\Po^k_+})\big)=\\
 &\quad d\varphi_0(x,y)+\sum_{I\in\Po^k_+}\sum_{\nu\in\Part(
I)}d_1b^\nu(x,y,v_\nu)+\sum_{I\in\Po^k_+}\sum_{\nu\in\Part(
I)}\sum_{i=1}^{\ell(\nu)}b^\nu(x,\reallywidehat[i]{v_\nu}),
\end{align*}
where 
$\reallywidehat[i]{v_\nu}:=(v_{\nu_1},\ldots,v_{\nu_{i-1}},w_{\nu_i},v_{\nu_{i+1
} } ,\ldots, v_{\nu_{\ell(\nu)}} )\in E_{\underline{\nu}}$.
The corresponding change of charts for the tangent bundle $TF$ is given by
\[
  (\varphi,d\varphi)\colon (U\times E_0)\times E^2\to(V\times E_0)\times E^2.
\]
For $I\in\Po^k_+$ let $\pr^I_1\colon E_I\times E_I\to E_I$ be the projection to 
the first and $\pr^I_2\colon E_I\times E_I\to E_I$ be the projection to the 
second component. Then
\begin{align*}
 &(\varphi,d\varphi)\big((x,(v_I)_{I\in\Po^k_+}), (y,(w_I)_{I\in\Po^k_+})\big)=
 (\varphi_0(x),d\varphi_0(x,y))+\\
 &\qquad\sum_{I\in\Po^k_+}\sum_{\nu\in\Part(
I)}
\textstyle\Big(\pr_1^I( b^\nu(x,v_\nu))+\pr_2^I\big(d_1b^\nu(x,y,v_\nu)+ 
\sum_{i=1}^{\ell(\nu)}b^\nu(x,\reallywidehat[i]{v_\nu})\big)\Big)
\end{align*}
holds.
Thus, $TF$ can be seen as a multilinear bundle of degree $k$ over $TM$ with 
typical fiber $E\times E$. The exact same calculation shows that for a morphism 
of 
multilinear bundles $f\colon F\to F'$, the tangent map $Tf\colon TF\to TF'$ 
is also a morphism of multilinear bundles. We have thus shown:
\begin{lemma}
 For each $k\in\N_0$, the tangent functor $T\colon\Man\to\Man$ restricts to a 
functor
 \[
  T\colon\MBunk{k}\to\MBunk{k}.
 \]
\end{lemma}
The functor $T\colon\MBunk{k}\to\MBunk{k}$ commutes with 
restrictions of bundles:
\begin{lemma}\label{lemmbuntrestr}
 Let $k\in\N_0$, $F\in\MBunk{k}$ and $P\subseteq\Po^k_+$ as in 
Lemma/ Definition \ref{lemdefmlsubbundle}. Then
 $(TF)|_P\cong T(F|_P)$ holds as multilinear bundles. 
If $f\colon F\to F'$ is a morphism of multilinear bundles, then 
\[
  (Tf)|_P=T(f|_P)\colon T(F|_P)\to T(F|_P)
\] 
holds under the above 
identification.
\end{lemma}
\begin{proof}
 Let $F$ have typical fiber $E$ and let the base $M$ of $F$ be modelled on 
$E_0$. Since each change of charts
$\varphi_{\a\b}\colon U_{\a\b}\times E\to U_{\b\a}\times E$
of $F$ restricts to a map 
\[
\varphi_{\a\b}|_P\colon U_{\a\b}\times E|_P\to U_{\b\a}\times E|_P,
\]
we have that $d\varphi_{\a\b}$ restricts to
\[
 d\varphi_{\a\b}|_P=d(\varphi_{\a\b}|_P)\colon (U_{\a\b}\times E_0)\times 
(E|_P\times E|_P)\to U_{\b\a}\times E|_P.
\]
It follows that
$(\varphi_{\a\b},d\varphi_{\a\b})|_P=(\varphi_{\a\b}|_P,d\varphi_{\a\b}|_P)$ 
holds.
We can repeat the same argument for morphisms.
\end{proof}
By using this lemma, we shall simply write $TF|_P$, resp.\ $Tf|_P$, for the 
respective restrictions in the sequel.

\subsection{Inverse Limits of Multilinear Bundles}
\begin{lemma}\label{lemqkn}
 Let $k\in\N_0$ and $ F$ be a multilinear bundle of degree $k$ 
with typical fiber $E$ and the bundle atlas 
$\{\varphi_\a\colon V_\a\to U_\a\times E\colon\a\in A\}$. For 
$n\leq k$, the 
projections
\[
 (q^k_n)_\a\colon U_\a\times E\to U_\a\times E|_{\Po^n_+},\quad 
\big(x,(v_I)_{I\in\Po^k_+}\big)\mapsto \big(x,(v_I)_{I\in\Po^n_+}\big)
\]
define a smooth surjective morphism $q^k_n\colon F\to F|_{\Po^n_+}$ 
with 
$\varphi_\a|_{\Po^n_+}\circ q^k_n\circ \varphi^{-1}_\a= 
(q^k_n)_\a$\nomenclature{$q^k_n$}{}.
\end{lemma}
\begin{proof}
 We only need to show that $q^k_n$ is well-defined, then smoothness and 
surjectivity follow immediately.
Let $\a,\b\in A$, $x\in U_{\a\b}$ and $(v_I)_{I\in\Po^k_+}\in E|_{\Po^n_+}$.
Then
$\varphi_{\a\b}(x,(v_I)_I)=\varphi_{\a\b}|_{\Po^n_+}(x,(v_I)_I)$ holds for 
the 
change of bundle charts $\varphi_{\a\b}$. In 
particular, we have
$\varphi_{\b\a}|_{\Po^n_+}\circ \varphi_{\a\b}(x,(v_I)_I)=(x,(v_I)_I)$. It 
follows
\[
 \varphi_{\b\a}|_{\Po^n_+}\circ (q^k_n)_\b\circ\varphi_{\a\b}= (q^k_n)_\a
\]
on $U_{\a\b}\times E$.
With this, the lemma follows from the local description 
of 
smooth maps between manifolds.
\end{proof}
\begin{definition}
Let $(F_k)_{k\in\N_0}$ be a family of multilinear bundles 
$F_k$ of degree $k$ with 
typical fiber $E^{(k)}$ and respective bundle atlas 
$\{\varphi_\a^{(k)}\colon V^{(k)}_\a\to U_\a\times E^{(k)}\colon\a\in 
A\}$ 
such that for all 
$n\leq k$, we have 
$E^{(k)}|_{\Po^n_+}=E^{(n)}$ and 
$\varphi_{\a}^{(k)}|_{\Po^n_+}=\varphi_\a^{(n)}$
with the identifications from Definition 
 \ref{defsubcube} and Lemma/Definition \ref{lemdefmlsubbundle}. In particular 
$F_k|_{\Po^n_+}=F_n$ and all $F_k$ are bundles over $F_0$. 
Then 
the family 
\[
  \big((F_k)_{k\in\N_0},(q^k_n)_{n\leq k}\big),
\] 
where $q^k_n$ is defined as in Lemma \ref{lemqkn}, is called an \textit{inverse 
system of multilinear bundles}\index{Inverse system!of multilinear bundles}.
We shall simply write $(F_k,q^k_n)$\nomenclature{$(F_k,q^k_n)$}{} in this 
situation. We call
$\{\varphi^\a_k\colon k\in\N_0,\a\in A\}$ an \textit{adapted 
atlas}\index{Inverse 
system!of multilinear bundles!adapted atlas} of $(F_k,q^k_n)$.
Two adapted atlases of $(F_k,q^k_n)$ are \textit{equivalent} if they lead to 
equivalent atlases for each $F_k$.

Let $(F_k,q^k_n)$ and $(F'_k,q'^k_n)$ be inverse systems of multilinear 
bundles. A \textit{morphism of inverse systems of multilinear 
bundles}\index{Inverse system!of multilinear bundles!morphism of} is a family 
$(f_k)_{k\in\N_0}$ of morphisms $f_k\colon F_k\to F'_k$ of multilinear 
bundles such that $q'^k_n\circ f_k= f_n\circ q^k_n$ for all $n\leq k$. We write 
$(f_k)_{k\in\N_0}\colon (F_k,q^k_n) \to(F'_k,q'^k_n)$.
\end{definition}
\begin{proposition}\label{proplimmbun}
 The inverse system of multilinear bundles with their morphisms 
are a subcategory of the category of inverse systems of
topological spaces and their morphisms.
Let $(F_k,q^k_n)$ be an inverse system of multilinear bundles and 
$\{\varphi_\a^{(k)}\colon k\in\N_0,\a\in A\}$ be an adapted atlas of 
$(F_k,q^k_n)$. Then 
$\{\varprojlim_k \varphi_\a^{(k)}\colon \a\in A\}$ is an atlas of
$\varprojlim_k F_k$. Equivalent adapted atlases of $(F_k,q^k_n)$ lead to 
equivalent atlases of $\varprojlim_k F_k$.\index{Inverse limit!of multilinear 
bundles}
With this manifold structure, $\varprojlim_k 
f_k\colon \varprojlim_k F_k\to \varprojlim_k F'_k$ is smooth for 
morphisms $(f_k)_{k\in\N_0}\colon (F_k,q^k_n) \to(F'_k,q'^k_n)$ of inverse 
systems of multilinear bundles.
\end{proposition}
\begin{proof}
  Let $ F_k$ have the typical fiber $E^{(k)}$ and let $F_0$ be 
modelled on $E_{\emptyset}$.
 It is clear from the local definition in Lemma \ref{lemqkn} that $q^n_m \circ 
q^k_n= q^k_m$ for all $m\leq n\leq k$. It then follows from the definition that 
inverse systems of multilinear bundles, resp.\ morphisms thereof, are inverse 
systems, resp.\ morphisms thereof, in the usual sense. Clearly, the composition 
of two morphisms of inverse systems of multilinear bundles is again a morphism 
of this type.
Let 
$\{\varphi_\a^{(k)}\colon V^{(k)}_\a\to 
U_\a\times E^{(k)}\colon k\in\N_0,\a\in A\}$
be an adapted atlas of $(F_k,q^k_n)$.
By definition $E^{(k)}_I=E^{(n)}_I$ holds for all $n\leq k$ and 
$I\in\Po^n_+$. Thus, for each $\a\in A$, the local projection 
\[
(q^k_n)_\a\colon U^\a\times \prod_{I\in\Po^k_+}E^{(k)}_I\to U^\a\times 
\prod_{I\in\Po^n_+}E^{(n)}_I
\]
is just the usual projection and 
\[
\varprojlim_k(U_\a\times E^{(k)})= 
U_\a\times\prod_{I\subseteq\N,0<|I|<\infty}E^{(\max(I))}_I,
\] 
which is an open 
subset of the locally convex space
\[E_{\emptyset}\times\prod_{I\subseteq\N,0<|I|<\infty}E^{(\max(I))}_I.\]
Also by definition, $q^k_n\circ\varphi_\a^{(k)}=\varphi_\a^{(n)}\circ 
(q^k_n)_\a$ holds for all $\a\in A$, $n\leq k$ and therefore 
$\varprojlim_k\varphi_\a^{(k)}\colon 
\varprojlim_k(U_\a\times E^{(k)})\to\varprojlim_k F_k$ is well-defined and a 
homeomorphism because each $\varphi_\a^{(k)}$ is so.
We have already seen in Lemma \ref{lemqkn} that the changes of charts 
$\varphi_{\a\b}^{(k)}\colon U_{\a\b}\times E^{(k)}\to U_{\b\a}\times E^{(k)}$
define a morphism of inverse systems of multilinear bundles and that we have
\[
\textstyle
\varprojlim_k \varphi_\b^{(k)}\circ \varprojlim_k 
\big(\varphi_\a^{(k)}\big)^{-1}|_{U_{\a\b}\times 
E^{(k)}}=
\varprojlim_k \varphi_{\a\b}^{(k)}.
\]
Clearly, $\varprojlim_k(U_{\a\b}\times 
E^{(k)})$ is an open subset of $\varprojlim_k(U_{\a}\times 
E^{(k)})$ and
$\varprojlim_k\varphi^{\a\b}_k$
because $\varphi_{\a\b}^{(k)}$ is smooth for each $k\in\N_0$, so is 
$\varprojlim_k\varphi_{\a\b}^{(k)}$ by Lemma \ref{lemdirectlimsmooth}.

It remains to be seen that $\varprojlim_k F_k$ 
is covered by $\{\varprojlim_k\varphi_\a^{(k)}\colon \a\in A\}$. Because the 
index set 
$\N_0$ is countable and the maps $q^k_n$ are all surjective, the  
projections $q_n\colon\varprojlim_k F_k\to F_n$ are also surjective
(see for example \cite[Exercise 7.6.10, p. 269]{DumFoo}).
For each $n\in\N_0$ and $\a\in A$, we have
$(q^n_0)^{-1}((\varphi_\a^{(0)})^{-1}(U^\a))=(\varphi_\a^{(n)})^{-1}(U_\a\times 
E^{(n)})$ which 
implies 
\[
\textstyle
q_0^{-1}((\varphi_\a^{(0)})^{-1}(U_\a))=\varprojlim_n(\varphi_\a^{(n)})^{-1}
\big(\varprojlim_n 
(U_\a\times E^{(n)})\big).
\]
Since the sets $(\varphi^\a_0)^{-1}(U^\a)$ cover $F_0$, the result follows. The 
change 
of charts with an adapted atlas leads to smooth maps in the same way.
Because $q_0\colon\varprojlim_k F_k\to F_0$ is surjective and for each $x\in 
F_0$, we have that $q_0^{-1}(\{x\})$ is homeomorphic to the 
Hausdorff space $\varprojlim_k E^{(k)}$, it follows that $\varprojlim_k F_k$ is 
Hausdorff.

Now, let $(f_k)_{k\in\N_0}\colon (F_k,q^k_n) \to(F'_k,q'^k_n)$  
be a morphism of inverse 
systems of multilinear bundles and $\{\psi^{(k)}_\b\colon V'_\b\to U'_\b\times 
E'^{(k)}\colon k\in\N_0,\b\in B\}$ be an adapted atlas of $(F'_k,q'^k_n)$. We 
define
\[
 f^{\a\b}_k:=\psi_{\b}^{(k)}\circ 
f_k\circ(\varphi_\a^{(k)})^{-1}|
_{\varphi_\a^{(k)}\circ f_k^{-1}
(V'^{(k)}_\b)}
 \]
for $\b\in B$ and $\a\in A$. Because $f_k$ is a morphism of multilinear 
bundles, we have $\varphi^{(k)}_\a\circ f_k^{-1}(V'^{(k)}_\b)=
\big(\varphi_\a^{(0)}\circ f_0^{-1}(V'^{(0)}_\b)\big)\times E^{(k)}$ 
for all 
$k\in\N_0$. By 
definition,
\[
 (q'^k_n)_\b\circ f^{\a\b}_k=f^{\a\b}_n\circ (q^k_n)_\a
\]
holds
 and thus
\[
  \textstyle
 \varprojlim_k \psi_\b^{(k)}\circ \varprojlim_k 
f_k\circ(\varprojlim_k\varphi_\a^{(k)})^{-1}
 =\varprojlim_k f^{\a\b}_k
\]
holds 
on $\varprojlim_k \big(\big(\varphi_\a^{(0)}\circ 
f_0^{-1}(V'^{(0)}_\b)\big)\times E'^{(k)}\big)$ for all $\a\in A$, $\b\in B$. 
These maps are smooth by the same argument as above.
\end{proof}
We denote by $\MBunk{\infty}$\nomenclature{$\MBunk{\infty}$}{} the category 
of all 
manifolds arising as such a limit (together with an equivalence class of 
atlases that come from limits of equivalent adapted atlases) and morphisms 
that come from a respective limit of morphisms.
Taking the inverse limit gives us a functor from the 
category of inverse systems of topological spaces to the category of 
topological spaces that respects products. 
By the above, if we restrict this functor to the subcategory 
of inverse systems of multilinear bundles (and the respective morphisms), we 
get a functor into the category $\MBunk{\infty}$.
 We also get a functor 
to $\Man$ along the forgetful functor.
\begin{example}\label{extinfty}
  Let $M$ be a manifold modelled on the locally convex space $E$ with the atlas 
$\{\varphi_\a\colon V_\a\to U_\a\colon \a\in A\}$.
  For $n\in\N_0$ we set $\pi^n_n:=\id_{T^nM}$ and we have the natural 
projection $\pi^{n+1}_n\colon T^{n+1}M\to T^nM$. For $n<k$, we define 
inductively $\pi^k_n:=\pi^k_{k-1}\circ\cdots\circ\pi^{n+1}_n\colon T^kM\to 
T^nM$.
 Continuing from Example \ref{extangent}(b), one easily sees that 
 $(T^kM,\pi^k_n)$ is an inductive system of multilinear bundles with the 
adapted atlas $\{T^k\varphi_\a\colon 
T^kV_\a\to 
U_\a\times\prod_{I\in\Po^k_+}\ep_IE\colon \a\in A,k\in\N_0\}$. It follows from 
Proposition \ref{proplimmbun} that 
$T^\infty M:=\varprojlim_k T^kM$ is a manifold with the atlas 
$\{\varprojlim_k\varphi_\a\colon\a\in A\}$. For any smooth map $f\colon M\to N$
between manifolds, one obviously has $\pi'^k_n\circ T^kf= T^nf\circ \pi^k_n$ if 
$\pi'^k_n\colon T^kN \to T^nN$ denotes the projection. Thus 
$(T^kf)_{k\in\N_0}$ 
is a morphism of inductive systems of multilinear bundles and 
$T^\infty f:=\varprojlim_k T^kf\colon T^\infty M\to T^\infty N$ is smooth.
Moreover, for any Lie group $(G,\mu,i,e)$, we get a Lie group $(T^\infty G, 
T^\infty\mu, T^\infty i, e)$ because the inverse limit preserves products.
\end{example}
\begin{lemma}\label{lemtmbun}
 If $(F_k,q^k_n)$ is an inverse system of multilinear bundles, then so is
 $(TF_k,Tq^k_n)$ and 
 \[
 \textstyle
  T\varprojlim_k((F_k,q^k_n))\cong\varprojlim_k (TF_k,Tq^k_n)
 \]
 holds
 as manifolds. If $(f_k)_{k\in\N_0}\colon (F_k,q^k_n) \to(F'_k,q'^k_n)$ is a 
morphism of inverse systems of multilinear bundles, then 
so is $(Tf_k)_{k\in\N_0}$ and we have
\[
 \textstyle
  \varprojlim_k Tf_k=T\varprojlim_k f_k\colon T\varprojlim_k((F_k,q^k_n))\to 
 T\varprojlim_k((F'_k,q'^k_n))
\]
under the above identification.
Thus, we may consider $ T\varprojlim_k((F_k,q^k_n))$ as an object in 
$\MBunk{\infty}$ in a natural way.
\end{lemma}
\begin{proof}
 One easily sees from the local description of $q^k_n$ in Lemma \ref{lemqkn} 
that $Tq^k_n$ is the projection $TF\to TF|_{\Po^n_+}$. If 
$\{\varphi_\a^{(k)}\colon k\in\N_0,\a\in A\}$ is an adapted atlas of $F$, it 
follows by 
functoriality of the tangent functor that $\{T\varphi_\a^{(k)}\colon 
k\in\N_0,\a\in A\}$ 
is an adapted atlas of $(TF_k,Tq^k_n)$. For the same reason 
$(Tf_k)_{k\in\N_0}\colon (TF_k,Tq^k_n)\to (TF'_k,Tq'^k_n)$ is again a morphism.
By Lemma \ref{lemdirectlimsmooth} 
$d\varprojlim_k\varphi_{\a\b}^{(k)}=\varprojlim_k 
d\varphi_{\a\b}^{(k)}$
holds 
for any change of charts $\varphi^{(k)}_{\a\b}$.
 Thus, the change of charts of 
$T\varprojlim_k((F_k,q^k_n))$ and $\varprojlim_k (TF_k,Tq^k_n)$ is the same.
The same argument works for morphisms.
\end{proof}
In other words, taking the inverse limit commutes with the tangent functor.

\bibliographystyle{plain}
\bibliography{literatur}

\begin{thebibliography}{10}

\bibitem{AllLau}
A.~Alldridge and M.~Laubinger.
\newblock Infinite-dimensional supermanifolds over arbitrary base fields.
\newblock {\em Forum Math.}, 24(3):565--608, 2012.

\bibitem{Bastiani}
A.~Bastiani.
\newblock Applications diff\'erentiables et vari\'et\'es diff\'erentiables de
  dimension infinie.
\newblock {\em J. Analyse Math.}, 13:1--114, 1964.

\bibitem{Batch}
M.~Batchelor.
\newblock The structure of supermanifolds.
\newblock {\em Trans. Amer. Math. Soc.}, (253):329--338, 1979.

\bibitem{BerLei}
F.~A. Berezin and D.~Le\v{i}tes.
\newblock Supermanifolds.
\newblock {\em Soviet Math. Dokl.}, (16):1218--1222, 1975.

\bibitem{Bert}
W.~Bertram.
\newblock {\em Differential Geometry, {Lie} Groups and Symmetric Spaces over
  General Base Fields and Rings}, volume 192 of {\em Memoirs of the American
  Mathematical Society}.
\newblock American Mathematical Society, 2008.

\bibitem{BGN}
W.~Bertram, H.~Gl{\"o}ckner, and K.-H. Neeb.
\newblock Differential calculus over general base fields and rings.
\newblock {\em Expo. Math.}, 3(22):213--282, 2004.

\bibitem{BoKo}
G.~Bonavolont\`{a} and A.~Kotov.
\newblock On the space of super maps between smooth supermanifolds.
\newblock {arXiv}:math/1304.0394, 2008.

\bibitem{DumFoo}
D.~S. Dummit and Foote~R. M.
\newblock {\em Abstract Algebra}.
\newblock John Wiley \& Sons, Inc., New Jersey, 2004.

\bibitem{GloDL}
H.~Gl{\"o}ckner.
\newblock Direct limits of infinite-dimensional {Lie} groups compared to direct
  limits in related categories.
\newblock {\em J. Funct. Anal.}, 245:19--61, 2007.

\bibitem{GloBO}
H.~Gl{\"o}ckner and K.-H. Neeb.
\newblock Infinite-dimensional {Lie} groups, vol. 1: Basic theory and main
  examples.
\newblock To appear in Springer-Verlag.

\bibitem{GloNeeb17}
H.~Gl\"{o}ckner and K.-H. Neeb.
\newblock Diffeomorphism groups of compact convex sets.
\newblock {\em Indag. Math. (N.S.)}, 28(4):760--783, 2017.

\bibitem{Ham}
R.~S. Hamilton.
\newblock The inverse function theorem of {Nash} and {Moser}.
\newblock {\em Bull. Amer. Math. Soc. (N.S.)}, 7(1):65--222, 1982.

\bibitem{Hanisch}
F.~Hanisch.
\newblock A supermanifold structure on spaces of morphisms between
  supermanifolds.
\newblock {arXiv}:math/1406.7484v1, 2014.

\bibitem{Keller}
H.~H. Keller.
\newblock {\em Differential Calculus in Locally Convex Spaces}.
\newblock Springer-Verlag, Berlin, 1974.

\bibitem{Lang}
S.~Lang.
\newblock {\em Fundamentals of Differential Geometry}.
\newblock Springer-Verlag, Berlin, 1999.

\bibitem{Lei80}
D.~A. Le\v{i}tes.
\newblock Introduction to the theory of supermanifolds.
\newblock {\em Russian Math. Surveys}, 35(1):1--64, 1980.

\bibitem{MolICTP}
V.~Molotkov.
\newblock Infinite-dimensional $\mathbb{Z}^k_2$-supermanifolds.
\newblock {\em ICTP preprints}, IC/84/183, 1984.

\bibitem{Mol2}
V.~Molotkov.
\newblock Sheaves of automorphisms and invariants of {B}anach supermanifolds.
\newblock In {\em Mathematics and mathematical education ({B}ulgarian)}, pages
  271--283. Publ. House Bulgar. Acad. Sci., Sofia, 1986.

\bibitem{Mol}
V.~Molotkov.
\newblock Infinite-dimensional and colored supermanifolds.
\newblock {\em Journal of Nonlinear Mathematical Physics}, 17:375--446, 2010.

\bibitem{NeSa}
K.-H. Neeb and H.~Salmasian.
\newblock Positive definite superfunctions and unitary representations of {Lie}
  supergroups.
\newblock {\em Transform. Groups}, 18(3):803--844, 2013.

\bibitem{SachseDiss}
C.~Sachse.
\newblock {\em {Global Analytic Approach to Super Teichm{\"u}ller Spaces}}.
\newblock PhD thesis, Technische Universit{\"a}t Darmstadt, 2007.

\bibitem{Sachse2}
C.~Sachse.
\newblock A categorial formulation of superalgebra and supergeometry.
\newblock {arXiv}:math/0802.4067v1, 2008.

\bibitem{SaWo}
C.~Sachse and C.~Wockel.
\newblock The diffeomorphism supergroup of a finite-dimensional supermanifold.
\newblock {\em Adv. Theor. Math. Phys.}, 15(2):285--323, 2011.

\bibitem{Schub}
H.~Schubert.
\newblock {\em Categories}.
\newblock Springer-Verlag, Berlin, 1972.

\bibitem{DissSchuett}
J.~Sch{\"u}tt.
\newblock {\em {Infinite-Dimensional Supermanifolds, Lie Supergroups and the
  Supergroup of Superdiffeomorphisms}}.
\newblock PhD thesis, Universit{\"a}t Paderborn, 2018.
\newblock \url{http://dx.doi.org/10.17619/UNIPB/1-568}.

\bibitem{Shv}
A.~S. Shvarts.
\newblock On the definition of superspace.
\newblock {\em Theoret. and Math. Phys.}, 60(1):657--660, 1984.

\bibitem{Vor}
A.~Voronov.
\newblock Maps of supermanifolds.
\newblock {\em Theoret. and Math. Phys.}, 60(1):660--664, 1984.

\end{thebibliography}
\newpage
\renewcommand{\nomname}{List of Symbols}
\printnomenclature
\newpage
\printindex
\end{document}